\documentclass[preprint,12pt]{elsarticle}
\usepackage{graphicx}
\usepackage{dcolumn}
\usepackage{bm}
\usepackage{amsmath} 
\usepackage{graphicx}
\usepackage{caption}
\usepackage{color}
\usepackage{amssymb}
\usepackage{makecell}
\usepackage[T1]{fontenc}
\usepackage{caption}
\usepackage{float}
\usepackage{array}
\usepackage{amsthm}
\usepackage{lineno}
\usepackage{mathrsfs}
\usepackage{amsmath}
\allowdisplaybreaks[4]  
\usepackage{booktabs}
\usepackage{epstopdf}
\usepackage{amsfonts}
\usepackage{bm}
\usepackage{lineno}
\usepackage{enumitem}
\usepackage{booktabs}
\numberwithin{equation}{section}
\usepackage{tabularx}
\usepackage[top=2cm, bottom=4cm, left=3cm, right=3cm] {geometry}
\usepackage{natbib, comment}

\usepackage{hyperref}
\hypersetup{
colorlinks=true, 
linkcolor=blue, 
anchorcolor=blue, 
citecolor=blue 
}

\newtheorem{thm}{Theorem}[section]
\newtheorem{cor}[thm]{Corollary}
\newtheorem{lem}[thm]{Lemma}
\newtheorem{prop}[thm]{Proposition}
\theoremstyle{definition}
\newtheorem{defn}[thm]{Definition}
\theoremstyle{definition}
\newtheorem{rem}[thm]{Remark}
\theoremstyle{definition}
\newtheorem{ex}[thm]{Example}

\journal{Journal of Differential Equations}
\begin{document}

\begin{frontmatter}

\title{Multicycle dynamics and high-codimension bifurcations in SIRS epidemic models with cubic psychological saturated incidence}

\author[label1]{Henan Wang}
\ead{wanghn24@mails.jlu.edu.cn}
\author[label1,label2]{Xu Chen}
\ead{chenxu21@mails.jlu.edu.cn}
\author[label1,label3]{Wenxuan Li}
\ead{Liwx18@mails.jlu.edu.cn}
\author[label1]{Suli Liu\corref{cor1}}
\cortext[cor1]{Corresponding author}
\ead{liusuli@jlu.edu.cn}
\author[label1]{Huilai Li}
\ead{lihuilai@jlu.edu.cn}
\affiliation[label1]{organization={School of Mathematics, Jilin University},
            addressline={Qianjin Street 2699},
            city={Changchun},
            postcode={130012},
            state={Jilin Province},
            country={China}}
\affiliation[label2]{organization={School of Mathematics, Changchun University of Technology},
            addressline={Yan'an Street 2055},
            city={Changchun},
            postcode={130012},
            state={Jilin Province},
            country={China}}
\affiliation[label3]{organization={Department of Mathematics and Physics, Suzhou Polytechnic University},
            addressline={Zhineng Avenue 106},
            city={Suzhou},
            postcode={215004},
            state={Jiangsu Province},
            country={China}}
\begin{abstract}
This study investigates bifurcation dynamics in an SIRS epidemic model with cubic saturated incidence, extending the quadratic saturation framework established by Lu, Huang, Ruan, and Yu (Journal of Differential Equations, 267, 2019). We rigorously prove the existence of codimension-three Bogdanov-Takens bifurcations and degenerate Hopf bifurcations, demonstrating the coexistence of three limit cycles within a single epidemiological model, a phenomenon that is rarely documented and exhibits significant dynamical complexity. Our analysis reveals that both the infection rate $\kappa$ (through specific inequality conditions) and psychological effect thresholds critically govern disease dynamics: from complete eradication to various persistence patterns including multiple periodic oscillations and coexistent steady states. By innovatively applying singularity theory, we characterize the topology of the bifurcation set through the local unfolding of singularities and the identification of nondegenerate singularities for fronts. For the Bogdanov-Takens bifurcation, we establish a fundamental connection between the multiplicity of positive equilibria, $A_k$-type singularities, and the differential topology of the bifurcation set via versal unfolding theory, confirming that the codimension-2 $BT$ bifurcation set is locally diffeomorphic to a regular curve. For the Hopf bifurcation set, we employ front identification techniques and show that the bifurcation surface generically admits two types of singularities: cuspidal edges and swallowtails. Numerical simulations verify the emergence of three limit cycles in monotonic parameter regimes and two limit cycles in nonmonotonic regimes. This work advances existing bifurcation research by incorporating higher-order interactions and comprehensive singularity analysis, thereby providing a mathematical foundation for decoding complex transmission mechanisms critical to the design of public health strategies.
\end{abstract}
\end{frontmatter}
\section{Introduction}
\label{sec:introduction}
Mathematical modeling of infectious diseases has evolved substantially since Kermack and McKendrick's seminal SIR model \cite{1927}, with compartmental frameworks progressing from bilinear incidence ($\kappa SI$) to incorporate biologically realistic mechanisms through nonlinear incidence rates. Key developments include the standard incidence $\frac{\kappa SI}{N}$ \cite{Li1999, Liu2021a, Li2025}, saturated incidence
$\frac{\kappa SI}{1 + \alpha I}$ \cite{V1977, Capasso1978}, and generalized forms $\frac{\kappa SI^n}{1 + \alpha I^m}$ \cite{Liu1986, Liu1987, Hethcote1991, Ruan2003, Tang2008, Saha2022, Cui2024}. Such nonlinear structures generate complex dynamics, including periodic outbreaks and bistability, making bifurcation theory indispensable for identifying critical thresholds where parameter variations induce qualitative behavioral transitions \cite{Kuznetsov2023, Perko1996}.

While \citet{Capasso1978} first conceptualized the integration of psychological effects into incidence rates to capture risk-behavior feedback, they did not specify its mathematical form. Subsequent studies concretized this relationship through nonmonotone incidence rates: \citet{Xiao2007a, Liu2015} formalized $\frac{\kappa  SI}{1+\alpha I^2}$, \citet{Zhou2007} and \citet{Lu2020} advanced $\frac{\kappa  SI}{1+\beta I+\alpha I^2}$ ($\beta > -2\sqrt{\alpha}$). These formulations mathematically encode psychological mechanisms-infection rates rise initially with pathogen exposure, then decay under behavioral mitigation (e.g., risk avoidance, protection measures, intervention policies), asymptotically approaching zero as $I \to +\infty$. This decay assumption, however, contradicts endemic saturation in diseases like influenza, where crowding effects sustain transmission \cite{Lu2019}, revealing limitations in modeling psychology-dominated dynamics.

Building on these foundations, \citet{Lu2019} introduced a psychologically motivated SIRS model with nonmonotone and saturated incidence:
\begin{equation}\label{eq:lu_model}
\begin{aligned}
\frac{{\rm d}S}{{\rm d}t} &= \Lambda - dS - \frac{\kappa SI^2}{1 + \beta I + \gamma I^2} + \delta R, \\
\frac{{\rm d}I}{{\rm d}t} &= \frac{\kappa SI^2}{1 + \beta I + \gamma I^2} - (d + \mu)I, \\
\frac{{\rm d}R}{{\rm d}t} &= \mu I - (d + \delta)R,
\end{aligned}
\end{equation}
where $\Lambda$ denotes the recruitment rate, $d$ is the natural death rate, $\mu$ is the recovery rate, $\delta$ is the immunity loss rate, $\kappa$ is the infection rate, and $\gamma > 0$ quantifies psychological inhibition. The constraint $\beta > -2\sqrt{\gamma}$ ensures $1 + \beta I + \gamma I^2 > 0$ for all $I \geq 0$. Let the incidence rate in model \eqref{eq:lu_model} be denoted by $g(I)S$, where
\begin{equation}\label{g(I):lu_model}
g(I) = \frac{\kappa I^2}{1 + \beta I + \gamma I^2}.
\end{equation}
The function $g(I)$ exhibits two distinct behavioral regimes:
(i) When $\beta > 0$, $g(I)$ is monotonically increasing and asymptotically approaches the saturation level $\kappa/\gamma$ as $I \to +\infty$;
(ii) When $-2\sqrt{\gamma} < \beta < 0$, $g(I)$ is nonmonotonic: increasing for small $I$, decreasing for large $I$, while still converging to $\kappa/\gamma$ as $I \to +\infty$.

Through center manifold theory and normal form reduction, they established that quadratic saturation induces codimension-two Bogdanov-Takens and Hopf bifurcations. This result revealed oscillatory dynamics and bistability, significantly advancing the understanding of complex epidemic behavior.

However, Lu et al.'s \cite{Lu2019} analysis of psychological effects was confined to quadratic saturation, which inherently limits the spectrum of achievable dynamics. Although \citet{Lu2020} and \citet{Li2015} demonstrated codimension-three Bogdanov-Takens bifurcations in the epidemic models, such high-codimension phenomena in SIRS systems are rarely documented, particularly those incorporating saturated psychological feedback. These bifurcations entail co-occurring criticalities (e.g., simultaneous Hopf and homoclinic events) that generate intricate attractor networks, posing significant analytical challenges due to high-dimensional parameter space, yet offering critical insights into epidemic regime transitions.
This gap is particularly notable for SIRS models with cubic incidence, where high-codimension degenerative Hopf bifurcations remain largely unexplored, impeding our understanding of multi-wave outbreak mechanisms.

\begin{center}
\begin{table}[htbp]
\caption{Monotonicity regimes of  $g(I)=\frac{I^3}{1+\beta I +\alpha I^2+\gamma I^3}$.}
\label{tab:monotonicity}
 \centering
  \begin{tabular}{@{}c c c l c@{}}
 \toprule
 \multicolumn{1}{c}{Function} & \multicolumn{1}{c}{Derivative} & Parameters & Behavior & Extrema \\
\multicolumn{1}{c}{$g(I)$} & \multicolumn{1}{c}{$g'(I)$} & ($\alpha, \beta, \gamma$) & ($I > 0$) & ($I>0$)
 \\
\midrule
  $\frac{I^3}{1+\gamma I^3}$ & $\frac{3I^3}{(1+\gamma I^3)^2}$ & $\begin{array}{ll}\alpha=\beta=0, \gamma>0\end{array}$ &$\begin{array}{ll} {\rm Inc.}\end{array}$& $\begin{array}{ll}0\end{array}$   \\
  \\
  $\frac{I^3}{1+\alpha I^2+\gamma I^3}$ & $\frac{I^2(3+\alpha I^2)}{(1+\alpha I^2+\gamma I^3)^2}$ & $\begin{array}{ll}\beta=0,\begin{cases}\alpha>0, \gamma>0 \\ \alpha<0, \gamma>0\end{cases}\end{array}$& $\begin{array}{ll}{\rm Inc.}\\ {\rm Inc.-Dec.}\end{array}$ & $\begin{array}{ll}0\\ 1\end{array}$  \\
  \\
  $\frac{I^3}{1+\beta I +\gamma I^3}$ &  $\frac{I^2(3+2\beta I)}{(1+\beta I +\gamma I^3)^2}$ & $\begin{array}{ll}\alpha=0,\begin{cases}\beta>0, \gamma>0 \\ \beta<0, \gamma>0\end{cases}\end{array}$ & $\begin{array}{ll}{\rm Inc.}\\ {\rm Inc.-Dec.}\end{array}$& $\begin{array}{ll}0\\ 1\end{array}$ \\
  \\
  $\frac{I^3}{1+\beta I +\alpha I^2+\gamma I^3}$ &  $\frac{I^2(3+2\beta I+\alpha I^2)}{(1+\beta I +\alpha I^2+\gamma I^3)^2}$  & $\begin{array}{ll}\alpha>0,\begin{cases}\beta>0, \gamma>0 \\ \beta<0, \gamma>0\end{cases}\\
 \alpha<0, \begin{cases}\beta>0, \gamma>0 \\ \beta<0, \gamma>0\end{cases}\end{array}$  & $\begin{array}{ll}{\rm Inc.}\\ {\rm Inc.-Dec.-Inc.}\vspace{3mm}\\ {\rm Inc.-Dec.}\\{\rm Inc.-Dec.}\end{array}$& $\begin{array}{ll}0\\ 2\vspace{3mm}\\1\\1\end{array}$  \\
\bottomrule
\end{tabular}
\parbox{\linewidth}{\footnotesize Note: Inc. denotes increase, and Dec. denotes decrease.}
\end{table}
\end{center}

To address these limitations, we extend Lu et al.'s \cite{Lu2019} framework through a cubic-saturated incidence function $g(I)S$, where $g(I)$ is defined by:
\begin{equation}\label{eq:modified_incidence}
g(I) = \dfrac{\kappa I^3}{1+\beta I+\gamma I^3},
\end{equation}
subject to $-3\sqrt[3]{\frac{1}{4}\gamma} < \beta $, guaranteeing denominator positivity for $I \geq 0$ which is a necessary condition for biological meaning. This formulation achieves structural minimalism while preserving dynamical regimes, as evidenced in Table \ref{tab:monotonicity}. By omitting quadratic terms ($\alpha I^2$) in the denominator, we simplify the incidence structure while retaining essential monotonic/nonmonotonic regimes: monotonic saturation ($\beta > 0$) and single-peak nonmonotonicity ($\beta < 0$) (Table \ref{tab:monotonicity}, row 3), avoiding artificial multi-extrema patterns from higher-order terms.
While asymptotically similar to quadratic saturation \eqref{g(I):lu_model}, the cubic term captures the mechanistic progression of an epidemic through its distinct growth-decay phases.
Figure \ref{fig:2}(a) demonstrates how $\beta < 0$ encodes psychological biphasic responses (risk $\rightarrow$ panic $\rightarrow$ behavioral adjustment), while Figure \ref{fig:2}(b) shows endemic saturation ($\kappa/\gamma$) consistent with crowding persistence in avian influenza.
The cubic term critically enables accelerated initial growth, capturing exponential transmission, and earlier/sharper-inflection dynamics that model abrupt interventions (e.g., COVID-19 lockdowns \cite{Harris2023, Jang2024}).

\begin{figure}[htbp]
  \centering
  \begin{minipage}[c]{0.38\textwidth}
    \centering
    \includegraphics[width=\textwidth]{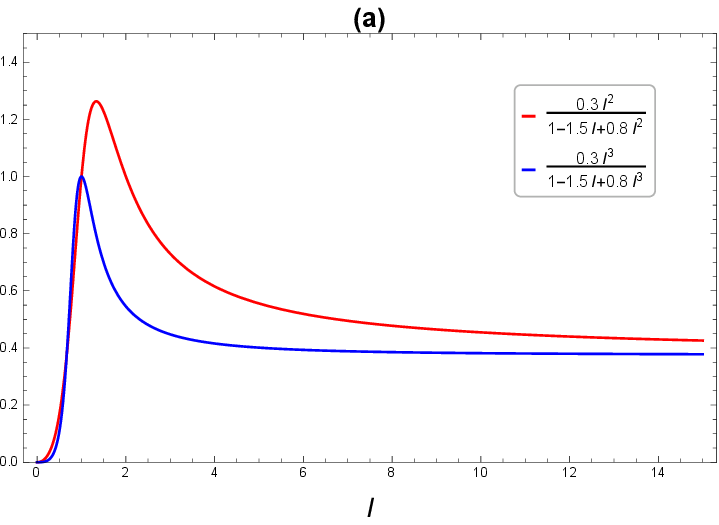}
  \end{minipage}
  \hspace{.25in}
  \begin{minipage}[c]{0.38\textwidth}
    \centering
    \includegraphics[width=\textwidth]{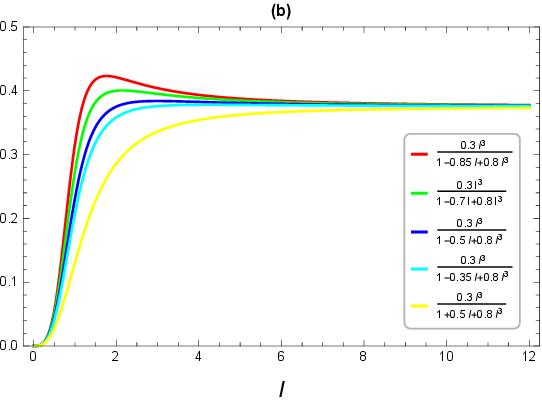}
  \end{minipage}
  \caption{(a) Comparison of modified cubic function $g(I)$ (blue) with Lu et al.'s quadratic form (red); (b) Variation of cubic function $g(I)$ with different $\beta$ values.}
  \label{fig:2}
\end{figure}

We employ singularity theory, founded by Thom \cite{Rosen1977} and advanced by Arnold \cite{Arnold2012}, Wall \cite{Wall1981}, Golubitsky \cite{Golubitsky1981}, Bruce \cite{Bruce1992}, and Izumiya \cite{Izumiya2014}, to characterize bifurcation structures induced by cubic nonlinearity. Central to our approach is versal unfolding \cite{Golubitsky1981, Bruce1992, Izumiya2014}, which provides minimal parametrized families that capture all local behaviors near singularities. Building on seminal contributions by  \citet{Tang2016, Tang2019}, and \citet{Liu2021} in planar systems, particularly their work on nilpotent singularities and homogeneous cubic centers, we establish three key advances over prior epidemiological bifurcation studies \cite{Ruan2003, Lu2019, Hu2011, Alexander2005, She2024, Xiao2007}.
First, we develop a unified framework preserving biological realism while extending Lu et al.'s model to both monotonic ($\beta>0$) and nonmonotonic ($\beta<0$) regimes.
Second, we achieve topological classification of bifurcation sets through singularity-theoretic characterization of saddle-node surfaces, revealing how codimension-three singularities organize global dynamics. Third, we provide the first rigorous demonstration of cubic-induced degenerate Hopf bifurcations that produce multiple limit cycles and codimension-three Bogdanov-Takens points in epidemic models. Numerical verification confirms these cubic nonlinearities induce previously unobtainable dynamics, including triple-cycle coexistence.

The remainder of this paper is organized as follows. Section 2 introduces the modified SIRS model with a cubic incidence rate and its non-dimensional form, establishing the foundational singularity theory framework for bifurcation analysis. Section 3 provides a comprehensive stability analysis of all equilibria. Key results include the classification of the saddle-node singularity (Theorem \ref{th:3.4}) and the stability criteria for the focus equilibrium in the bistable regime (Theorem \ref{th:3.6}), which provide the theoretical basis for analyzing Bogdanov-Takens and degenerate Hopf bifurcations in subsequent sections.  Section 4 is devoted to a detailed study of both codimension-2 and codimension-3 Bogdanov-Takens bifurcations. Section 5 analyzes degenerate Hopf bifurcations and the emergence of multiple limit cycles. Finally, Section 6 summarizes the key results, discusses their epidemiological implications, and suggests potential directions for future research.

\section{Singularity Theory Framework for Bifurcation Analysis}
\label{2}
This section constructs the mathematical foundation for our study. We first reduce the dimension of the epidemic model and introduce dimensionless parameters. Subsequently, we reformulate the problem of finding equilibria and their degeneracies within the framework of singularity theory. This approach allows us to derive a universal unfolding that captures the local bifurcation structure of the system in a minimal parametrization.

\subsection{Model Reduction and Non-Dimensionalization}
Based on our preceding monotonicity analysis, we focus on the SIRS model with the cubic-saturated incidence rate:
\begin{equation}\label{eq:2.1}
\begin{aligned}
\frac{{\rm d}S}{{\rm d} t} &= \Lambda - dS - \frac{\kappa SI^3}{1 + \beta I + \gamma I^3} + \delta R, \\
\frac{{\rm d}I}{{\rm d}t} &= \frac{\kappa SI^3}{1 + \beta I + \gamma I^3} - (d + \mu)I, \\
\frac{{\rm d}R}{{\rm d}t} &= \mu I - (d + \delta)R,
\end{aligned}
\end{equation}
where $S(t)$, $I(t)$, and $R(t)$ denote susceptible, infected, and recovered individuals at time $t$ respectively, with non-negative initial conditions $S(0), I(0), R(0) \geq 0$. The parameters $\Lambda, d, \mu, \delta, \gamma, \kappa$ are strictly positive, while $\beta$ satisfies $-3\sqrt[3]{\gamma/4} < \beta $.

The critical parameter constraint on $\beta$ emerges from the requirement that the incidence denominator $f(I) = 1 + \beta I + \gamma I^3$ remains positive for all $I \geq 0$. When $\beta\geq0$, it is obvious that $f(I)>0$ for all $I\geq0$. When $\beta < 0$, the derivative $f'(I) = \beta + 3\gamma I^2$ vanishes at $I_* = \sqrt{-\beta/(3\gamma)}$, which corresponds to an extremum. The second derivative test $f''(I_*) = 6\gamma I_* > 0$ confirms this critical point as a local minimum. Substituting $I_*$ into $f(I)$ yields the minimal value:
\[
f_{\text{min}} = 1 + \beta I_* + \gamma I_*^3 = 1 + \frac{2\beta}{3}\sqrt{-\frac{\beta}{3\gamma}}.
\]
Enforcing positivity $f_{\text{min}} > 0$ leads to the inequality:
\[
1 + \frac{2\beta}{3}\sqrt{-\frac{\beta}{3\gamma}} > 0 \implies \beta > -3\sqrt[3]{\frac{\gamma}{4}}.
\]

Let the total population represented by $N(t)= S(t) + I(t) + R(t)$, then \(\dfrac{{\rm d}N}{ {\rm d} t} = \Lambda - dN\). We introduce the invariant manifold below without proof; for justification, please refer to \cite{Liu2021a, Lu2019}.

\begin{lem}[Invariant Manifold]\label{lem:2.1}
The system \eqref{eq:2.1} possesses an invariant manifold:
\[
\Sigma = \left\{ (S,I,R) \in \mathbb{R}^3_+ \ \big|\ S + I + R = \frac{\Lambda}{d} \right\},
\]
where \(\mathbb{R}^3_+ = \{(S, I, R) : S \geq 0, I \geq 0, R \geq 0\}\).
\end{lem}

This limit set enables a dimensional reduction by expressing susceptible individuals as \(S(t) = \frac{\Lambda}{d} - I(t) - R(t)\), thus we consider the simplified planar system governing infection-recovery dynamics:
\begin{equation}\label{eq:2.2}
\begin{aligned}
\frac{{\rm d}I}{{\rm d}t} &= \frac{\kappa I^3}{1 + \beta I + \gamma I^3}\left(\frac{\Lambda}{d} - I - R\right) - (d + \mu)I, \\
\frac{{\rm d}R}{{\rm d}t} &= \mu I - (d + \delta)R.
\end{aligned}
\end{equation}
Then, the positively invariant set of system \eqref{eq:2.2} is
\[\Sigma_{red} = \{(I,R) \in \mathbb{R}^2_+ \ |\ I + R \leq \Lambda/d\}.\]

We further simplify the system \eqref{eq:2.2} through nondimensionalization to minimize parameter redundancy and reveal fundamental dynamic relationships. Introducing characteristic scales for population and time, we define the dimensionless variables:
\begin{equation}\label{eq:2.3}
\begin{aligned}
I=\sqrt{\dfrac{d\mu}{\kappa\Lambda}}x, \quad  R=\sqrt{\dfrac{d\mu}{\kappa\Lambda}}y, \quad \tau = \mu t.
\end{aligned}
\end{equation}
Substituting \eqref{eq:2.3} into \eqref{eq:2.2} yields the nondimensional system (retaining symbol $t$):
\begin{equation}\label{eq:2.4}
\begin{aligned}
\frac{{\rm d}x}{{\rm d}t} &= \frac{x^3}{1 + a x + b x^3}(1 - c x - c y) - m x := f(x,y;\lambda), \\
\frac{{\rm d}y}{{\rm d}t} &= x - n y := g(x,y;\lambda),
\end{aligned}
\end{equation}
where the composite parameters combine original quantities through:
\begin{equation}\label{eq:2.5}
\begin{aligned}
a=\beta\sqrt{\dfrac{d\mu}{\kappa\Lambda}},\quad b=\dfrac{d\mu\gamma}{\kappa\Lambda}\sqrt{\dfrac{d\mu}{\kappa\Lambda}}, \quad c=\dfrac{d}{\Lambda}\sqrt{\dfrac{d\mu}{\kappa\Lambda}},\quad m=\dfrac{d+\mu}{\mu},\quad n=\dfrac{d+\delta}{\mu}.
\end{aligned}
\end{equation}
The original constraint $-3\sqrt[3]{\gamma/4} < \beta$ translates directly to $-3\sqrt[3]{b/4} < a$ through the established parameter mappings, preserving the cubic-root relationship in dimensionless form.
According to Lemma \ref{lem:2.1} and transformation \eqref{eq:2.5}, we conclude  all orbits of system \eqref{eq:2.4} eventually approach, enter, and stay in the following compact set
\[
\widetilde{\Sigma} = \left\{ (x,y) \in \mathbb{R}^2_+ \ \Big|\ x + y \leq \frac{\Lambda}{d}\big/\sqrt{\dfrac{d\mu}{\kappa\Lambda}} =\frac{1}{c}\right\}.
\]

\subsection{$\mathcal{R}^+$-Versal Unfolding of the Reduced Function and Equilibria}

Singularity theory provides a systematic methodology for addressing bifurcation problems in dynamical systems. Crucially, the equilibria of dynamical systems correspond precisely to singularities of smooth mappings, which are the primary objects studied in singularity theory. We therefore employ this framework to determine how parameters $(a,b,c,m,n) \in \mathbb{R}^5$ govern the number of equilibria in system \eqref{eq:2.4}, then characterize the local differential topology at these equilibria.

The 5-parameter bifurcation problem corresponding to the system \eqref{eq:2.4} is the smooth map
\(\Phi:\mathbb{R}^2\times \mathbb{R}^5\rightarrow \mathbb{R}^2\)
 given by $\Phi(x,y;\lambda)=(f(x,y;\lambda),g(x,y;\lambda))$. For each $\lambda=(a,b,c,m,n)\in \mathbb{R}^5$, we denote the map $\phi_\lambda:\mathbb{R}^2\rightarrow \mathbb{R}^2$ by
\(\phi_\lambda:=\Phi(x,y;\lambda)\).
The zero set of $\Phi$ is
\begin{align*}
N_\Phi:=\{(x,y,\lambda)\in\mathbb{R}^2\times \mathbb{R}^5\big|\, \phi_\lambda=\Phi(x,y;\lambda)=0\}.
\end{align*}

\begin{defn}[\cite{Bruce1992, Izumiya2014}]  \label{defn:2.2}
The {\it singular set} of the bifurcation problem $\Phi$ is
\begin{align*}
\Sigma_\Phi:=\{(x,y,\lambda)\in N_\Phi\big|\det({\rm d}\phi_\lambda(x,y))=0\}.
\end{align*}
The {\it bifurcation set} or {\it discriminant} is the subset of the parameter space $\mathbb{R}^5$,
\begin{align*}
\Delta_\Phi:=\pi_\Phi(\Sigma_\Phi).
\end{align*}
determined by the projection $\pi_\Phi:N_\Phi\rightarrow \mathbb{R}^5,\pi_\Phi(x,y,\lambda)=\lambda$. That is,
\begin{align*}
\Delta_\Phi:=\{\lambda\in \mathbb{R}^5\big|\Phi_\lambda(x,y)=\det({\rm d}\phi_\lambda(x,y))=0\}.
\end{align*}
\end{defn}

The zero set $N_\Phi$ satisfies \(f(x,y;\lambda)=g(x,y;\lambda)=0\), which is equivalent to
\begin{equation}\label{eq:2.6}
\begin{aligned}
&x^3(1-cx-cy)-mx(1+a x+b x^3)=0,\\
&x- ny=0.
\end{aligned}\end{equation}
Substituting $y=\frac{1}{n}x$ into the first equation of \eqref{eq:2.6}, one obtains the equilibrium condition
\begin{equation}\label{eq:2.7}
x\Big((c(n+1)+bmn)x^3-nx^2+amnx+nm\Big)=0.
\end{equation}
In combination with equations \eqref{eq:2.6}, it is easy to calculate
\begin{equation}\label{eq:2.8}
\begin{aligned}
\det({\rm d}\phi_\lambda(x,y))&=\det\left(
  \begin{array}{cc}
                  \dfrac{-2 a m x-4 (b m+c) x^3-3 c x^2 y-m+3 x^2}{1+a x+b x^3} &\dfrac{-cx^3}{1+a x+b x^3}\\
                  1& -n
     \end{array}
\right)\\
 &=\dfrac{c x^3-n(-2 a m x-4 x^3 (b m+c)-3 c x^2 y-m+3 x^2)}{1+a x+b x^3}.
     \end{aligned}\end{equation}
Substituting $y=\dfrac{1}{n}x$ into the expression \eqref{eq:2.8}, then $\det({\rm d}\phi_\lambda(x,y))=0$ is equivalent to
\begin{align}\label{eq:2.9}
4(c(n+1)+bmn)x^3-3nx^2+2amnx+mn=0.
\end{align}
Since $c(n+1)+bmn>0$, equations \eqref{eq:2.7} and \eqref{eq:2.9} are equivalent to
\begin{equation}\label{eq:2.10}
x\Big(x^3-\frac{n}{c(n+1)+bmn}x^2+\frac{amn}{c(n+1)+bmn}x+\frac{mn}{c(n+1)+bmn}\Big)=0
\end{equation}
and
\begin{equation}\label{eq:2.11}
4x^3-\frac{3n}{(c(n+1)+bmn)}x^2+\frac{2amn}{(c(n+1)+bmn)}x+\frac{mn}{(c(n+1)+bmn)}=0.
\end{equation}
Letting
 \begin{align}\label{eq:2.12}
\theta:=-\frac{n}{c(n+1)+bmn},\quad p_1:=\frac{amn}{c(n+1)+bmn},\quad q_1:=\frac{mn}{c(n+1)+bmn}.
\end{align}
For equations \eqref{eq:2.7} and \eqref{eq:2.9}, making the translation
\begin{align}\label{eq:2.13}
x=z -\dfrac{\theta}{3},
\end{align}
then they respectively convert into
\begin{align}\label{eq:2.14}
\left(z-\frac{\theta}{3}\right) \left(z^3+z \left(p_1-\frac{\theta^2}{3}\right)+\frac{2\theta^3}{27}-\frac{\theta  p_1}{3}+q_1\right)=0
\end{align}
and
\begin{align}\label{eq:2.15}
\left(z^3+z \left(p_1-\frac{\theta^2}{3}\right)+\frac{2\theta^3}{27}-\frac{\theta p_1}{3}+q\right)+\left(z-\frac{\theta}{3}\right) \left(3 z^2+p_1-\frac{\theta ^2}{3}\right)=0.
\end{align}
It is clear that equations \eqref{eq:2.7} and \eqref{eq:2.9} both admit the trivial solution
$x=0$, which corresponds to the disease-free equilibrium. Under the translation $x=z-\frac{\theta}{3}$, this solution becomes $z=\frac{\theta}{3}$. For
$z\neq\frac{\theta}{3}$, the two distinct positive endemic equilibria of system \eqref{eq:2.4} satisfy the following equations
\begin{align}\label{eq:2.16}
 z^3+z \left(p_1-\frac{\theta^2}{3}\right)+\frac{2\theta^3}{27}-\frac{\theta p_1}{3}+q_1=0
\end{align}
and
\begin{align}\label{eq:2.17}
 3 z^2+p_1-\frac{\theta^2}{3}=0.
\end{align}
Letting
\begin{align}\label{eq:2.18}
 p:=p_1-\frac{\theta^2}{3},\quad q:=\frac{2\theta^3}{27}-\frac{\theta p_1}{3}+q_1.
\end{align}
Equations \eqref{eq:2.16} and \eqref{eq:2.17} are transformed into, respectively,
\begin{align}\label{eq:2.19}
 z^3+pz+q=0
\end{align}
and
\begin{align}\label{eq:2.20}
 3 z^2+p=0.
\end{align}

Through a continuous  mapping $\varphi:U\subset \mathbb{R}^5\rightarrow \mathbb{R}^2$ given by
\begin{align}
\quad \varphi(a,b,c,m,n)=(p,q),
\end{align}
the five control variables of system \eqref{eq:2.4} are reduced to two control variables.
Showing the scenes of equilibria and their bifurcation of system \eqref{eq:2.4} may be transformed into the analyses for the 2-parameter {\it unfolding} of a potential  function defined by $H: I\times \mathbb{R}^2\rightarrow\mathbb{R}$ given by
\begin{equation}\label{eq:2.22}
\begin{aligned}
H(z,\bm v)=\frac{1}{4}z^4+\frac{1}{2}pz^2+q,
\end{aligned}
\end{equation}
where
\begin{align}\label{eq:2.23}
\bm v=(p,q) \in\mathbb{R}^2.
\end{align}
 Next, we will review some basic notions about the unfolding of functions in singularity theory (refer to the details in \cite{Bruce1992, Izumiya2014}).

\begin{defn}[\cite{Bruce1992, Izumiya2014}] \label{def:2.3}
Let $G :( \mathbb{R}\times\mathbb{R}^r,(s_0,x_0))\rightarrow\mathbb{R}$ be a function germ. We call $G$ an {\it r-parameter unfolding of $\widetilde{g}$}, when $\widetilde{g}(s)=G_{x_0}(s,\bm{x}_0)$. We say that $\widetilde{g}(s)$ has $A_k $-singularity $(k\geq1)$ at $s_0$, if $\widetilde{g}^{(p)}(s_0)=0 $ for all $1\leq p\leq k$ and $f^{(k+1)}(s_0)\neq 0$. We also say that $\widetilde{g}(s)$ has $A_{\geq k} $-singularity at $s_0$ for all $1\leq p\leq k$.
\end{defn}

Let $G$ be an unfolding of $\widetilde{g}$ and $\widetilde{g}(s)$ has an $A_k $-singularity $(k\geq1)$ at $s_0$. We write the $(k-1) $-jet of the partial derivative $\frac{\partial G}{\partial x_i}$ at $t_0$ by $j^{(k-1)}\left(\frac{\partial G}{\partial x_i}(s,\bm x)\right)(s_0)=\sum_{j=1}^{k-1}\alpha_{ji}(s-s_0)^j$ for $i=1, \ldots ,r$. Then  $G$ is called an {\it $\mathcal{R}^+$-versal unfolding} if the $(k-1)\times r$ matrix of coefficients $(\alpha_{ji})$ has rank $k-1$ $(k-1\leq r)$ (refer to the details in \cite{Bruce1992}).

We now introduce the following set concerning the unfolding:
\begin{align*}
\Delta_G=\bigg\{ \bm{x}\in \mathbb{R}^r\mid \frac{\partial  G}{\partial s}(s, \bm{x})=\frac{\partial^2  G}{\partial s^2}(s, \bm{x})=0\bigg\},
\end{align*}
which is called the {\it bifurcation set of  $G$}.
We will review and further explain the following lemma about $\Delta_G$.

\begin{lem}[\cite{Bruce1992,Izumiya2014}]\label{lem:2.4} Let $G:(\mathbb{R}\times
\mathbb{R}^r,(s_0,\bm{x}_0))\rightarrow \mathbb{R} $ be an
$r$-parameter unfolding of $f(s)$, which has $A_ k$ singularity $(k\ge 1)$ at
$s_0$. Suppose that $G$ is an $\mathcal{R}^+$-versal unfolding.
\begin{enumerate}[label=\textnormal{(\roman*)}]
\item
If $k=2$, then $\Delta_G$ is locally diffeomorphic to $\{0\}\times \mathbb{R}^{r-1}$;
\item
If  $k=3$, then $\Delta_G$ is locally diffeomorphic to $C(2,3)\times \mathbb{R}^{r-2}$;
\item
If $k=4$, then $\Delta_G$ is locally diffeomorphic to $SW\times \mathbb{R}^{r-3}$.
\end{enumerate}
\end{lem}

We remark that all of the diffeomorphisms in the above assertions are diffeomorphism germs. We call $C(2, 3) = \{(x_1, x_2) | x_1 = u^2, x_2 = u^3\}$ a {\it $(2, 3)$-cusp} (Figure \ref{fig:unfolding}(a)), $C(2,3)\times\mathbb{R}$ a {\it cuspidal edge} (Figure~\ref{fig:unfolding}(b)), and $SW = \{(x_1, x_2, x_3)|x_1 = 3u^4 +u^2v, x_2 = 4u^3 +2uv, x_3 = v \}$ a {\it swallowtail} (Figure~\ref{fig:unfolding}(c)), respectively.

\begin{figure}[htbp]
  \centering
  \begin{minipage}[c]{0.15\textwidth}
    \centering
    \includegraphics[width=\textwidth]{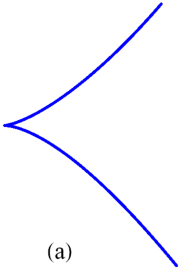}\\
  \end{minipage}\qquad\qquad
  \begin{minipage}[c]{0.25\textwidth}
    \centering
    \includegraphics[width=\textwidth]{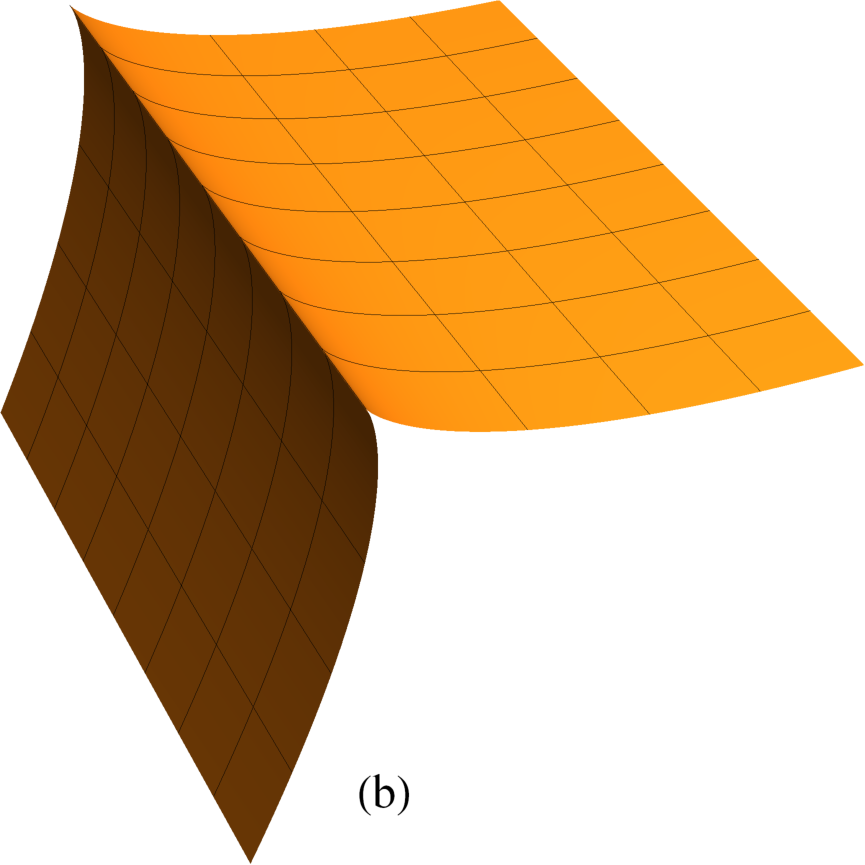} \\
  \end{minipage}
\begin{minipage}[c]{0.35\textwidth}
    \centering
    \includegraphics[width=\textwidth]{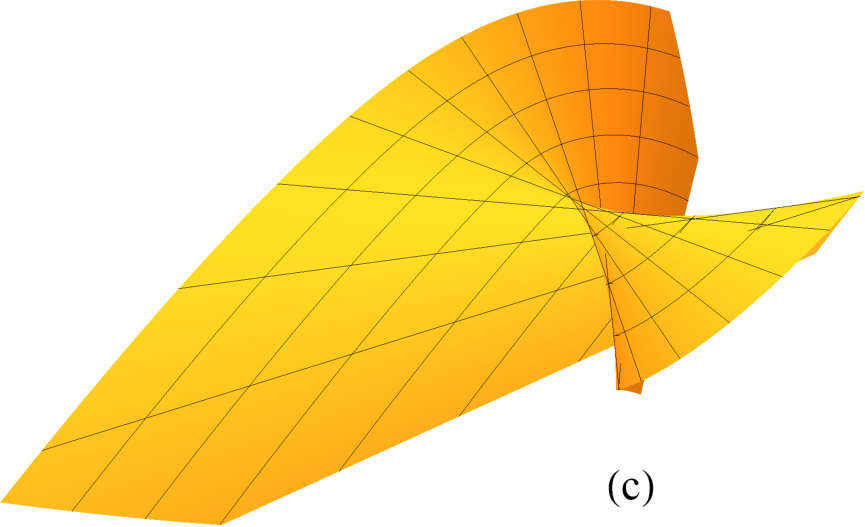} \\
  \end{minipage}
 \caption{ (a) $(2, 3)$-cusp; (b) Cuspidal edge $C(2, 3)\times \mathbb{R}$; (c) Swallowtail $SW$.}
\label{fig:unfolding}
\end{figure}

For any fixed vector $\bm v \in \mathbb{R}^3$, let $h_{\bm{v}}(s) = H(z, \bm{v})$. We state the following lemma.

\begin{lem} \label{lem:2.5} For th potential function $h_{\bm{v}}(z)$, the following facts can be stated:
\begin{enumerate}[label=\textnormal{(\arabic*)}]
\item
$h_{\bm{v}}'(z)=0$ if and only if there exists a real number $z$ of system \eqref{eq:2.4} such that
\begin{align*}
z^3+pz+q=0.
\end{align*}
\item
$h_{\bm{v}}'(z)=h''_{\bm{v}}(z)=0$ if and only if there exists a real number $z$ of system \eqref{eq:2.4} such that
\begin{align*}
\bm{v}=\Big(-3z^2, 2z^3\Big),
\end{align*}
that is, $\bm v$ is a parameterized curve with the parameters $z$ in $\mathbb{R}^2$.
\item
$h_{\bm{v}}'(z)=h''_{\bm{v}}(z)=h'''_{\bm{v}}(z)=0$ if and only if
$\bm{v}=\Big(0, 0\Big)$,
that is,  $z=0$.
\item
$h^{(4)}_{\bm{v}}(z)\neq0$ for any $z\in I$.
\end{enumerate}
\end{lem}

\begin{proof}
{\rm (1)} The assertion is evident. \par
{\rm (2)} Under the condition \( h_{\bm{v}}'(z) = 0 \), we calculate
\[
h''_{\bm{v}}(z) = 3z^2 + p.
\]
Setting \( h''_{\bm{v}}(z) = 0 \) gives us \( p = -3z^2 \). Substituting this back into \( z^3 + pz + q = 0 \) leads to \( q = 2z^3 \). Thus, the assertion in {\rm (2)} follows. \par
{\rm (3)} Assume that both \( h_{\bm{v}}'(z) = 0 \) and \( h''_{\bm{v}}(z) = 0 \) hold. We then find
\[
h'''_{\bm{v}}(z) = 6z.
\]
The equation \( h'''_{\bm{v}}(z) = 0 \) is satisfied if and only if \( z = 0 \). Therefore, the conditions \( h_{\bm{v}}'(z) = h''_{\bm{v}}(z) = h'''_{\bm{v}}(z) = 0 \) hold if and only if \( \bm{v} = (0, 0) \). \par
{\rm (4)} This conclusion is clear from the fact that \( h^{(4)}_{\bm{v}}(z) = 6 \).
\end{proof}

From the definition of $A_k$-singularity and the proof of Lemma \ref{lem:2.5}, we can immediately draw the following conclusions.

\begin{cor}\label{cor:2.6}
\begin{enumerate}[label=\textnormal{(\arabic*)}]
\item
The function $h_{\bm{v}}(z)$ at $z_0$ has $A_2$-singularity if and only if there exists an equilibrium $\bigg(z_0-\dfrac{\theta}{3},\dfrac{1}{n}\Big(z_0-\dfrac{\theta}{3}\Big)\bigg)$ of system \eqref{eq:2.4} such that
\(  \bm{v}=\Big(-3z_0^2, 2z_0^3\Big), \)
and it is required that $z_0\neq 0$.
\item
The function $h_{\bm{v}}(z)$ at $z_0$ has $A_3$-singularity
 if and only if there exists an equilibrium $\bigg(z_0-\dfrac{\theta}{3}, \dfrac{1}{n}\Big(z_0-\dfrac{\theta}{3}\Big)\bigg)$ of system \eqref{eq:2.4} such that
$\bm{v}=\Big(0,0\Big)$, which implies that $z_0=0$.
\end{enumerate}
\end{cor}

\begin{rem}\label{rem:2.7} From the perspective of linear algebra, the $A_k$-singularity of a polynomial function accurately reflects the multiplicity of its real roots. Specifically, the polynomial function \( h_{\bm v}(z) \) at the point \( z_0 \) exhibits \( A_k \)-singularity (for \( k = 1, 2, 3, \ldots \)) if and only if \( h_{\bm v}'(z) \) has real roots of multiplicity \( k \).
\end{rem}

\begin{lem}\label{lem:2.8}
If \(h_{\bm v}\) has an \(A_k\)-singularity (where \(k =2,3\)) at \(z_0\), then \(H\) serves as a $\mathcal{R}^+$-versal unfolding of \(h_{\bm v}\).
\end{lem}
\begin{proof} Noticing the expression
\begin{equation}\label{eq:2.24}
H(z,\bm v)=z^3+pz+q,
\end{equation}
one calculates
\begin{align*}
&\frac{\partial H}{\partial p}(z_0,\bm v)=z_0,\quad \frac{\partial H}{\partial q}(z_0,\bm v)=1.
\end{align*}
Therefore, the $1$-jets of $\frac{\partial H}{\partial p}$ and $\frac{\partial H}{\partial q}$ are
\begin{align*}
\begin{split}
j^1\frac{\partial H}{\partial p}(z_0,\bm v)=z-z_0\ \text{and}\ j^1\frac{\partial H}{\partial q}(z_0,\bm v)=0.
\end{split}
\end{align*}
When $h_{\bm{v}}$ has $A_k$-singularity ($k=2,3$) at $z_0$, one considers the following matrix:
\begin{align*}A=\begin{pmatrix}
      z_0  & 1 \\
       1 & 0
\end{pmatrix}.
\end{align*}
It follows from $\det(A)=-1$ that $H$ is a versal unfolding of $h_{\bm v}$.
 This completes the proof.
 \end{proof}

Neglecting the biological constraints for the parameters \( p \) and \( q \), which depend on \(\lambda = (a, b, c, m, n)\), we will mathematically state the differential structures of the bifurcation set.

\begin{thm}\label{th:2.9}
\begin{enumerate}[label=\textnormal{(\arabic*)}]
\item
The image of the bifurcation set of function $H(z,\bm v)$ is a planer curve, the discriminant of function $H(z,\bm v)$ is
\begin{align*}
\Delta_H:=\Big\{(p,q)\in \mathbb{R}^2\Big|\Big(\dfrac{p}{3}\Big)^3+\Big(\dfrac{q}{2}\Big)^2=0\Big\}.
\end{align*}
\item
Suppose that $z_0$ is a real root of $h_{\bm{v}}(z)$ with multiplicity 2, then the image $\big(-3z^2,2z^3\Big)$ of the bifurcation set is locally diffeomorphic to a line at $\big(-3z_0^2,2z_0^3\big)$ if $z_0\neq0$.
\item
Suppose that $z_0$ is a real root of $h_{\bm{v}}(z)$ with multiplicity 3, then the image $\big(-3z^2, 2z^3\big)$  of  the bifurcation set is locally diffeomorphic to a $(2,3)$-cusp   at $\big(-3z_0^2, 2z_0^3\big)$ if $z_0=0$.
\end{enumerate}
\end{thm}
\begin{proof}
{\rm(1)} According to the definition of bifurcation set and the proof of Lemma \ref{lem:2.5}, the bifurcation set is given by
\begin{align*}
\Delta_H:=\Big\{(p,q)\in \mathbb{R}^2\Big|p=-3z^2, q=2z^3, z\in I\Big\}.
\end{align*}
It is a planar curve with parameter $z\in I$. Eliminating the parameter $z$, we have
 \begin{align*}
\Delta_H:=\Big\{(p,q)\in \mathbb{R}^2\Big|\Big(\dfrac{p}{3}\Big)^3+\Big(\dfrac{q}{2}\Big)^2=0\Big\}.
\end{align*}
The assertion (1) is concluded   immediately from the definition of the bifurcation set  and Corollary \ref{cor:2.6}.\\
\par
{\rm(2)}
Corollary \ref{cor:2.6} and Remark \ref{rem:2.7} indicate  that  $h_{\bm{v}}(s)$ has $A_{2}$-singularity at $z_0$ if and only if
\begin{align*}
\bm{v}=\big(-3z_0^2,2z_0^3\big),
\end{align*}
where $z_0\neq0$, and $z_0$ is a real root of $h_{\bm{v}}'(z)$ with multiplicity 2, then using Lemma \ref{lem:2.8} shows that if $h_{\bm{v}}(z)$ has $A_{2}$-singularity at $z_0$, then  $H(z,\bm{v})$ is an $\mathcal{R}^+$-versal unfolding of $h_{\bm{v}}(z)$, then using Lemma \ref{lem:2.4}, the assertion (2) follows.
\vspace{1em}
\par
{\rm (3)} Again reviewing the conclusion of Corollary \ref{cor:2.6} and Remark \ref{rem:2.7} that $h_{\bm{v}}(s)$ has $A_{3}$-singularity at $z_0$ if and only if $\bm{v}=\big(0,0\big)$ and $z_0=0$ is a real root of $h_{\bm{v}}'(z)$ with multiplicity 3, hence employing Lemma \ref{lem:2.8} and Lemma \ref{lem:2.4}, the assertion holds.

\end{proof}
\begin{figure}[h]
  \centering
  \begin{minipage}[c]{0.28\textwidth}
    \centering
    \includegraphics[width=\textwidth]{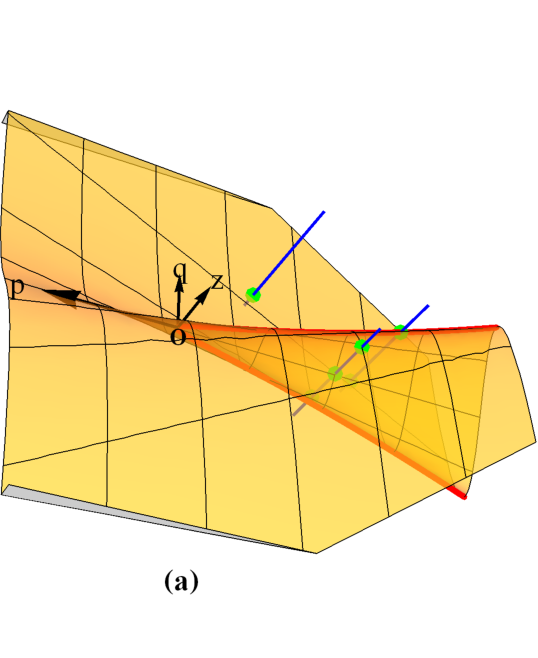}\\
  \end{minipage}
  \quad\quad\quad\quad
  \begin{minipage}[c]{0.275\textwidth}
    \centering
    \includegraphics[width=\textwidth]{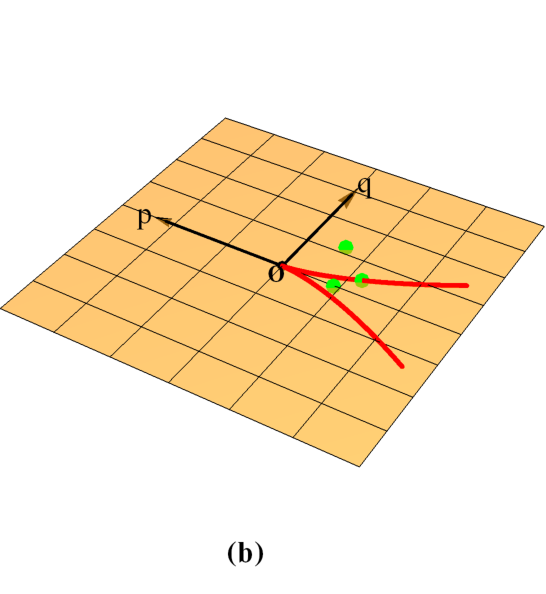}\\
  \end{minipage}
  \caption{(a) The  bifurcation diagram  $q(p,z)=-z^3-pz$; (b) The bifurcation set in $pq$-plane.}
  \label{fig:bif}
\end{figure}
\begin{rem} \label{rem:2.10}
Figure \ref{fig:bif} depicts the bifurcation set of the potential function $H$: (a) the surface $q = -z^3 - pz$ in $(p,q,z)$-space, and (b) its projection onto the $pq$-plane, which is the cuspidal curve $\Delta_H$.
For a fixed parameter pair $(p, q)$ (a point in the $pq$-plane), the number of real roots $z$ of $h_{\boldsymbol{v}}'(z) = z^3 + pz + q = 0$ corresponds to the number of intersections of a vertical line (parallel to the $z$-axis) through $(p, q)$ with the surface in (a). This number is one, two, or three depending on whether $(p, q)$ lies outside, on, or inside the cusp $\Delta_H$, respectively.

The number of real roots translates into the number of equilibria for the mathematical system defined by the zeros of $\Phi$. However, a full analysis of the signs of these roots (which determines the number of \emph{positive} (biologically meaningful) equilibria in system \eqref{eq:2.4}) requires further examination of the specific structure of our model, as will be rigorously established in the next theorem.

In particular, the cusp point at the origin $(p, q) = (0, 0)$, which corresponds to a triple root at $z=0$, represents a degenerate case that is precluded by the biological constraints of our model.
\end{rem}
Denote $\mathfrak{D}_{(p,q)} := \left(\frac{p}{3}\right)^3 + \left(\frac{q}{2}\right)^2$.
Substituting the explicit expressions
\begin{align*}
p =& \frac{n \left(3 a m (b m n + c n + c) - n\right)}{3 (b m n + c n + c)^2}, \\
q =& \frac{n \left(9 a m n (b m n + c n + c) + 27 m (b m n + c n + c)^2 - 2 n^2\right)}{27 (b m n + c n + c)^3},
\end{align*}
we obtain the expanded form:
\begin{align*}
\mathfrak{D}_{(p,q)} =&\frac{m n^2}{108 (b m n+c n+c)^4}\Big(2 c m n (n+1) (2 a^3 m+9 a+27 b m)+n^2(4 a^3 b m^3-a^2 m\\
&+18 a b m^2+27 b^2 m^3-4)+27 c^2 m (n+1)^2\Big).
\end{align*}
This root distribution analysis leads naturally to the following classification theorem for positive equilibria, which is determined by the discriminant $\mathfrak{D}_{(p,q)}$:
\begin{thm}\label{th:2.11}
System \eqref{eq:2.4} always has a disease-free equilibrium $(0,0)$. Moreover, the number of positive equilibria is determined by the discriminant $\mathfrak{D}_{(p,q)}$ as follows:
\begin{enumerate}
    \item[\rm (1)] If $\mathfrak{D}_{(p,q)} > 0$, then there exists exactly one real root of \eqref{eq:2.19}, which is negative. Hence, system \eqref{eq:2.4} has no positive equilibrium.
    \item[\rm (2)] If $\mathfrak{D}_{(p,q)} = 0$, then \eqref{eq:2.19} has a double real root and a distinct simple real root. Under the biological constraints, the double root is positive, and the simple root is negative. Thus, system \eqref{eq:2.4} has exactly one positive equilibrium (a double equilibrium).
    \item[\rm (3)] If $\mathfrak{D}_{(p,q)} < 0$, then \eqref{eq:2.19} has three distinct real roots. Among them, exactly one is negative and the other two are positive. Therefore, system \eqref{eq:2.4} has two distinct positive equilibria.
\end{enumerate}
\end{thm}

\begin{proof}
The existence of the disease-free equilibrium $(0,0)$ is trivial. The number of real roots of the cubic equation \eqref{eq:2.19} is classical and depends on the sign of the discriminant $\mathfrak{D}_{(p,q)}$ (see, e.g. \cite{Saunders1980} or any standard algebra reference):
 If $\mathfrak{D}_{(p,q)} > 0$, one real root (and two complex conjugates).
 If $\mathfrak{D}_{(p,q)} = 0$, a multiple root and all roots real.
 If $\mathfrak{D}_{(p,q)} < 0$, three distinct real roots.

We now analyze the sign of the roots under the biological constraints. From \eqref{eq:2.10}, the positive equilibria correspond to positive roots of
\[
x^3 + \theta x^2 + p_1 x + q_1 = 0,
\]
where $q_1 = \frac{mn}{c(n+1) + bmn} > 0$ and $\theta = -\frac{n}{c(n+1) + bmn} < 0$. By Descartes' rule of signs, the number of positive roots is either zero or two. Since $q_1 > 0$, the product of the roots is negative, implying an odd number of negative roots. The sum of the roots is $-\theta > 0$, which precludes three negative roots. Hence, there is exactly one negative root and either two positive roots or a pair of complex conjugates. The discriminant condition $\mathfrak{D}_{(p,q)}$ then distinguishes these cases, yielding the stated conclusion.
\end{proof}

\begin{rem} \label{rem:2.12}
While Theorem \ref{th:2.9}(3) describes the cusp singularity arising from a triple root, this case is biologically infeasible. Theorem \ref{th:2.11} shows that under the model constraints, equation \eqref{eq:2.10} cannot have a triple positive root. The only attainable degeneracy is a double positive equilibrium, for which the bifurcation set is locally diffeomorphic to a line (Theorem \ref{th:2.9}(2)).
\end{rem}

\section{Stability and Local Bifurcation Analysis}
This section provides a comprehensive stability analysis of all equilibria for system \eqref{eq:2.4}. We first confirm the global stability of the disease-free equilibrium. The core of our analysis then focuses on the positive equilibria, where we establish criteria for their existence and local stability properties. This leads to the precise identification of parameter regimes that yield structurally unstable dynamics, including saddle-node and Bogdanov-Takens points. The theoretical results derived here, particularly concerning the degeneracy conditions at these bifurcation points, form the essential foundation for the analysis of higher-codimension phenomena in the following sections. \par
 The Jacobian matrix of system \eqref{eq:2.4} at $(0,0)$ is
 \begin{align*}
\left(
  \begin{array}{cc}
                  -m &0\\
                  1& -n
     \end{array}
\right),
 \end{align*}
 which has two negative eigenvalues $-m$ and $-n$. We obtain the following result.

\begin{thm} \label{th:2.15} The disease-free equilibrium $(0,0)$ of system \eqref{eq:2.4} is a stable hyperbolic node.
\end{thm}
\begin{figure}[h]
\centering
\includegraphics[width=0.4\textwidth]{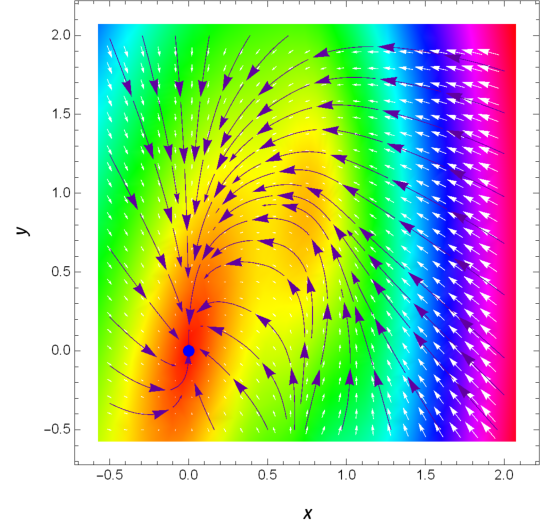}
\caption{ The phase portraits of system (2.4) (the streamlines with density distribution) for $a=-1.5,b=1,c=0.3,m=1.4,n=0.6$ at $(x^*,y^*)=(0,0)$.}
\label{fig:9}
\end{figure}

A phase portrait is shown in Figure \ref{fig:9}.
To discuss whether the disease can invade the population, we study the global stability of the equilibrium $(0,0)$. Since $x=0$ is an invariant line and $\Sigma_{\text{red}}$ is positively invariant, index theory implies that system \eqref{eq:2.4} admits no nontrivial periodic orbits in $\mathbb{R}^2_+$ when it lacks positive equilibria.

\begin{thm}\label{th:3.2}
The disease-free equilibrium $\left(\frac{\Lambda}{d},0,0\right)$ of \eqref{eq:2.1} is globally asymptotically stable in the interior $\mathbb{R}_+^3$, and the disease cannot invade the population whenever the condition
\(
\mathfrak{D}_{(p,q)} > 0
\)
is satisfied.
\end{thm}

\begin{rem}\label{rem:3.3}
This remark analyzes the critical threshold for disease eradication through  $\mathfrak{D}_{(p,q)}=\dfrac{m n^2}{108 (b m n+c n+c)^4}\widetilde{\rho}(b)$, where
\begin{align*}
\widetilde{\rho}(b):=&\ 27 m^3 n^2 b^2 + (4 a^3 m^3 n^2 + 18 a m^2 n^2 + 54 c m^2 n (n+1))b \\
&+ 27 c^2 m (n+1)^2 + 2 a c m (2 a^2 m + 9)n (n+1) - (4 + a^2 m)n^2.
\end{align*}
Treating $\widetilde{\rho}$ as a quadratic function of $b$, the equation $\widetilde{\rho}(b) = 0$ yields two distinct real roots $b_1 < b_2$:
\begin{align*}
b_1 &= -\frac{2 a^3 m^3 n + 2 n \left(m(a^2 m + 3)\right)^{3/2} + 9 a m^2 n + 27 c m^2 (n+1)}{27 m^3 n}, \\
b_2 &= -\frac{2 a^3 m^3 n - 2 n \left(m(a^2 m + 3)\right)^{3/2} + 9 a m^2 n + 27 c m^2 (n+1)}{27 m^3 n}.
\end{align*}
Using the parameter mapping from \eqref{eq:2.5}, these roots correspond to critical psychological effect thresholds:
\begin{align*}
\gamma_i := \frac{b_i \kappa \Lambda}{d\mu \sqrt{d\mu/(\kappa\Lambda)}}, \quad i = 1, 2,
\end{align*}
that is,
\begin{align*}
\gamma_1 &= -\mu  \left(\frac{2 \beta ^3}{27 \mu }+\frac{\beta  \kappa  \Lambda }{3 d \mu  (d+\mu )}+\frac{\kappa  (d+\delta +\mu )}{\mu  (d+\delta ) (d+\mu )}\right)-\frac{2}{27} \left(\frac{\beta ^2 d^2+\beta ^2 d \mu +3 \kappa  \Lambda }{d (d+\mu )}\right)^{3/2}, \\
\gamma_2 &= -\mu  \left(\frac{2 \beta ^3}{27 \mu }+\frac{\beta  \kappa  \Lambda }{3 d \mu  (d+\mu )}+\frac{\kappa  (d+\delta +\mu )}{\mu  (d+\delta ) (d+\mu )}\right)+\frac{2}{27} \left(\frac{\beta ^2 d^2+\beta ^2 d \mu +3 \kappa  \Lambda }{d (d+\mu )}\right)^{3/2}.
\end{align*}
The disease will die out for all positive initial populations when the psychological effect parameter $\gamma$ satisfies one of the following conditions, depending on the relative positions of $\gamma_1$ and $\gamma_2$ with respect to zero:
\begin{enumerate}
    \item[(i)] If $\gamma_1 < \gamma_2 < 0$, then eradication occurs when $\gamma > 0$;
    \item[(ii)] If $\gamma_1 < 0 < \gamma_2$, then eradication occurs when $\gamma > \gamma_2$;
    \item[(iii)] If $0 < \gamma_1 < \gamma_2$, then eradication occurs when $0 < \gamma < \gamma_1$ or $\gamma > \gamma_2$.
\end{enumerate}
\end{rem}

The following notations are presented, which will be used in subsequent results and their proofs.
 \begin{align}
 \begin{split}
\zeta&:=-\frac{x^*(- 2a^2 m n-6 n+(9 b m n+12 c n+12c)(x^*) +(4ac+4 a c n)(x^*)^2+3 b n(x^*)^3)}{2n(1+a x^*+b (x^*)^3)^2},\\
\eta&:=-\frac{x^*(-2 a^2 m n-6 n+(9 b m n+12 c n+9 c)x^* +(2ac+4 a c n)(x^*)^2+3 b n(x^*)^3)}{n(1+a x^*+b (x^*)^3)^2},\\
c^*&:=\frac{-bmn^*(x^*)^3+n^*(x^*)^2-amn^*x^*-mn^*}{(n^*+1)(x^*)^3},\\
n^*&:=\frac{1}{2} \left(-1+\frac{\sqrt{\left(1+a x^*+b \left(x^*\right)^3\right) \left( (1-4 m)(1+a x^*+b (x^*)^3)+4(x^*)^2\right)}}{1+a x^*+b \left(x^*\right)^3}\right),\\
x^*&:=am+\sqrt{a^2 m^2+3 m},\\
y^*&:=\dfrac{1}{n}(am+\sqrt{a^2 m^2+3 m}).
 \end{split}
 \end{align}

\begin{thm}\label{th:3.4}
Assume that $\mathfrak{D}_{(p,q)} = 0$,
and the following inequality holds
\begin{equation}\label{eq:3.2}
-bmn(x^*)^3 + n(x^*)^2 - amnx^* - mn > 0,
\end{equation}
then system \eqref{eq:2.4} satisfies:
\begin{enumerate}[label=\textnormal{(\arabic*)}]
\item
For $n < n^*$ (resp. $n > n^*$), the unique positive equilibrium $(x^*,y^*)$ is a repelling (resp. attracting) saddle-node.
\item
At $n = n^*$, system \eqref{eq:2.4} topologically equivalent to the second-order normal form:
\begin{equation}\label{eq:3.3}
       \begin{aligned}
       \dfrac{{\rm d} x}{{\rm d}t}&=y,\\
         \dfrac{{\rm d} y}{{\rm d}t}&= \zeta\xi_1 x^2+\eta xy+\mathcal{O}(|(x,y)|^3),
\end{aligned}\end{equation}
where $\zeta=\xi_3-\dfrac{\xi_1\xi_4}{\xi_2}, \eta=2\xi_3-\dfrac{\xi_1\xi_4}{\xi_2}$, and $\xi_i, i=1,2,3,4$ are defined in \eqref{paraxi}.
\end{enumerate}
\end{thm}
To enhance readability, we relegate the complete proof of this theorem to the \ref{proof_th:3.4}.

\begin{figure}[htbp]
  \centering
  \begin{minipage}[c]{0.36\textwidth}
    \centering
    \includegraphics[width=\textwidth]{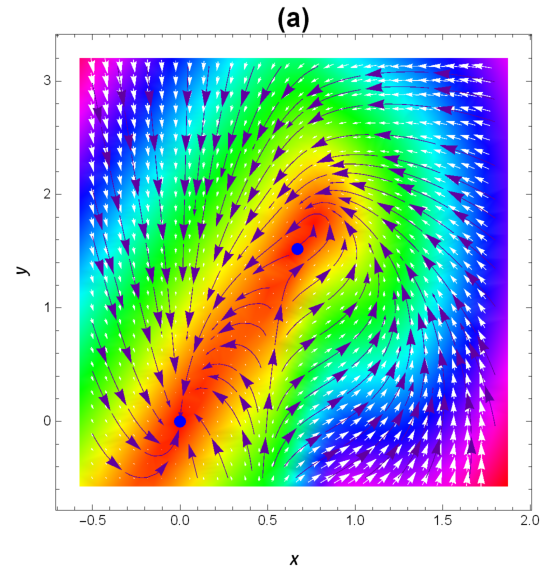}\\
  \end{minipage}\quad\quad\quad
  \begin{minipage}[c]{0.36\textwidth}
    \centering
    \includegraphics[width=\textwidth]{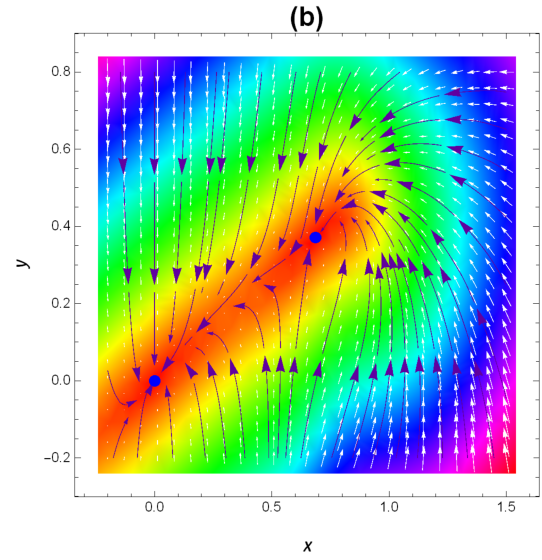} \\
  \end{minipage}\\
 \caption{The phase portraits of system (2.4) (the streamlines with density distribution). (a) A repelling saddle-node  for $a=-1.5,b=1,c=0.3,m=0.5,n=0.426960$ at $(x^*,y^*)=(0.686141, 1.60704)$; (b) A attractive saddle-node  for $a=-1.5,b=1,c=0.650837,m=0.5,n=1.85$ at $(x^*,y^*)=(0.686141, 0.370887)$.}
  \label{fig:5}
\end{figure}

Numerical simulations in Fig.~\ref{fig:5} verify the two dynamical regimes described in Theorem~\ref{th:3.4}: Repelling saddle-node ($n < n^*$) shown in Fig.~\ref{fig:5} (a), attracting saddle-node ($n > n^*$) shown in Fig.~\ref{fig:5} (b).\par
We note that assertion (2) of Theorem \ref{th:3.4} implies that $(x^*, y^*)$ is a candidate for a Bogdanov-Takens cusp (codimension 2 or higher). Numerical simulations in Fig.~\ref{fig:7} give  the phase portraits of a codimension-2  and  a codimension-3 cusp, respectively. The detailed analysis of the Bogdanov-Takens bifurcation of codimension 2 or 3 will be given in the next section.
\begin{rem}\label{rem:3.5}
When the psychological effect parameter $\gamma$ attains a critical value $\gamma_i$, system (2.1) possesses two equilibria: a disease-free equilibrium and an endemic equilibrium. Under this condition, the eventual outcome of the disease outbreak, specifically whether it is eradicated or persists, depends on both the infection rate $\kappa$ and the initial population size. Specifically, if the following inequality holds:
\[
\frac{\sqrt{d^2+\delta ^2+2 d (\delta +\mu )+\delta  \mu +\mu ^2}}{\mu }\leq \frac{x^*}{1+x^* \sqrt{\frac{d \mu }{\kappa  \Lambda }} \left(\beta +\frac{\gamma_i  d \mu }{\kappa  \Lambda }(x^*)^2\right)}\ (\text{i.e.} \ n\leq n^*),
\]
then the disease will be eradicated for almost all initial populations (as shown in Fig. \ref{fig:5}~(a) and Fig. \ref{fig:7}). Conversely, if this inequality is reversed, the infection will persist for some initial conditions and approach the positive endemic equilibrium (Fig.\ref{fig:5}~(b)). Here, $x^*$ is given by:
\[
x^* = \left(\left( \frac{d}{\mu} + 1 \right) \left( \frac{\beta^2 d(d + \mu)}{\kappa\Lambda} + 3 \right)\right)^{1/2} + \beta \left( \frac{d}{\mu} + 1 \right) \sqrt{ \frac{d\mu}{\kappa\Lambda} }.
\]
Additionally, the specific critical value $\gamma_i$ is determined according to the cases established in Remark \ref{rem:3.3}: $\gamma = \gamma_2$ when $\gamma_1 < 0 < \gamma_2$, or $\gamma = \gamma_i$ (where $i = 1$ or $2$) when $0 < \gamma_1 < \gamma_2$.
\end{rem}
\begin{figure}[htbp]
  \centering
\begin{minipage}[c]{0.38\textwidth}
    \centering
    \includegraphics[width=\textwidth]{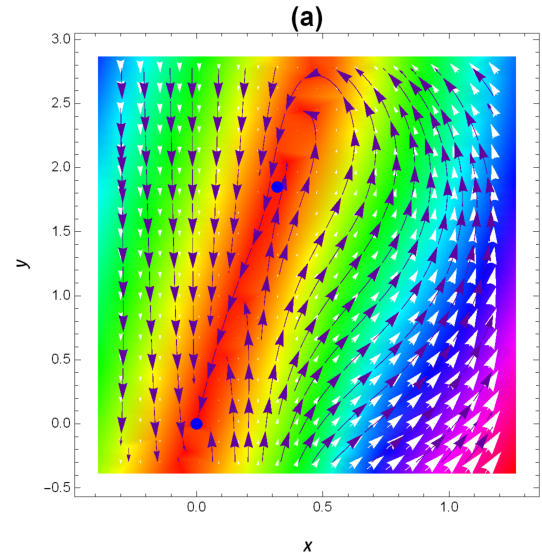}\\
  \end{minipage}\quad\quad\quad
  \begin{minipage}[c]{0.36\textwidth}
    \centering
    \includegraphics[width=\textwidth]{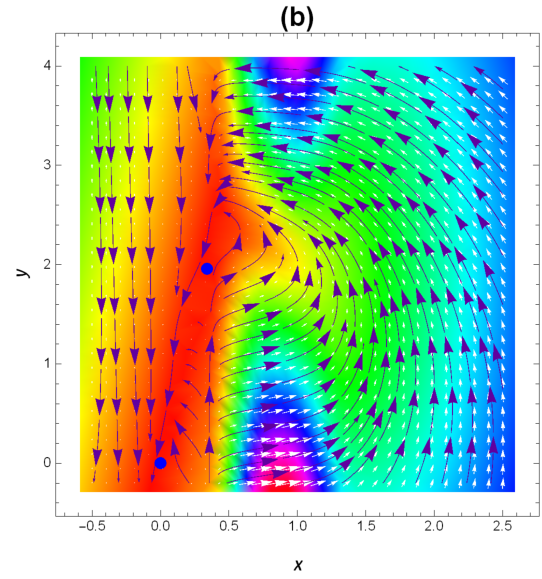} \\
  \end{minipage}
 \caption{(a) A codimension-2 cusp with parameters $(a, b, c, m, n) = (-1.5, 1.8045924, 0.330275, 0.05, 0.172824)$ and the double positive equilibrium $(x^{*}, y^{*}) = (0.319493, 1.848663)$. (b) A codimension-3 cusp  with parameters $(a, b, c, m, n) = (-1.8, 1, 0.330275, 0.064380, 0.172824)$  emerges at the double positive equilibrium $(x^{*}, y^{*}) = (0.338614, 1.959299)$.
    }
  \label{fig:7}
\end{figure}
\begin{thm}\label{th:3.6}
When the discriminant satisfies $\mathfrak{D}_{(p,q)} < 0$, system \eqref{eq:2.4} possesses two distinct positive equilibria $(x_1, y_1)$ and $(x_2, y_2)$ with $x_1 < x_2$. Let $J(x,y) = d\phi_\lambda(x,y)$ denote the Jacobian matrix at equilibrium $(x,y)$, and let $N(x,y)$ represent the numerator of $\det(J(x,y))$ derived in equation \eqref{eq:2.8}. We conclude that $(x_1, y_1)$ is always a hyperbolic saddle,
the equilibrium $(x_2,y_2)$ exhibits three stability regimes:
\begin{enumerate}[label=\textnormal{(\arabic*)}]
\item Stable hyperbolic focus/node when $\text{\rm Tr}(J(x_2,y_2)) < 0$;
\item Weak focus or center when $\text{\rm Tr}(J(x_2,y_2)) = 0$;
\item Unstable hyperbolic focus/node when $\text{\rm Tr}(J(x_2,y_2)) > 0$.
\end{enumerate}
\end{thm}

\begin{proof}
By Theorem \ref{th:2.11}, the condition $\mathfrak{D}_{(p,q)} < 0$ ensures two distinct positive equilibria.
The characteristic cubic equation $\omega(z) = z^3 + pz + q = 0$ admits three real roots $\bar{z} , z_1 , z_2$ under this discriminant regime, where$\bar{z} < z_\ell < z_1 < z_r < z_2$, $z_\ell = -\sqrt{-p/3}$ and $z_r = \sqrt{-p/3}$ are critical points. The derivative $\omega'(z) = 3z^2 + p$ indicates the monotonicity: increase on $(-\infty, z_\ell) \cup (z_r, +\infty)$ and decrease on $(z_\ell, z_r)$. Consequently, the equilibrium-associated roots satisfy $\omega'(z_1) < 0$ for $z_1 \in (z_\ell, z_r)$ and $\omega'(z_2) > 0$ for $z_2 \in (z_r, +\infty)$.\par
From \eqref{eq:2.8}, the Jacobian determinant decomposes as:
\begin{align*}
\det(J(x,y)) =&\frac{N(x,y)}{1+a x+b x^3},\\
N(x,y) =&c x^3 - n(-2 a m x -4 x^3 (b m+c) -3 c x^2 y -m +3 x^2,
\end{align*}
where the denominator $1+a x+b x^3 > 0$ for all positive $x$. Through coordinate transformation \eqref{eq:2.13} and polynomial reductions \eqref{eq:2.14}-\eqref{eq:2.20}, the numerator simplifies to:
\[
N(x,y) = (c(n+1)+bmn)\left(\omega(z) + x\omega'(z)\right).
\]
\par
For the lower equilibrium $(x_1,y_1)$ corresponding to $z_1$, we have:
\[
N(x_1,y_1) = (c(n+1)+bmn)x_1\omega'(z_1) < 0,
\]
since $x_1 > 0$ and $\omega'(z_1) < 0$ in the decreasing interval $(z_\ell, z_r)$. This negative determinant confirms the saddle point property.\par
For the upper equilibrium $(x_2,y_2)$ corresponding to $z_2$, we find:
\[
N(x_2,y_2) = (c(n+1)+bmn)x_2\omega'(z_2) > 0
\]
given $x_2 > 0$ and $\omega'(z_2) > 0$ in the increasing interval $(z_r, +\infty)$. The stability classification follows from analyzing the trace $\text{Tr}(J(x_2,y_2))$ using standard linearization techniques.
\end{proof}
\begin{rem}
Theorem \ref{th:3.6} and Remark \ref{rem:3.3} reveal that epidemic severity increases when the psychological effect parameter $\gamma$ falls within specific ranges determined by the relative positions of $\gamma_1$ and $\gamma_2$. Specifically, when $\gamma_1 < 0 < \gamma_2$, the range $0 < \gamma < \gamma_2$ leads to multiple positive coexistent steady states, while when $0 < \gamma_1 < \gamma_2$, the range $\gamma_1 < \gamma < \gamma_2$ produces this multi-stability phenomenon. In both cases, the existence of multiple endemic equilibria enhances disease persistence and potential severity.
\end{rem}
The nature of the Hopf bifurcation at $(x_2, y_2)$ (i.e., its criticality and the possible emergence of multiple limit cycles in degenerate cases) will be thoroughly analyzed in section 5.

\section{Bogdanov-Takens bifurcations}
This section provides a detailed analysis of the Bogdanov-Takens (BT) points identified in the previous sections. The central objective is to fully characterize these degenerate equilibria by deriving their normal forms for both codimension-2 and codimension-3 cases. We then construct universal unfoldings to describe the complete local bifurcation diagram in parameter space and support our theoretical findings with numerical simulations that confirm the emergence of predicted dynamical features, such as limit cycles and homoclinic orbits.
\subsection{Bogdanov-Takens bifurcation with codimension-2}
\begin{thm}\label{th:4.1} Withe the same conditions with Theorem \ref{th:3.4}, when $n = n^*$ and $\zeta\eta \neq 0$, $(x^*,y^*)$ becomes a Bogdanov-Takens cusp of codimension two.
\end{thm}
\begin{proof}Continuing from the proof of Theorem \ref{th:3.4}, under the non-degeneracy condition $\zeta\eta \neq 0$, the scaling transformation
\begin{align}\label{eq:4.1}
x\rightarrow \frac{\xi_1\zeta}{\eta^2}x,\quad y\rightarrow \frac{\xi_1^2\zeta^2}{\eta^3}y,\quad t\rightarrow \frac{\eta}{\xi_1\zeta}t,
\end{align}
reduces the system to the normal form:
\begin{equation}\label{eq:4.2}
       \begin{aligned}
       \dfrac{{\rm d}x}{{\rm d}t}&=y,\\
         \dfrac{{\rm d}y}{{\rm d}t}&= x^2+xy+\mathcal{O}(|(x,y)|^3).
\end{aligned}\end{equation}
 Although the system contains higher-order terms $\mathcal{O}(|(x,y)|^3)$, the theories established in \cite{Perko1996} (Theorems 2--3) ensure that these terms do not alter the topological structure of the nonhyperbolic critical point or the qualitative types of bifurcations in its unfolding. Therefore, the local bifurcation phenomena are completely characterized by the quadratic part of the normal form ($x^2 + xy$ in \eqref{eq:4.2}), and our subsequent analysis will focus on this essential structure.
\end{proof}
\begin{rem}
The proof of Theorem~\ref{th:3.4} shows that the conditions $\mathfrak{D}_{(p,q)} = 0$ in Theorem~\ref{th:4.1} are equivalent to the system \eqref{eq:2.34}. Hence, the same conditions may be rewritten as $c = c^*$, $n = n^*$, and $\zeta\eta \neq 0$.
\end{rem}
In the following, we will take $c$ and $n$ as bifurcation parameters and develop a universal
unfolding for the cusp singularity of order 2 near the point $(x^*,y^*)$. Let
\begin{align*}
c=c^*+\epsilon_1,\quad n=n^*+\epsilon_2.
\end{align*}
Then the perturbed system is
\begin{equation}\label{eq:4.3}
       \begin{aligned}
       \dfrac{{\rm d} x}{ {\rm d} t}&=\dfrac{x^3}{1+a x+b x^3}\Big(1-(c^*+\epsilon_1)(x+y)\Big)-mx,\\
         \dfrac{{\rm d} y}{{\rm d}t}&=  x- (n^*+\epsilon_2)y.
\end{aligned}\end{equation}

\begin{thm}\label{th:4.3}
Let $\zeta \eta \neq 0$. For sufficiently small perturbations $\epsilon = (\epsilon_1, \epsilon_2)$, the system \eqref{eq:4.3} is $C^\infty$-equivalent to the following system:
\begin{equation}\label{eq:4.4}
       \begin{aligned}
       \dfrac{{\rm d} x}{{\rm d}t}&=y,\\
         \dfrac{{\rm d}y}{{\rm d}t}&= \mu_1(\epsilon)+\mu_2(\epsilon)y+ x^2+xy,
\end{aligned}\end{equation}
 here $R(x,y,\epsilon)$ is $C^\infty$ in $(x, y, \epsilon)$, and the Jacobian matrix of the transformation from $(\epsilon_1, \epsilon_2)$ to $(\mu_1, \mu_2)$ satisfies:
\begin{equation}\label{eq:4.5}
\det\left(\frac{\partial(\mu_1, \mu_2)}{\partial(\epsilon_1, \epsilon_2)}\right)\Big|_{\epsilon =0} \neq 0.
 \end{equation}
 Consequently, system \eqref{eq:4.3} is a universal unfolding of the cusp singularity of order two, and it undergoes a Bogdanov-Takens bifurcation of codimension two in a small neighborhood of the unique positive equilibrium $(x^*, y^*)$. Therefore, there exist parameter values such that system \eqref{eq:2.4} exhibits an unstable limit cycle, and other parameter values where it displays an unstable homoclinic loop.
 \end{thm}
The complete proof is deferred to the \ref {proof_th_4.3} for clarity.
\par
According to Chow's results \cite{Chow1994}, there exists a   neighborhood $U$ of $(\epsilon_1,\epsilon_2)=(0,0)$ in $\mathbb{R}^2$ such
that the bifurcation diagram of (1.5) inside $U$ consists of the origin
$(\mu_1,\mu_2)=(0,0)$. Denoting $\text{jet}\mu(\epsilon_1,\epsilon_2)$,  that is, the 2-jet of a bivariate function, which refers to its second-order Taylor polynomial, retaining terms up to quadratic terms in the expansion, one can obtain the following  local representations of the bifurcation
curves up to second-order approximations by using the 2-jet of $\mu_1(\epsilon_1,\epsilon_2)$ and $\mu_2(\epsilon_1,\epsilon_2)$, that is,
\begin{align*}
\text{j}^2(\mu_1(\epsilon_1,\epsilon_2))&=r_1\epsilon_1+r_2\epsilon_2+r_3\epsilon_1^2+r_4\epsilon_1\epsilon_2+r_5\epsilon_2^2,\\
\text{j}^2(\mu_2(\epsilon_1,\epsilon_2))&=s_1\epsilon_1+s_2\epsilon_2+s_3\epsilon_1^2+s_4\epsilon_1\epsilon_2+s_5\epsilon_2^2.
\end{align*}
${\rm (\romannumeral1) }$ The saddle-node bifurcation curve is given by
{\footnotesize
\begin{equation}
\begin{split}
&SN=\Big\{(\epsilon_1,\epsilon_2)\in U\big|j^2(\mu_1(\epsilon_1,\epsilon_2))=0,j^2(\mu_2(\epsilon_1,\epsilon_2))\neq0\Big\}=SN^+\bigcup SN^-\\
=&\Big\{(\epsilon_1,\epsilon_2)\in U\big|j^2(\mu_1(\epsilon_1,\epsilon_2))=0,\epsilon_1<0\Big\}\bigcup\Big\{(\epsilon_1,\epsilon_2)\in U\big|j^2(\mu_1(\epsilon_1,\epsilon_2))=0,\epsilon_1>0\Big\}.\\
\end{split}
\end{equation}}
${\rm (\romannumeral2) }$ The Hopf bifurcation curve is given by
\begin{equation}
H=\bigg\{(\epsilon_1,\epsilon_2)\in U\big|j^2(\mu_1(\epsilon_1,\epsilon_2))=-j^2(\mu_2^2(\epsilon_1,\epsilon_2)),j^2(\mu_2(\epsilon_1,\epsilon_2))>0\Big\}.\\
\end{equation}
${\rm (\romannumeral3) }$ The Homoclinic bifurcation curve is given by
\begin{equation}
H=\bigg\{(\epsilon_1,\epsilon_2)\in U\big|j^2(\mu_1(\epsilon_1,\epsilon_2))=-j^2\Big(\dfrac{49}{25}\mu_2^2(\epsilon_1,\epsilon_2)\Big),j^2(\mu_2(\epsilon_1,\epsilon_2))>0\Big\}.
\end{equation}

\begin{center}
\begin{tabular}{cc}
\begin{minipage}[c]{0.5\textwidth}
\centering
\includegraphics[scale=0.45]{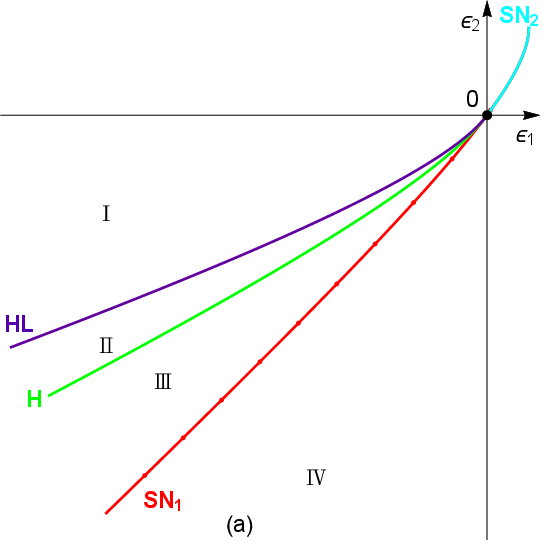}\\
\end{minipage}
\begin{minipage}[c]{0.5\textwidth}
\centering
\includegraphics[scale=0.45]{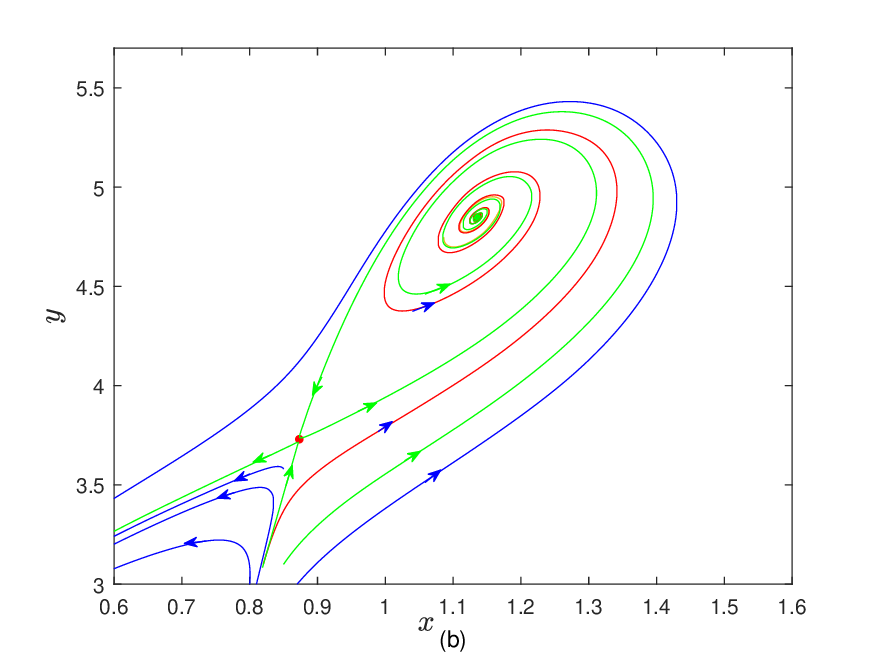}\\
\end{minipage}
\end{tabular}
\end{center}
\begin{center}
\begin{tabular}{cc}
\begin{minipage}[c]{0.5\textwidth}
\centering
\includegraphics[scale=0.45]{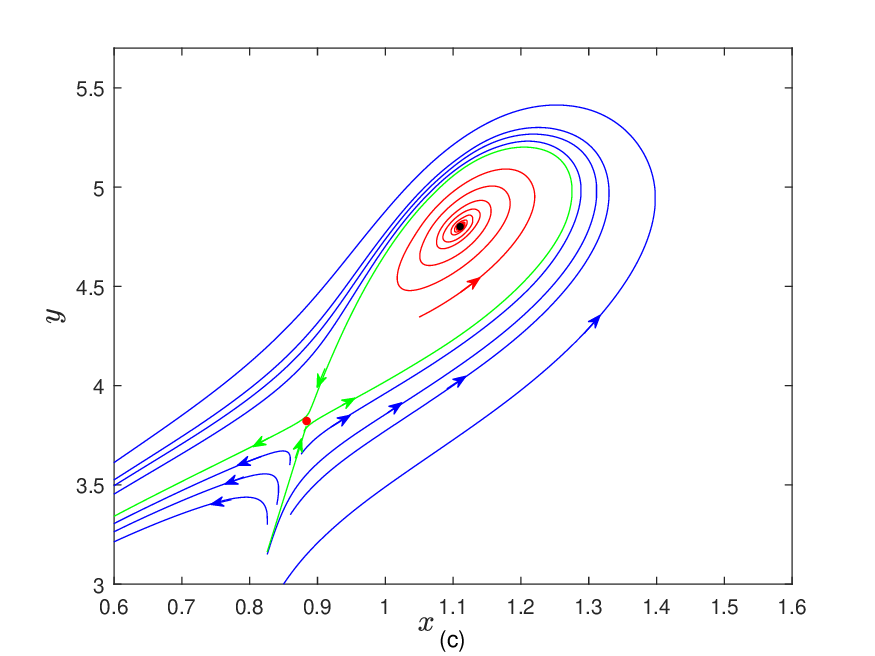}\\
\end{minipage}
\begin{minipage}[c]{0.5\textwidth}
\centering
\includegraphics[scale=0.45]{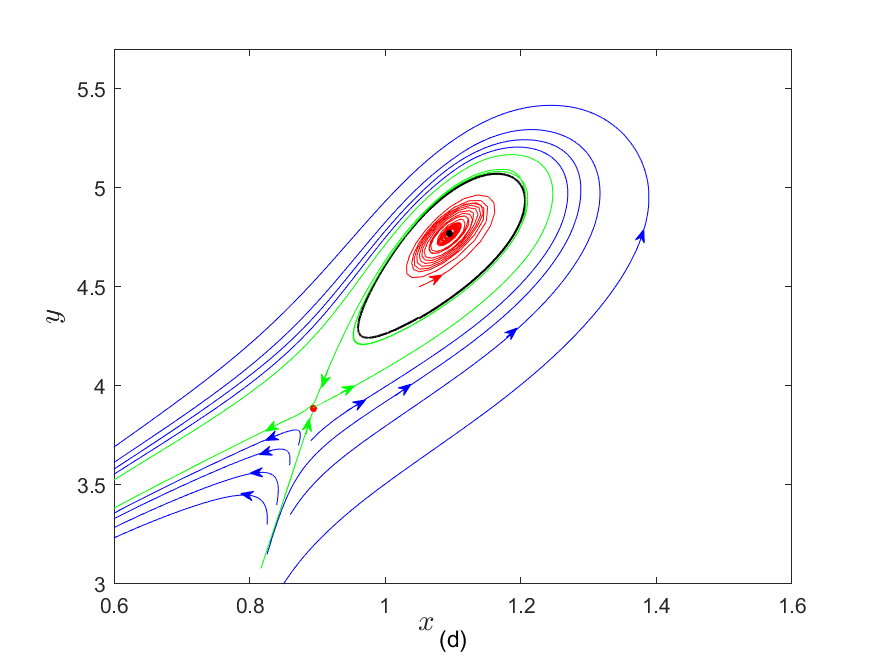}\\
\end{minipage}
\end{tabular}
\end{center}
\begin{center}
\begin{tabular}{cc}
\begin{minipage}[c]{0.5\textwidth}
\centering
\includegraphics[scale=0.45]{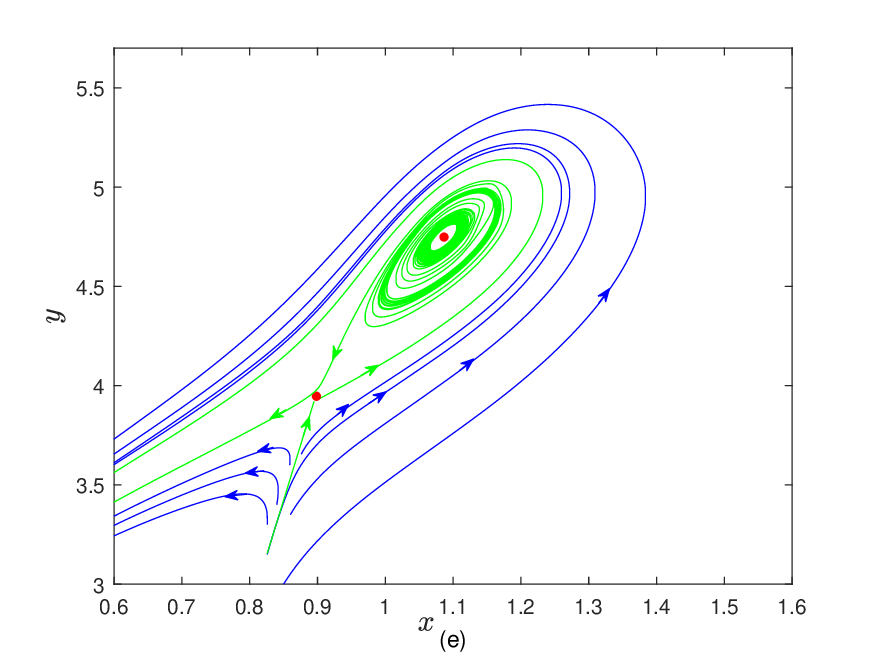}\\
\end{minipage}
\begin{minipage}[c]{0.5\textwidth}
\centering
\includegraphics[scale=0.45]{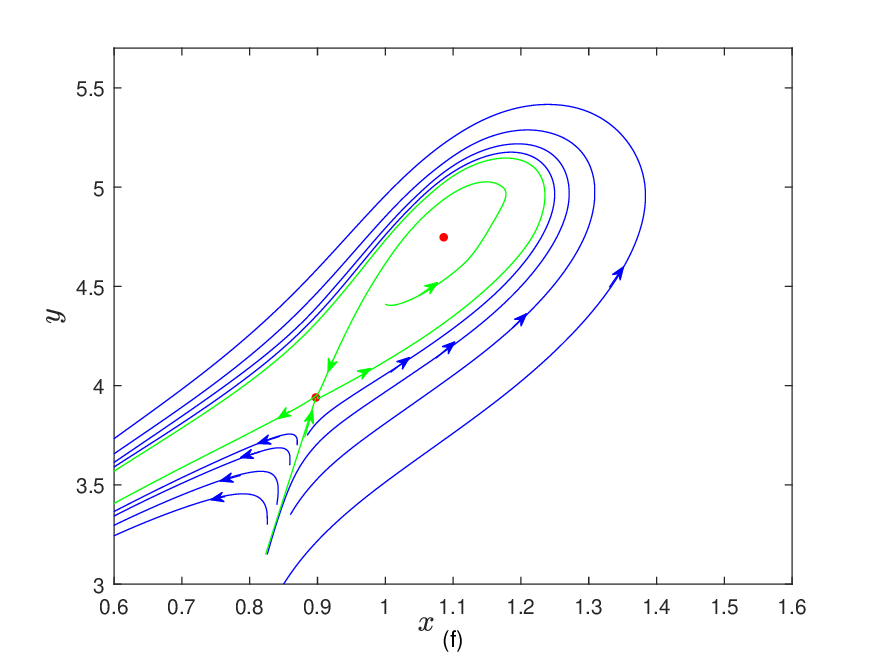}\\
\end{minipage}
\end{tabular}
\end{center}
\begin{center}
\begin{tabular}{cc}
\begin{minipage}[c]{0.5\textwidth}
\centering
\includegraphics[scale=0.45]{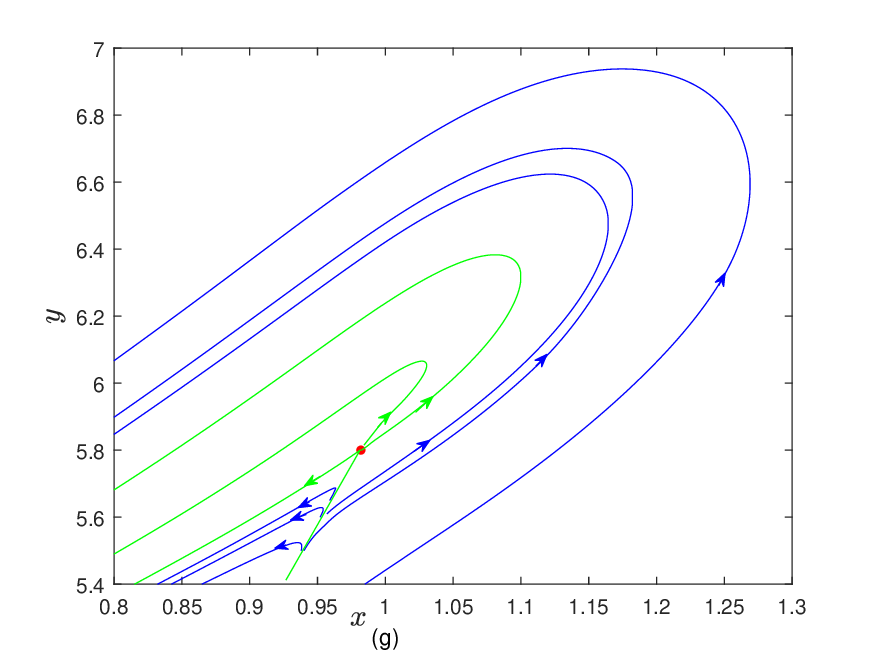}\\
\end{minipage}
\begin{minipage}[c]{0.5\textwidth}
\centering
\includegraphics[scale=0.45]{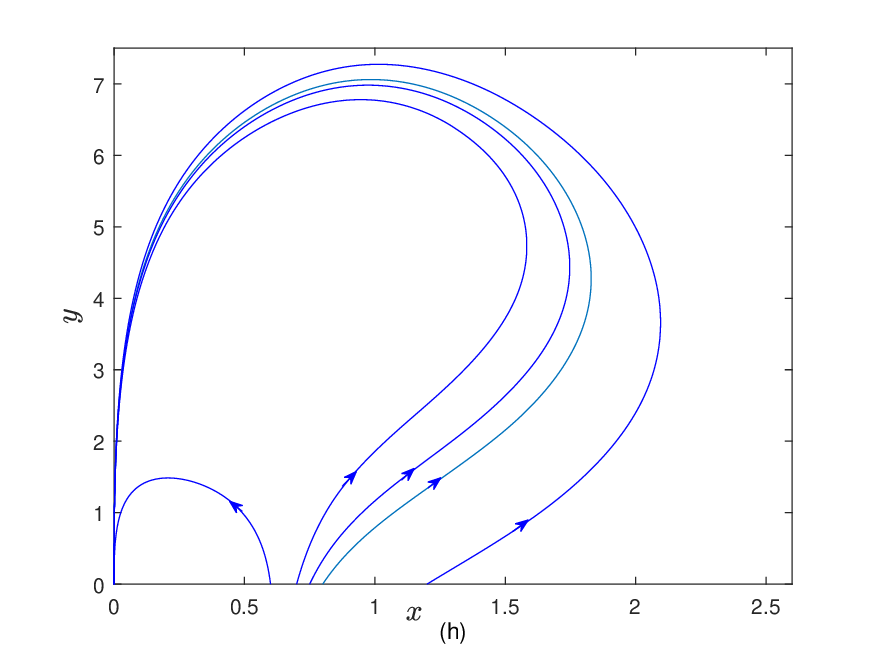}\\
\end{minipage}
\end{tabular}
\end{center}
\begin{figure}[H]
\ContinuedFloat
  \centering
  \begin{minipage}[c]{0.5\textwidth}
    \centering
    \includegraphics[scale=0.45]{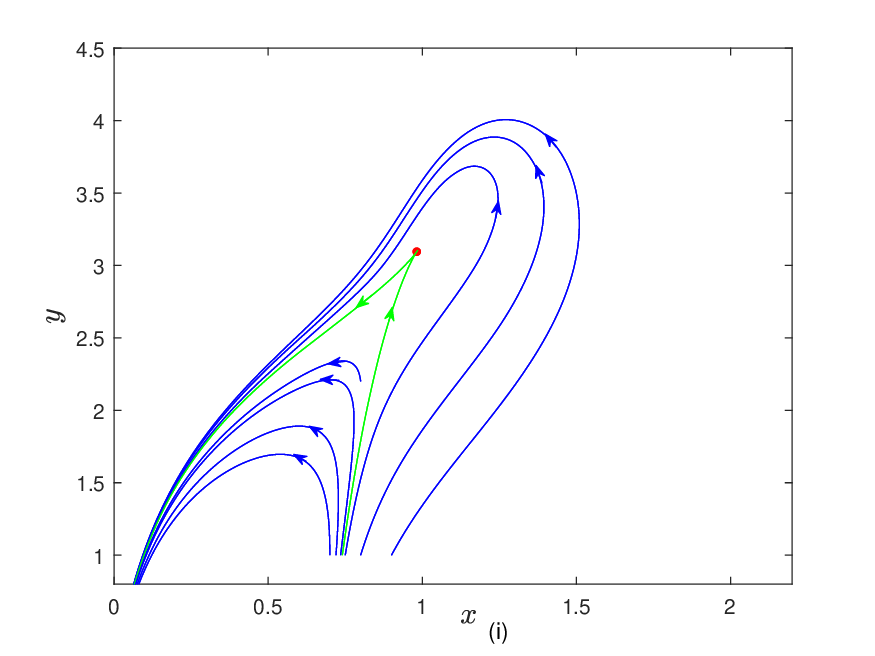}\\
  \end{minipage}
  \hfill
  \begin{minipage}[c]{0.45\textwidth}
    \centering
    \includegraphics[scale=0.45]{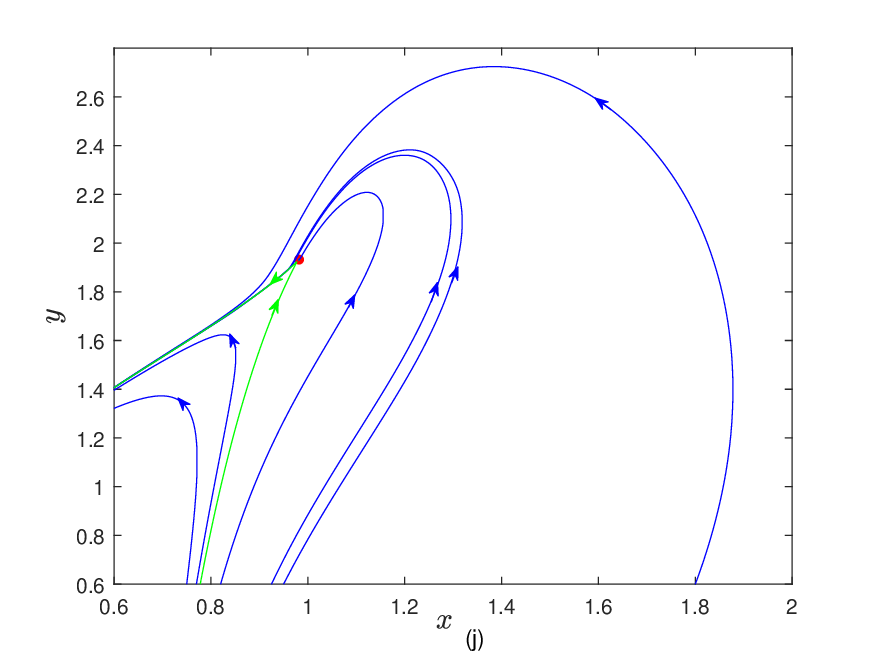}\\
  \end{minipage}
  \caption{\footnotesize The bifurcation diagram and phase portraits of system \eqref{eq:4.3} for $a=-0.3$, $b=0.5$, $c=0.1253449$, $n=0.3173105$, $m=0.4$. (a) The bifurcation diagram; (b) A stable focus when $(\epsilon_1,\epsilon_2) = (-0.0303449,- 0.0830482)$ lies in the region  I; (c) An unstable homoclinic loop when $(\epsilon_1,\epsilon_2) = (-0.0303449,- 0.08601394)$ lies on the curve $HL$; (d) When the point $(\epsilon_1, \epsilon_2) = (-0.0303449, -0.0876136)$ lies in region II, near the curve $H$, system exhibits a unique limit cycle in a small neighborhood of the focus.  (e) An unstable focus when $(\epsilon_1,\epsilon_2)= (-0.0303449, -0.0884821)$ lies in the curve $H$; (f) An unstable focus exists at $(\epsilon_1, \epsilon_2) = (-0.0303449, -0.08852161)$ in region III, away from the  curve $H$;  (g) An repelling saddle-node when $(\epsilon_1,\epsilon_2) = (-0.05,-0.148005)$ lies in the curve $SN^+$; (h) No positive equilibria when $(\epsilon_1,\epsilon_2) = (-0.0303449,0.177887)$ lies in the region  IV; (i) A cusp when $(\epsilon_1,\epsilon_2) = (0,0)$ lies in the origin $O$; (j) An attracting  saddle-node when $(\epsilon_1,\epsilon_2) = (0.05,0.190902)$ lies in the curve $SN^-$.}
  \label{fig:8}
\end{figure}

\begin{rem}\label{rem:4.4}
The saddle-node bifurcation curve near \(\epsilon = 0\) is locally diffeomorphic to a regular curve, as established by the unfolding theory in singularity theory (Theorems~\ref{th:2.9} and~\ref{th:2.11}), and is  observable in the bifurcation diagram (Figure~\ref{fig:8}(a)). The diagram in the original perturbation coordinates \((\epsilon_1, \epsilon_2)\) is topologically equivalent to the universal unfolding in the \((\mu_1, \mu_2)\)-plane under the local diffeomorphism. Specifically, the saddle-node bifurcation curve in the universal unfolding corresponds to a straight line $\mu_1=0$ in the \((\mu_1, \mu_2)\)-plane, reflecting the geometric simplification afforded by the normal form. This structural consistency between the singularity-theoretic analysis and dynamical systems approaches confirms the robustness of the codimension-2 cusp unfolding. The equivalence preserves the bifurcation hierarchy while accounting for smooth parameter-dependent distortions governed by the transversality condition \(\det\left(\frac{\partial \mu}{\partial \epsilon}\right)\big|_{\epsilon=0} \neq 0\).
\end{rem}

\begin{rem}
The bifurcation diagram and phase portraits of system \eqref{eq:4.3} for the given parameter values \( a = -0.3 \), \( b = 0.5, m=0.4 \), \( c = 0.1253449 \), and \( n = 0.3173105 \) are shown in the following figures. These figures (see Fig.~\ref{fig:8}) illustrate the dynamical behavior of the system as the perturbation parameters \(\epsilon_1\) and \(\epsilon_2\) vary, with different phenomena observed in four distinct regions marked by the curves \(H\) (Hopf bifurcation), \(HL\) (homoclinic loop), and \(SN\) (saddle-node bifurcation).
\begin{itemize}
    \item[(a)] \textbf{Bifurcation Diagram:} The bifurcation diagram shows how the qualitative nature of the system changes as \(\epsilon_1\) and \(\epsilon_2\) are varied. The diagram is divided into four regions by the three bifurcation curves \(H\), \(HL\), and \(SN\). Each region corresponds to different dynamical behaviors such as stable or unstable equilibria, limit cycles, and homoclinic orbits.

    \item[(b)] \textbf{Stable Focus:} For \((\epsilon_1, \epsilon_2) = (-0.0303449, -0.0830482)\), the system exhibits a stable focus located in region I. Here, the equilibrium point is stable with complex conjugate eigenvalues, indicating a damped oscillatory motion around the equilibrium.

    \item[(c)] \textbf{ Unstable Homoclinic Loop:} At \((\epsilon_1, \epsilon_2) = (-0.0303449, -0.08601394)\), the system shows an unstable homoclinic loop. The trajectory forms a homoclinic orbit, connecting the saddle point to itself, but the orbit is unstable, causing the trajectory to diverge over time.

    \item[(d)] \textbf{Unstable Limit Cycle:} When \((\epsilon_1, \epsilon_2) = (-0.0303449, -0.0876136)\), the system enters region II, where an unstable limit cycle is observed. This indicates a periodic solution that is unstable, where nearby trajectories either spiral outwards or inwards, depending on their initial conditions.

    \item[(e)] \textbf{ Unstable Focus:} For \((\epsilon_1, \epsilon_2) = (-0.0303449, -0.0884821)\), the system shows an unstable focus, located on the Hopf bifurcation curve \(H\). The equilibrium point is unstable with complex eigenvalues, but the system still exhibits oscillatory behavior.

    \item[(f)] \textbf{Unstable Behavior in Region III:} At \((\epsilon_1, \epsilon_2) = (-0.0303449, -0.08852161)\), the system is located in region III. Here, the system's dynamics are unstable, and this point lies near a critical threshold of bifurcations, suggesting the presence of chaotic or highly sensitive behavior to perturbations.

    \item[(g)] \textbf{Repelling Saddle-Node:} At \((\epsilon_1, \epsilon_2) = (-0.05, -0.148005)\), the system reaches a repelling saddle-node bifurcation along the \(SN^+\) curve. This indicates the creation of a pair of equilibria, one stable and one unstable, where the stable equilibrium repels trajectories.

    \item[(h)] \textbf{ No Positive Equilibria:} For \((\epsilon_1, \epsilon_2) = (-0.0303449, 0.177887)\), the system has no positive equilibria, as indicated by the point lying in region IV. In this case, the system's state does not admit a biologically meaningful equilibrium, implying extinction of the population.
 \item[(i)] \textbf{ Cusp Point:} Finally, at \((\epsilon_1, \epsilon_2) = (0, 0)\), the system reaches a cusp bifurcation at the origin. This point marks a critical transition in the system's dynamics, where small changes in parameters lead to significant qualitative changes in the system's behavior.

    \item[(j)] \textbf{Attracting Saddle-Node:} At \((\epsilon_1, \epsilon_2) = (0.05, 0.190902)\), an attracting saddle-node bifurcation is observed along the \(SN^-\) curve. Here, a stable and unstable equilibrium pair is formed, with the stable equilibrium attracting nearby trajectories.

\end{itemize}

These observations highlight the rich dynamical behavior of the system, including bifurcations leading to various types of equilibria, periodic orbits, and chaotic or unstable solutions. The interplay between the Hopf, homoclinic loop, and saddle-node bifurcations governs the transitions between different types of behavior in the system.
\end{rem}
\subsection{Bogdanov-Takens bifurcation: cusp of order 3}

To analyze the codimension-3 Bogdanov-Takens bifurcation associated with a third-order cusp, we derive the governing system by imposing the singularity conditions $\det(J^*)=0$ and $\mathrm{Tr}(J^*)=0$, along with the degeneracy condition $\eta=0$. This leads to the following coupled system for parameters $c$, $n$, and $b$:
\begin{equation*}
\begin{cases}
\mathcal{N}_T = c(x^*)^3 - n^2D = 0, \\
4(c(n+1)+bmn)(x^*)^3 - 3n(x^*)^2 + 2amnx^* + mn = 0,\\
\eta=0,
\end{cases}
\end{equation*}
or equivalently,
\begin{equation*}
\begin{cases}
\mathcal{N}_T = c(x^*)^3 - n^2D = 0, \\
(c(n+1) + bmn)(x^*)^{3} - n(x^*)^{2} + amnx^* + mn = 0,\\
\eta=0,
\end{cases}
\end{equation*}
where the critical value $x^*$ is explicitly given by
\[
x^* = am + \sqrt{a^2m^2 + 3m}.
\]
To resolve this system, we first solve the degeneracy condition $\eta=0$ for parameter $b$, obtaining
\begin{align*}
b=\dfrac{2 a^2 m n-4 a c n \left(x^*\right)^2-2 a c \left(x^*\right)^2-12 c n x^*-9 c x^*+6 n}{9 m n x^*+3 n \left(x^*\right)^3}.
\end{align*}
Substituting this expression for $b$ into the remaining equations and simplifying via the relation $(x^*)^2=2amx^*+3m$ yields the critical parameter values:
\begin{align*}
c&=\dfrac{4 m \left(a x^*+3\right)^2}{x^* \left(m \left(4 a^2 \left(x^*\right)^2+60 a x^*+90\right)+8 a^2 \left(x^*\right)^2+30 a x^*-18 \left(x^*\right)^2+27\right)}:=c^*,\\
n&= \frac{2 m \left(a x^*+3\right)}{2 a x^*+3}:=n^*.
\end{align*}
Consequently, the critical value $b^*$ is determined as
\begin{align*}
b=\dfrac{2 a^2 m n^*-4 a c^* n^*\left(x^*\right)^2-2 a c^* \left(x^*\right)^2-12 c^* n^* x^*-9 c^* x^*+6 n^*}{9 m n^* x^*+3 n^* \left(x^*\right)^3}:=b^*.
\end{align*}

To establish the normal form of the bifurcation, we employ a sequence of coordinate transformations. The following lemma provides the foundation for simplifying the system to its essential dynamics:

\begin{lem}[\cite{Dumortier1987}] \label{lem:4.6}
Consider the system
\begin{equation}\label{eq:3.27}
       \begin{aligned}
       \dfrac{{\rm d} x}{{\rm d}t}&=y,\\
         \dfrac{{\rm d}y}{{\rm d}t}&=x^2+\displaystyle\sum_{i=3}^4a_{i0}x^i +\displaystyle\sum_{i=1}^2y^i(a_{i+1,i}x^{i+1}+a_{i+2,i}x^{i+2})+\mathcal{O}(|(x,y)|^4).
       \end{aligned}
\end{equation}
Through smooth coordinate changes and time rescaling, the system \eqref{eq:3.27} can be transformed into the simplified form
\begin{equation}\label{eq:3.28}
       \begin{aligned}
       \dfrac{{\rm d}X}{{\rm d}t}&=Y,\\
         \dfrac{{\rm d}Y}{{\rm d}t}&=X^2+EX^3Y+\mathcal{O}(|(X,Y)|^4),
       \end{aligned}
\end{equation}
where the key coefficient is given by
\begin{align*}
E=a_{31}-a_{30}a_{21}.
\end{align*}
\end{lem}

Building upon this normal form reduction, we can now characterize the local dynamics near the degenerate singularity.

\begin{thm}\label{thm:cusp-equivalence}
Under the conditions $\xi_4 \neq -n\xi_3$ and $\xi_4= -2n\xi_3$ (i.e., when $\eta=0$), system \eqref{eq:2.4} is \(C^\infty\)-equivalent near the origin to
\begin{equation}\label{eq:cusp-normal-form}
\begin{aligned}
\dfrac{{\rm d} x}{{\rm d}t} &= y, \\
\dfrac{{\rm d} y}{{\rm d}t} &= x^2 + \chi x^3 y + \mathcal{O}(|(x,y)|^4),
\end{aligned}
\end{equation}
where $\chi$ is a constant determined by the original system parameters. The remainder terms vanish smoothly as \(|(x, y)| \to 0\).
\end{thm}

The proof of the upper Theorem can be found in
\ref{proof-thm:cusp-equivalence}.
Having established the normal form equivalence, we now characterize the singularity at the equilibrium.
\begin{thm}
The nilpotent equilibrium $(x^*, y^*) = ( am + \sqrt{a^2 m^2 + 3m},\linebreak \dfrac{1}{n}(am + \sqrt{a^2 m^2 + 3m}))$ of system (2.4) is a cusp singularity of order 3 when \((c, n, b) = (c^*, n^*, b^*)\).
\end{thm}
\begin{proof}
When $(c, n, b) = (c^*, n^*, b^*)$, we have demonstrated that system \eqref{eq:3.31} is $C^\infty$-equivalent to system \eqref{3.32}. The coefficient $\chi$ is explicitly calculated by substituting the expressions for $\xi_5, \xi_6, \xi_7, \xi_8$ into the formula $\chi = c_5 - c_2c_3$, resulting in:
\begin{align*}
\chi = -\frac{1}{n^5 \xi_3^4} \Big( 3n^2 \xi_5^2 + \xi_3^2(8n \xi_5 + 6\xi_6) + n \xi_3(4n \xi_7 + 3\xi_8) + 5n \xi_5 \xi_6 + 4 \xi_3^4 + 2 \xi_6^2 \Big).
\end{align*}
To analyze the sign of $\chi$, we observe that all parameters $\xi_3, \xi_5, \xi_6, \xi_7, \xi_8$ are implicitly functions of $x^*$ through their definitions. By substituting these parameter expressions into $\chi$, we explicitly express it as a rational function of $x^*$:
\[
\chi = \dfrac{(2 a x^* + 3)^{10}}{8 (x^*)^{10} (a x^* + 3)^7 \Psi^2(x^*)} \cdot \hbar(x^*),
\]
where
\[
\Psi(x^*) = 16 a^3 (x^*)^3 + 4 a^2(x^{*2} +21)(x^*)^2 + 24 a (x^{*2} +6)x^* + 9(4x^{*2} +9)
\]
and
\begin{align*}
\hbar(x^*)=&32 (\vartheta +3)^6(x^*)^4+16 (\vartheta +3)^3 (2 \vartheta+3)^2(2 \vartheta^2+10 \vartheta+15)(x^*)^2+(2 \vartheta+3)^4\\
       &\quad\times (2 \vartheta^3+21 \vartheta^2+108 \vartheta+162) \quad (\vartheta = ax^*)
\end{align*}
 The positivity of $\hbar(x^*)$ follows from:
\begin{align*}
&\widetilde{A} = 32 (\vartheta+3)^6 > 0,\\
& \widetilde{B} = 16 (\vartheta +3)^3 (2\vartheta+3)^2(2\vartheta^2 +10\vartheta+15) > 0,\\
&\widetilde{C} = (2\vartheta+3)^4(2\vartheta^3 +21\vartheta^2 +108\vartheta+162) > 0,\\
&  (\widetilde{B})^2 -4\widetilde{A}\widetilde{C} > 0\ \text{for} \ -\dfrac{3}{2} < \vartheta < 0.
\end{align*}
Since $\hbar(x^*) > 0$ and all denominators are positive, we conclude $\chi \neq 0$. By the characterization of cusp singularities, the equilibrium is a cusp of order 3.
\end{proof}
\begin{rem}
        In Fig.~\ref{fig:7}~(a), we have shown a codimension-2 cusp with parameters $(a, b, c, m, n) = (-1.5, 1.8045924, 0.330275, 0.05, 0.172824)$ and equilibrium $(x^{*}, y^{*}) = (0.319493, 1.848663)$ with $\zeta\eta\neq0$.  By adjusting parameters to $(a, b, c, m, n) = (-1.8, 1, \linebreak 0.330275, 0.064380, 0.172824)$ to enforce the degeneracy condition $\eta=0$, a codimension-3 cusp (Fig.\ref{fig:7}~(b)), emerges at $(x^{*}, y^{*}) = (0.338614, 1.959299)$, demonstrating how higher-codimension singularities organize from the structure of lower-codimension ones. The existence of a codimension-3 cusp,confirms the theoretical prediction that the cubic incidence term can introduce a third independent degeneracy into the Bogdanov-Takens bifurcation. This is a non-trivial result, as many systems cannot support bifurcations beyond codimension-2. The visual similarity yet topological distinction between the codimension-2 and codimension-3 phase portraits highlights the subtlety of high-codimension dynamics.
\end{rem}
We consider \((c, n, b)\) as bifurcation parameters and construct a versal unfolding for the cusp singularity of codimension 3 near the reference point \((c^*, n^*, b^*)\). Introducing small perturbations
\[
c = c^* + \epsilon_1, \quad n = n^* + \epsilon_2, \quad b = b^* + \epsilon_3,
\]
the original system \eqref{eq:2.2} transforms into the following perturbed system:
\begin{equation}\label{4.12}
\begin{cases}
\displaystyle \frac{{\rm d}x}{{\rm d}t} = \dfrac{x^3}{1 + ax + (b^*+\epsilon_3)x^3}\Big(1 - (c^*+\epsilon_1)x - (c^*+\epsilon_1)y\Big) - mx,
\\[12pt]
\displaystyle \frac{{\rm d}y}{{\rm d}t} = x - (n^*+\epsilon_2)y.
\end{cases}
\end{equation}

To analyze the structural stability of this perturbed system near the cusp singularity, we establish its equivalence to a universal normal form through smooth coordinate transformations. The following theorem formalizes this reduction process and verifies the versality of the unfolding. For improved readability, the full proof appears in the
\ref{Proof-thm:cusp_versal}.

\begin{thm}\label{thm:cusp_versal}
For sufficiently small \(\epsilon = (\epsilon_1, \epsilon_2, \epsilon_3)\) and under the non-degeneracy condition
\[
\begin{aligned}
12(\vartheta+1)(2\vartheta+3)^4 &+ 3\vartheta D^4 - 3\vartheta D^3+ (2\vartheta+3)^2(32\vartheta+69)D^2 \\
&- (2\vartheta+3)^2(82\vartheta^2+240\vartheta+171)D \neq 0,
\end{aligned}
\]
system \eqref{4.12} is \(C^\infty\)-equivalent to the normal form
\begin{align}\label{4.13}
\begin{cases}
\dot{x} = y, \\[2pt]
\dot{y} = \mu_1(\epsilon) + \mu_2(\epsilon)y + \mu_3(\epsilon)xy + x^2 + x^3y + R(x,y,\epsilon),
\end{cases}
\end{align}
where the remainder term \(R(x,y,\epsilon)\) satisfies
\[
R(x,y,\epsilon) = y^2\mathcal{O}(|x,y|^2) + \mathcal{O}(|x,y|^5) + \mathcal{O}(\epsilon)\big(\mathcal{O}(y^2) + \mathcal{O}(|x,y|^3)\big) + \mathcal{O}(\epsilon^2)\mathcal{O}(|x,y|),
\]
and is jointly \(C^\infty\)-smooth in \((x,y,\epsilon)\). Furthermore, the parameters \(\mu = (\mu_1, \mu_2, \mu_3)\) satisfy the transversality condition
\begin{align}\label{4.14}
\det\left(\frac{\partial (\mu_1, \mu_2, \mu_3)}{\partial (\epsilon_1, \epsilon_2, \epsilon_3)}\right)\bigg|_{\epsilon=0} \neq 0.
\end{align}
This establishes system \eqref{4.12} as a versal unfolding of the cusp singularity of codimension 3.
\end{thm}
 \begin{rem}\label{rem:SN_manifold_codim}
    The geometry of the saddle-node bifurcation set exemplifies a fundamental principle of unfolding theory. In the codimension-2 unfolding~\eqref{eq:4.4}, it is the line $\mu_1 = 0$ in the $(\mu_1, \mu_2)$-plane. In the higher-dimensional parameter space of the codimension-3 unfolding~\eqref{4.13}, this set becomes the plane $\mu_1 = 0$ in the $(\mu_1, \mu_2, \mu_3)$-space. This illustrates how increasing the codimension elevates bifurcation sets to higher-dimensional strata: the parameter $\mu_1$ controls the splitting of the degenerate singularity, while $\mu_2$ and $\mu_3$ parametrize directions tangent to the bifurcation set itself. In the bifurcation diagram (Figure 8), this plane is represented by the yellow surfaces ($SN^{+}$ and $SN^{-}$).
\end{rem}

Following Chow's theoretical framework \cite{Chow1994}, the Hopf bifurcation surface for system \eqref{4.13} is characterized by:
\begin{align}
\mu_2 &= \mu_3(-\mu_1)^{1/2} + (-\mu_1)^{3/2} + \mathcal{O}\left((-\mu_1)^{7/4}\right).
\end{align}
The homoclinic bifurcation surface satisfies:
\begin{align}
\mu_2 &= \frac{5}{7}\mu_3(-\mu_1)^{1/2} + \frac{103}{77}(-\mu_1)^{3/2} + \mathcal{O}\left((-\mu_1)^{7/4}\right).
\end{align}
While the saddle-node bifurcation surface involves solutions to a Riccati equation, the saddle-node bifurcation of limit cycles exhibits the following geometric relationships:
\begin{itemize}
\item[(1)] Tangency to the homoclinic bifurcation surface along the curve
\begin{align}
\left\{-u^2,\ 4u^3,\ 3u\right\},
\end{align}

\item[(2)] Tangency to the Hopf bifurcation surface along the curve
\begin{align}
\left\{-u^2,\ -\frac{4}{11}u^3,\ -\frac{15}{11}u\right\}.
\end{align}
\end{itemize}

Using Hermite interpolation, we derive the saddle-node bifurcation surface equation:
\begin{align}
\left\{-u^2, \,
    \begin{aligned}[t]
        &\frac{1}{11} u \biggl( u^2 (74 v^3 - 111 v^2 + 33) \\
        &\quad + 3025 (v - 1) v^2 \biggr),
    \end{aligned}
    \quad
    \begin{aligned}[t]
        &\frac{1}{55} v^2 \biggl( 241 u^2 (2v - 3) \\
        &\quad + 21175 (v - 1) \biggr) + 2u^2
    \end{aligned}
\right\}.
\end{align}

Figure~\ref{fig:81} presents the complete bifurcation diagram generated via Mathematica, with notation definitions provided in Table~\ref{tab:bifur_notations}.

\begin{figure}[htbp]
\centering
\includegraphics[width=0.85\textwidth]{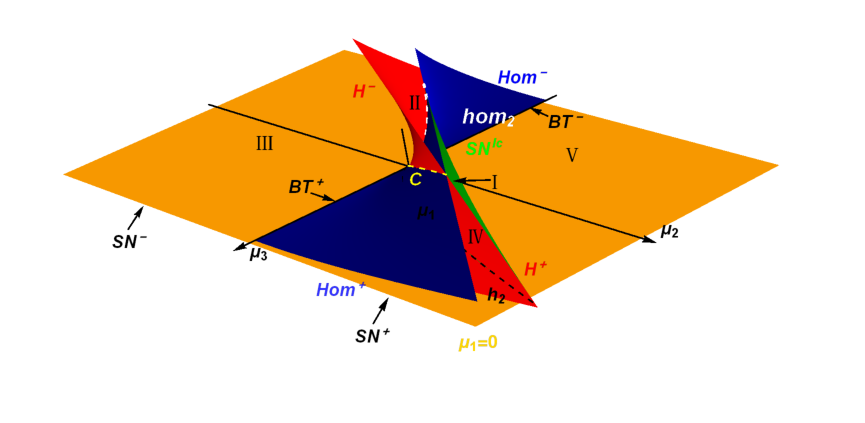}
\caption{Codimension-3 cusp singularity bifurcation diagram.}
\label{fig:81}
\end{figure}

\begin{table}[htbp]
  \centering
  \caption{Bifurcation Notations with Color Coding for Figure~\ref{fig:81}.}
  \label{tab:bifur_notations}
  \begin{tabular}{@{}lll@{}}
    \toprule
    \textbf{Symbol} & \textbf{Description} & \textbf{Color} \\
    \midrule
    $H^+$           & Repelling Hopf bifurcation       & Red \\
    $H^-$           & Attracting Hopf bifurcation      & Red \\
    $Hom^+$         & Repelling homoclinic bifurcation & Blue \\
    $Hom^-$         & Attracting homoclinic bifurcation & Blue \\
    $SN^{\ell c}$   & Saddle-node bifurcation of limit cycles & Yellow \\
    $h_2$           & Codimension-two Hopf bifurcation & Black \\
    $hom_2$         & Codimension-two homoclinic bifurcation & White \\
    $C$             & Intersection curve of $H$ and $Hom$ & Yellow \\
    $BT^+$          & Bogdanov-Takens (positive $xy$)  & Black \\
    $BT^-$          & Bogdanov-Takens (negative $xy$)  & Black \\
    $SN^+$          & Repelling saddle-node            & Yellow \\
    $SN^-$          & Attracting saddle-node           & Yellow \\
    \bottomrule
  \end{tabular}
\end{table}

\begin{rem}
The bifurcation diagram in Figure~\ref{fig:81} stratifies the parameter space of system \eqref{4.13} into six generic dynamical regimes of codimension-0, whose phase portraits are systematically cataloged in Table~\ref{tab:phase3}. The full stratification includes codimension-1 bifurcation surfaces and codimension-2 singularities that bound or intersect these regimes, described as follows:
\begin{itemize}
\item[(1)] \textbf{Region I}: Bounded by $\mathcal{H}=H^+\cup H^-$, $Hom=Hom^+\cup Hom^-$, and $SN^{\ell c}$. Contains two hyperbolic limit cycles (the inner one is stable and the outer one is unstable).

\item[(2)] \textbf{Region II}: Between supercritical Hopf ($H^+$) and repelling homoclinic ($Hom^+$) surfaces. Sustains a unique unstable limit cycle.

\item[(3)] \textbf{Region III}: Adjacent to $SN^+$. Exhibits no limit cycles; phase space contains saddle and anti-saddle equilibria.

\item[(4)] \textbf{Region IV}: Between subcritical Hopf ($H^-$) and attracting homoclinic ($Hom^-$) surfaces. Contains a unique stable limit cycle.

\item[(5)] \textbf{Region V}: Adjacent to $SN^-$. Similar to Region III but with different equilibrium stability.

\item[(6)] \textbf{Plane $\mu_1 = 0$ ($SN^+\cup SN^-$)}: Saddle-node bifurcation surface. For $\mu_1 >0$, system \eqref{4.13} lacks equilibria; all bifurcations occur in $\mu_1 <0$.

\item[(7)] \textbf{Surface $SN^{\ell c}$}: Contains semi-stable limit cycles formed through limit cycle coalescence.

\item[(8)] \textbf{Surface $\mathcal{H}=H^+\cup H^-$}: Generates codimension-1 Hopf bifurcations, excluding $h_2$.

\item[(9)] \textbf{Surface $Hom=Hom^+\cup Hom^-$}: Generates codimension-1 homoclinic bifurcations, excluding $hom_2$.

\item[(10)] \textbf{Curve $C$}: Intersection of $\mathcal{H}$ and $Hom$ surfaces where codimension-1 Hopf and homoclinic bifurcations coexist.

\item[(11)] \textbf{Curve $h_2$}: Codimension-2 Hopf bifurcation locus; $SN^{\ell c}$ tangency point to $\mathcal{H}$.

\item[(12)] \textbf{Curve $hom_2$}: Codimension-2 homoclinic bifurcation locus; $SN^{\ell c}$ tangency point to $Hom$.
\end{itemize}
\end{rem}

\begin{table}[htbp]
  \centering
  \caption{The codimension 0 phase portraits of equation \eqref{4.13}.}
  \label{tab:phase3}
  \begin{tabular}{|l|l|l|}
    \hline
    \includegraphics[width=4cm]{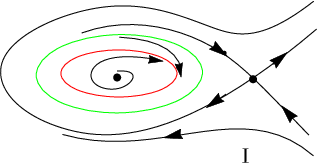} &
    \includegraphics[width=4cm]{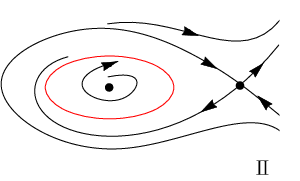} &
    \includegraphics[width=4cm]{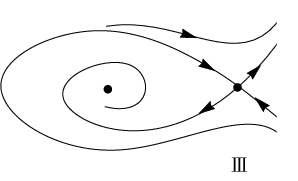} \\
    \hline
    \includegraphics[width=4cm]{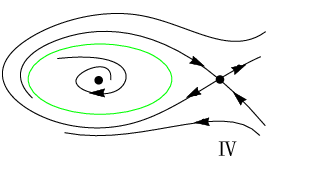} &
    \includegraphics[width=4cm]{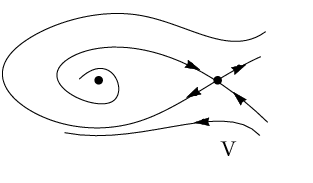} &
    \includegraphics[height=3cm]{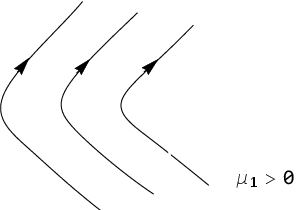} \\
    \hline
  \end{tabular}
\end{table}

\begin{rem}\label{rem:epsilon_diagram}
The bifurcation diagram in the original perturbation coordinates \((\epsilon_1, \epsilon_2, \epsilon_3)\) is topologically equivalent to the universal diagram in \((\mu_1, \mu_2, \mu_3)\)-space under the local diffeomorphism guaranteed by the transversality condition \(\det\left(\frac{\partial \mu}{\partial \epsilon}\right)\big|_{\epsilon=0} \neq 0\) (Theorem~\ref{thm:cusp_versal}). Geometrically, the \(\epsilon\)-diagram arises from a smooth reparameterization of the \(\mu\)-diagram via the Jacobian matrix \(\left(\frac{\partial \mu}{\partial \epsilon}\right)\big|_{\epsilon=0}\), preserving the topological structure of critical manifolds (e.g., cusp surfaces, hysteresis branches) while introducing affine distortions. The orientation of the diagram depends on \(\det\left(\frac{\partial \mu}{\partial \epsilon}\right)\big|_{\epsilon=0}\): if \(\det\left(\frac{\partial \mu}{\partial \epsilon}\right)\big|_{\epsilon=0}> 0\), the \(\epsilon\)-coordinates retain the orientation of stability domains and bifurcation branches in \(\mu\)-space; if \(\det\left(\frac{\partial \mu}{\partial \epsilon}\right)\big|_{\epsilon=0} < 0\), the diagram becomes a mirror image of the \(\mu\)-system, reversing orientation-dependent features such as stability switching directions. In both cases, the universal unfolding property ensures that all codimension-3 bifurcation modes remain transverse to the cusp singularity, and the \(\epsilon\)-parameterized diagram inherits the versality of \eqref{4.12} under \(C^\infty\)-equivalence.
\end{rem}

\subsection{$\mathcal{K}$-Equivalence and the Codimension Discrepancy}

The codimension-3 designation for the Bogdanov-Takens bifurcation in Subsection 4.2, derived from a dynamical systems perspective, contrasts with the result obtained by applying the formalisms of singularity theory. This divergence arises from the different notions of equivalence employed in the two fields. Here, we employ \textit{contact equivalence} ($\mathcal{K}$-equivalence) to rigorously compute the codimension of the singularity within its own framework and resolve this apparent discrepancy.

The foundational definitions for $\mathcal{K}$-equivalence, the extended tangent space $T_e\mathcal{K}\cdot f$, and the contact codimension $\operatorname{codim}(f, \mathcal{K})$ are provided in Definitions~\ref{def:contact_equiv}--\ref{def:contact_codim} of \ref{app:components}. Furthermore, Theorem~\ref{thm:k_determined} establishes that a finite $\mathcal{K}$-codimension implies finite determinacy, which justifies our analysis of the normal form.

Since the transformations \eqref{eq:2.39}, \eqref{eq:2.41}, and \eqref{eq:4.1} are diffeomorphisms, the maps involved in systems \eqref{eq:2.38}, \eqref{eq:2.40}, \eqref{eq:2.42}, and \eqref{eq:4.2} are $\mathcal{K}$-equivalent (Definition~\ref{def:contact_equiv}) and consequently share the same codimension (Proposition~\ref{prop:same_codim}). We therefore compute the codimension of the simplest representative, system (4.2), with map germ $\psi(x,y)=(y,x^2+xy)$.

The extended tangent space $T_e\mathcal{K}\cdot\psi$ (Definition~\ref{def:extended_tangent_space}) is calculated as follows. The tangent map $\mathrm{t}\psi$ (Definition~\ref{def:tangent_map}) yields generators, and the ideal $I_\psi = \langle y, x^2+xy \rangle$ contributes additional terms:
\begin{equation}
\begin{split}
T_e\mathcal{K}\cdot\psi &= {\rm t}\psi(\theta_2) + I_\psi\theta(\psi) \\
&= \left\langle \begin{pmatrix} y \\ 0 \end{pmatrix}, \begin{pmatrix} x^2 \\ 0 \end{pmatrix}, \begin{pmatrix} 0 \\ y \end{pmatrix}, \begin{pmatrix} 0 \\ x^2 \end{pmatrix}, \begin{pmatrix} 0 \\ 2x+y \end{pmatrix}, \begin{pmatrix} 1 \\ x \end{pmatrix} \right\rangle \\
&= \left\langle \begin{pmatrix} y \\ 0 \end{pmatrix}, \begin{pmatrix} 0 \\ y \end{pmatrix}, \begin{pmatrix} 0 \\ x \end{pmatrix}, \begin{pmatrix} 1 \\ 0 \end{pmatrix} \right\rangle.
\end{split}
\end{equation}
A basis for the quotient space $\theta(\psi)/T_e\mathcal{K}\cdot\psi$ is $\left\{ \begin{pmatrix} 0 \\ 1 \end{pmatrix}, \begin{pmatrix} x \\ 0 \end{pmatrix} \right\}$. By Definition~\ref{def:contact_codim}, the contact-codimension is:
\begin{equation}
\operatorname{codim}(\psi,\mathcal{K}) = \dim \left( \theta(\psi) \big/ T_e\mathcal{K}\cdot \psi \right) = 2.
\end{equation}
This confirms that within the framework of $\mathcal{K}$-equivalence, the singularity is of codimension 2. Since this codimension is finite, Theorem~\ref{thm:k_determined} implies the germ $\psi$ is finitely determined.

\begin{rem}
The discrepancy between the singularity-theoretic codimension 2 and the dynamical codimension 3 originates from fundamentally different treatments of nonlinear degeneracies:

\noindent\textbf{Singularity Theory:} Classifies the degeneracy of the equilibrium's zero set. Under $\mathcal{K}$-equivalence, the cubic term $x^3y$ is absorbable ($x^3y \in \mathfrak{m}^4\cdot\theta(\psi)$) and thus does not increase codimension.

\noindent\textbf{Dynamical Systems:} Requires versal unfolding of the full phase portrait. The term $\chi x^3y$ becomes essential as it governs:
\begin{itemize}
    \item[(1)] Separatrix splitting (via Melnikov integral $\mathcal{M} \propto \chi$);
    \item[(2)] Limit cycle multiplicity transitions (1 $\leftrightarrow$ 3);
    \item[(3)] Cusp formation in bifurcation curves.
\end{itemize}
The third parameter $\chi$ emerges naturally to preserve bifurcation diagram versality under nonlinear degeneracies, reflecting the need for global control of dynamical phenomena.
\end{rem}
\subsection{Comparison with Lu et al. (2019) on Bogdanov-Takens Bifurcations}
We provide a comparative analysis between our findings on Bogdanov-Takens (BT) bifurcations and those established by Lu et al. \cite{Lu2019}. The key advancements and inheritances are summarized in Table~\ref{tab:BT_comparison}.
\begin{table}[h]
\centering
\caption{Comparison with Lu et al. (2019) on Bogdanov-Takens Bifurcations.}
\label{tab:BT_comparison}
\renewcommand{\arraystretch}{1.3} 
\begin{tabular}{|p{0.3\textwidth}|p{0.25\textwidth}|p{0.35\textwidth}|}
\hline
\textbf{Aspect} & \textbf{Lu et al. (2019)} & \textbf{This work} \\
\hline
Incidence function & Quadratic saturation & Cubic saturation \\
\hline
Max codimension of BT & 2 & 3 \\
\hline
Singularity type & Cusp of order 2 & Cusp of order 3 \\
\hline
Unfolding theory & Not applied & Versal unfolding applied \\
\hline
Number of limit cycles & Up to 2 & Up to 3 \\
\hline
Normal form derivation & Standard BT form & Higher-order degenerate BT \\
\hline
Numerical verification & Limited to codim-2 & Extended to codim-3 \\
\hline
\end{tabular}
\end{table}

As shown in Table~\ref{tab:BT_comparison}, while Lu et al. established the foundation for BT bifurcations in SIRS models with quadratic incidence, this work extends the analysis to cubic incidence, achieving higher codimension and more complex dynamics. The use of singularity theory and versal unfolding provides a more systematic framework for analyzing degenerate bifurcations.
\section{Multiple periodic orbits}
While the previous section focused on bifurcations of equilibria (Bogdanov-Takens), we now turn to the analysis of periodic solutions. This section is dedicated to the study of Hopf bifurcations and their degeneracies in system (2.4). We establish conditions for the existence of weak foci of higher order and prove the emergence of multiple limit cycles through degenerate Hopf bifurcations, a phenomenon rigorously confirmed by both analytical methods and numerical simulations.
\subsection{Characterization of generic Hopf Bifurcation}

As established in Section 2, any oscillatory solution must encircle the biologically meaningful equilibrium $E_2(x_2, y_2)$. The stability transition of this equilibrium occurs when the trace of its Jacobian vanishes, indicating a potential Hopf bifurcation. Rewriting system \eqref{eq:2.4} as
\begin{equation}
    \begin{aligned}
        \dfrac{{\rm d}x}{{\rm d}t} & = p(x)(G(x) - y), \\
        \dfrac{{\rm d}y}{{\rm d}t} & = x - n y,
    \end{aligned}
    \label{eq:5.1}
\end{equation}
where
\begin{equation}
    \begin{aligned}
        G(x)&:= \dfrac{1}{c} - x - \dfrac{m(1 + ax + bx^3)}{cx^2}, \\
        p(x)&:= \dfrac{cx^3}{1 + ax + bx^3}.
    \end{aligned}
    \label{eq:5.2}
\end{equation}
The eigenvalues of $J(x_2,y_2)$ become purely imaginary when the trace condition is satisfied:
\begin{equation}
    \operatorname{Tr}(J) = p(x_2)G'(x_2) - n = 0.
    \label{eq:5.3}
\end{equation}
Building on the classification framework in Theorem \ref{th:3.6} for $\mathfrak{D}_{(p,q)} < 0$, we designate $b$ as the bifurcation parameter. The positive equilibrium $(x_2,y_2)$, which depends on $b$, satisfies
\begin{equation*}
    \big(c(n+1) + bmn\big)x_2^3 - nx_2^2 + amn x_2 + nm = 0.
\end{equation*}
Consequently, we obtain
\begin{equation*}
    \frac{\mathrm{d} x_2}{\mathrm{d} b} = \frac{-m x_2^4}{x_2^2 - 2 a m x_2 - 3 m}.
\end{equation*}

\begin{thm}\label{th:5.1}
    Under the discriminant condition $\mathfrak{D}_{(p,q)} < 0$, the system \eqref{eq:2.4} undergoes a generic Hopf bifurcation at $E_2(x_2,y_2)$ if and only if either:
    \begin{align*}
        G'(x_2) = \frac{n}{p(x_2)} \quad \text{and} \quad G''(x_2) \neq -\frac{n p'(x_2)}{p^2(x_2)},
    \end{align*}
    or:
    \begin{align*}
        2m - n + a(m - n)x_2 - (c + bm + bn)x_2^3 = 0 \quad \text{and} \quad n(2a x_2 + 3) - 2m(a x_2 + 3) \neq 0.
    \end{align*}
\end{thm}

\begin{proof}
    The condition $\operatorname{Tr}(J(x_2,y_2)) = 0$ is equivalent to $G'(x_2) = \frac{n}{p(x_2)}$. Under this condition, the Jacobian has purely imaginary eigenvalues $\lambda(b)$ and $\overline{\lambda}(b)$ satisfying:
    \begin{align*}
        \operatorname{Re}(\lambda(b)) = \operatorname{Re}(\overline{\lambda}(b)) = \frac{1}{2} \operatorname{Tr}(J(x_2,y_2)) = 0.
    \end{align*}
    Using the relation $p(x_2)G'(x_2) - n = 0$, we compute the transversality condition:
    \begin{align*}
        \frac{\mathrm{d}}{\mathrm{d}b} \operatorname{Re}(\lambda(b))
        &= \frac{\partial}{\partial x_2} \operatorname{Re}(\lambda(b)) \cdot \frac{\mathrm{d} x_2}{\mathrm{d} b} \\
        &= \frac{1}{2} \left( p'(x_2)G'(x_2) + p(x_2)G''(x_2) \right) \frac{-m x_2^4}{x_2^2 - 2 a m x_2 - 3 m} \\
        &= \frac{1}{2} \left( \frac{n p'(x_2)}{p(x_2)} + p(x_2)G''(x_2) \right) \frac{-m x_2^4}{x_2^2 - 2 a m x_2 - 3 m}.
    \end{align*}
    For transversality, the following nondegeneracy condition must hold:
    \begin{align*}
        \frac{n p'(x_2)}{p(x_2)} + p(x_2)G''(x_2) = \dfrac{n(2a x_2 + 3) - 2m(a x_2 + 3)}{x_2(1 + a x_2 + b x_2^3)} \neq 0.
    \end{align*}
    Direct calculation yields:
    \begin{align*}
        G'(x_2) - \frac{n}{p(x_2)} = \dfrac{2m - n + a(m - n)x_2 - (c + bm + bn)x_2^3}{1 + a x_2 + b x_2^3}.
    \end{align*}
    This completes the proof.
\end{proof}

\subsection{Higher-Order Weak Foci and Stability}

The theory of degenerate Hopf bifurcations originates in Hopf's seminal work \cite{e.hofp}, establishing fundamental conditions for periodic solution emergence from equilibria. Bautin \cite{NN} advanced this field through analysis of the vanishing first Lyapunov coefficient ($\sigma_1=0$), revealing multiple limit cycle bifurcations. Arnold \cite{Arnold1988} formalized the framework via singularity theory and universal unfoldings, while Guckenheimer \& Holmes \cite{Guckenheimer1983} systematized the dynamical systems approach.

Higher-codimension degenerate Hopf bifurcations (codimension >2) present significant analytical challenges. Classical techniques-notably Lyapunov coefficient computation \cite{Kuznetsov2023,2007, Chow1982} and singularity theory \cite{Golubitsky1981} provide foundations for codimension-two cases but encounter prohibitive algebraic complexity when extended to codimension-three scenarios. Existing methods struggle with intricate focus quantities where sign determination becomes intractable under higher degeneracies, necessitating structure-specific approaches.

Building on  Han's algebraic framework for Hopf cyclicity in planar systems \cite{Han2001}, we transform our system to the generalized Lienard form:
\begin{equation}\label{eq:5.4}
\begin{aligned}
    \dot{x} &= \psi(y) - F(x) \\
    \dot{y} &= -g(x).
\end{aligned}
\end{equation}
Central to our analysis is Han's classical lemma on weak focus characterization \cite{Han2001}. This result was later reformulated by \cite{Arsie2022} with equivalent content but a slightly modified presentation. We adopt the formulation from \cite{Arsie2022}:

\begin{lem}[\cite{Han2001,Arsie2022}] \label{lem:weak_focus}
Let $\psi$, $F$, and $g$ be $C^\infty$-functions defined in a neighborhood of the origin satisfying:
\begin{align*}
\psi(0)=g(0)=F(0)=0,\psi'(0)>0 \ \text{and} \ g'(0)>0.
\end{align*}
Define $H(x) = \int_{0}^{x} g(s)  {\rm d}s$. Suppose there exists a $C^\infty$-function $\theta(x) = -x + \mathcal{O}(x^2)$ such that:
\begin{align*}
H(\theta(x)) &\equiv H(x); \\
F(\theta(x)) - F(x) &= \sum_{i\geq 1} B_i x^i.
\end{align*}
Then the origin is a focus of order $k$ if:
\begin{equation*}
B_i = 0 \quad \text{for} \quad i = 1, 2, \dots, 2k \quad \text{and} \quad B_{2k+1} \neq 0
\end{equation*}
with stability determined by the sign of $B_{2k+1}$:
\begin{equation*}
\begin{cases}
\text{stable} & \text{if } B_{2k+1} < 0 \\
\text{unstable} & \text{if } B_{2k+1} > 0.
\end{cases}
\end{equation*}
\end{lem}
Through strategic control of polynomial coefficients, we simultaneously satisfy three independent degeneracy conditions via Lemma \ref{lem:weak_focus}. This approach enables explicit sign determination for focus quantities and simultaneous resolution of multiple degeneracy constraints. Consequently, we rigorously confirm three coexisting limit cycles near the origin. The bifurcation structure emerges from synergistic asymptotic and algebro-geometric arguments applied to controlled degeneracy parameters.

One  translates the positive equilibrium $(x_2,y_2)$ to the origin by the transform  $ x\rightarrow x+x_2, y\rightarrow y+y_2$,  and substituting $y_2=\frac{1}{n}x_2$ to the above system yields
\begin{equation}
\label{eq:5.5}
\begin{aligned}
\dfrac{{\rm d} x}{{\rm d}t}&=p(x+x_2)(G(x+x_2)-y_2-y),\\
\dfrac{{\rm d} y}{{\rm d}t}&=x+x_2-n(y+y_2),\\
\end{aligned}
\end{equation}
where
\begin{equation}
\label{eq:5.6}
\begin{aligned}
G(x+x_2)&=\Big(\dfrac{1}{c}-(x+x_2)-\dfrac{m(1+a(x+x_2)+b(x+x_2)^3)}{c(x+x_2)^2}\Big), \\
p(x+x_2)&=\dfrac{c(x+x_2)^3}{1+a(x+x_2)+b(x+x_2)^3}.\\
\end{aligned}\end{equation}
It is clear that there exists a neighborhood $U_1$ with radius less then $x_2$ such that $p(x+x_2)\neq0$. Making the time scaling $t\rightarrow \dfrac{1}{p(x+x_2)}t$ for system in neighborhood $U_1$, thus the system \eqref{eq:5.5} are transformed into the  following system
\begin{equation}
\label{eq:5.7}
\begin{aligned}
\dfrac{{\rm d} x}{{\rm d}t}&=G(x+x_2)-y_2-y,\\
\dfrac{{\rm d} y}{{\rm d}t}&=\dfrac{1}{p(x+x_2)}(x+x_2)+\dfrac{-n}{p(x+x_2)}(y+y_2),\\
\end{aligned}
\end{equation}
the transformation  $x\rightarrow x,G(x+x_2)-y_2-y\rightarrow y$ will bring this into the standard Lienard form
\begin{equation}\label{eq:5.8}
\begin{aligned}
\dfrac{{\rm d} x}{{\rm d}t}&=y,\\
\dfrac{{\rm d} y}{{\rm d}t}&=-g(x)+f(x)y,\\
\end{aligned}
\end{equation}
with
\begin{align*}
g(x)&=\frac{1}{p(x+x_2)}\Big((x+x_2)-nG(x+x_2)\Big),\\
 f(x)&=G'(x+x_2)-\frac{n}{p(x+x_2)}.
\end{align*}
By the transform $x\rightarrow x, y\rightarrow y-F(x)$, where \begin{align*}
F(x)&=-\int_{0}^{x}f(s){\rm d}s=-\int_{0}^{x}(G'(s+x_2)-\frac{n}{p(s+x_2)}){\rm d}s\\
&=-G(x+x_2)+G(x_2)+\int_{0}^{x}\frac{n}{p(s+x_2)}{\rm d}s,
\end{align*}
we change system \eqref{eq:5.8} to the following system
\begin{equation}
\label{eq:5.9}
\begin{aligned}
\dfrac{{\rm d}x}{{\rm d}t}&=y-F(x),\\
\dfrac{{\rm d} y}{{\rm d}t}&= -g(x),\\
\end{aligned}
\end{equation}
which is a generalized Lienard system in terms of $\psi(y)=y$. It is easy to verify that
\begin{equation*}\begin{aligned}
\psi(0)&=0,\quad F(0)=0,\quad g(0)=\dfrac{n}{p(x_2)}\Big(\dfrac{1}{n}x_2-G(x_2)\Big)=0,\\
\psi'(0)&=1,\quad g'(0)=\dfrac{n((x_2)^2-2am(x_2)-3m)}{1+a(x_2)+b(x_2)^3}>0.
\end{aligned}
\end{equation*}
Let
\begin{align*}
H(x)=\int_0^xg(s){\rm d}s=\int_0^x\bigg(\frac{n}{p(s+x_2)}\bigg(\frac{1}{n}(s+x_2)-G(s+x_2)\bigg)\bigg){\rm d}s.
\end{align*}
Noticing $H(0)=H'(0)=0$, the Taylor's series of $H(x)$ at $x=0$ has the form
\begin{align*}
\sum_{n=2}^7\frac{H^{(n)}(0)}{n!}x^n+\mathcal{O}(x^8),
\end{align*}
where
\begin{align*}
H''(0)&=\frac{1-n G'(x_2)}{p(x_2)}:=h_2,\\
H^{(3)}(0)&=\frac{2 p'(x_2) (n G'(x_2)-1)}{p^2(x_2)}-\frac{n G''(x_2)}{p(x_2)}:=h_3,\\
H^{(4)}(0)&=-\frac{n G^{(3)}(x_2)}{p(x_2)}+\frac{3 n G''(x_2) p'(x_2)}{p^2(x_2)}+\Big(\frac{1}{n}-G'(x_2)\Big)\Big(\frac{6 n p'^2(x_2)}{p^3(x_2)}-\frac{3 n p''(x_2)}{p^2(x_2)}\Big):=h_4,\\
H^{(5)}(0)&=-\frac{n G^{(4)}(x_2)}{p(x_2)}+\frac{4 n G^{(3)}(x_2) p'(x_2)}{p^2(x_2)}+\frac{6 n G''(x_2) p''(x_2)}{p^2(x_2)}-\frac{12 n G''(x_2) p'^2(x_2)}{p^3(x_2)}\\
&\quad+\Big(\dfrac{1}{n}-G'(x_2)\Big)\Big(-\frac{4 n p^{(3)}(x_2)}{p^2(x_2)}-\frac{24 n p'^3(x_2)}{p(^4 x_2)}+\frac{24 n p'(x_2) p''(x_2)}{p^3(x_2)}\Big):=h_5,\\
H^{(6)}(0)&=-\frac{n G^{(5)}(x_2)}{p(x_2)}+\frac{5 n G^{(4)}(x_2) p'(x_2)}{p^2(x_2)}+G^{(3)}(x_2)\Big(\frac{10 n p''(x_2)}{p^2(x_2)}-\frac{20 n p'^2(x_2)}{p^3(x_2)}\Big)\\
&\quad +G''(x_2)(\frac{10 n p^{(3)}(x_2)}{p^2(x_2)}+\frac{60 n p'^3(x_2)}{p^4(x_2)}-\frac{60 n p'(x_2) p''(x_2)}{p^3(x_2)}) +\Big(\frac{1}{n}-G'(x_2)\Big)\\
&\quad\times\Big(-\frac{5 n p^{(4)}(x_2)}{p^2(x_2)}+\frac{30 n p''^2(x_2)}{p^3(x_2)}+\frac{120 n p'^4(x_2)}{p^5(x_2)}+\frac{40 n p^{(3)}(x_2) p'(x_2)}{p^3(x_2)}\\
&\quad-\frac{180 n p'^2(x_2) p''(x_2)}{p^4(x_2)}\Big):=h_6,\\
H^{(7)}(0)&=-\frac{n G^{(6)}(x_2)}{p(x_2)}+\frac{6 n G^{(5)}(x_2) p'(x_2)}{p^2(x_2)}+G^{(4)}(x_2)(\frac{15 n p''(x_2)}{p^2(x_2)}-\frac{30 n p'^2(x_2)}{p^3(x_2)})\\
&\quad+G^{(3)}(x_2)(\frac{20 n p^{(3)}(x_2)}{p^2(x_2)}+\frac{120 n p'^3(x_2)}{p^4(x_2)}-\frac{120 n p'(x_2) p''(x_2)}{p^3(x_2)})+G''(x_2)\\
&\quad\times\Big(\frac{15 n p^{(4)}(x_2)}{p^2(x_2)}-\frac{90 n p''^2(x_2)}{p^3(x_2)}-\frac{360 n p'^4(x_2)}{p^5(x_2)}-\frac{120 n p^{(3)}(x_2) p'(x_2)}{p^3(x_2)}\\
&\quad+\frac{540 n p'^2(x_2) p''(x_2)}{p^4(x_2)}\Big)+\Big(\frac{1}{n}-G'(x_2)\Big)(-\frac{6 n p^{(5)}(x_2)}{p^2(x_2)}-\frac{720 n p'^5(x_2)}{p^6(x_2)}\\
&\quad+\frac{60 n p^{(4)}(x_2) p'(x_2)}{p^3(x_2)}+\frac{120 n p^{(3)}(x_2) p''(x_2)}{p^3(x_2)}-\frac{360 n p^{(3)}(x_2) p'^2(x_2)}{p^4(x_2)}\\
&\quad+\frac{1440 n p'^3(x_2) p''(x_2)}{p^5(x_2)}-\frac{540 n p'(x_2) p''^2(x_2)}{p^4(x_2)}):=h_7,
\end{align*}
letting
\begin{align*}
\theta(x)=-x+\nu_2x^2+\nu_3x^3+\nu_4x^4+\nu_5x^5+\nu_6x^6+\mathcal{O}(x^7)
\end{align*}
and $H(\theta(x))\equiv H(x)$, one concludes
\begin{align*}
\nu_2&=-\frac{h_3}{3h_2},\\
\nu_3&=-\nu_2^2,\\
\nu_4&=\frac{h_5+10h_4\nu_2}{-60 h_2}+2\nu_2^3,\\
\nu_5&=\frac{h_4}{2h_2}\nu_2^2+\frac{h_5\nu_2}{20 h_2}-4\nu_2^4,\\
\nu_6&=-\frac{19 h_4 \nu _2^3}{12 h_2}-\frac{11 h_5 \nu _2^2}{60 h_2}+\frac{\left(70 h_4^2-21 h_2 h_6\right) \nu _2}{2520 h_2^2}+\frac{7 h_4 h_5-h_2 h_7}{2520 h_2^2}+9 \nu _2^5,
\end{align*}
where
\begin{align*}
h_2=\frac{n \left(a x_2+b x_2^3+1\right) \left(x_2^2-m \left(2 a x_2+3\right)\right)}{c^2 x_2^6}>0.
\end{align*}

\begin{thm}\label{th:5.3}
The codimension of the Hopf bifurcation is determined by the order $k$ ($k=0,1,2,3$) of the weak focus and summarized as follows:
$$
\begin{array}{c|c|l}
 \hline
 \text{Order } k & \text{Codimension} & \text{Conditions} \\ \hline
 0 & 0& \mathscr{G}'(x_2) \neq 0 \\
  & &  \\
1 & 1 & \mathscr{G}'(x_2) = 0, -3\nu_2\mathscr{G}''(x_2) \neq \mathscr{G}^{(3)}(x_2) \\
 & &  \\
2 & 2 & \mathscr{G}'(x_2) = 0, \ -3\nu_2\mathscr{G}''(x_2)= \mathscr{G}^{(3)}(x_2),\\
 &  &\mathscr{G}''(x_2)( h_5+10 h_4 \nu_2) \neq h_2(\mathscr{G}^{(5)}(x_2)+10 \mathscr{G}^{(4)}(x_2)\nu_2)
   \\
    & &  \\
3 & 3 & \mathscr{G}'(x_2) = 0,\ -3\nu_2\mathscr{G}''(x_2)= \mathscr{G}^{(3)}(x_2),\\
 & & \mathscr{G}''(x_2)( h_5+10 h_4 \nu_2) = h_2(\mathscr{G}^{(5)}(x_2)+10 \mathscr{G}^{(4)}(x_2)\nu_2), \\
 & & \dfrac{630 h_2 h_4 \nu_2^3+(70 h_4^2-21 h_2 h_6) \nu_2+7 h_4 h_5-h_2 h_7}{h_2^2} \mathcal{G}''(x_2) \\
 & & \neq -\mathscr{G}^{(7)}(x_2)-21 \nu_2 \mathscr{G}^{(6)}(x_2)+\dfrac{7\left(90 h_2 \nu_2^3+10 h_4 \nu_2+k_5\right)}{h_2}\mathscr{G}^{(4)}(x_2) \\
\hline
\end{array}
$$
with stability determined by the sign of $B_{2k+1}$:
\begin{equation*}
\begin{cases}
\text{stable} & \text{if } B_{2k+1} < 0; \\
\text{unstable} & \text{if } B_{2k+1} > 0,
\end{cases}
\end{equation*}
where $\mathscr{G}(x)=G(x)+\displaystyle\int_{0}^{x}\dfrac{n}{p(s)}{\rm d}s$.
\end{thm}

\begin{figure}[h]
  \centering
  \begin{minipage}[c]{0.48\textwidth}
    \centering
    \includegraphics[width=\textwidth]{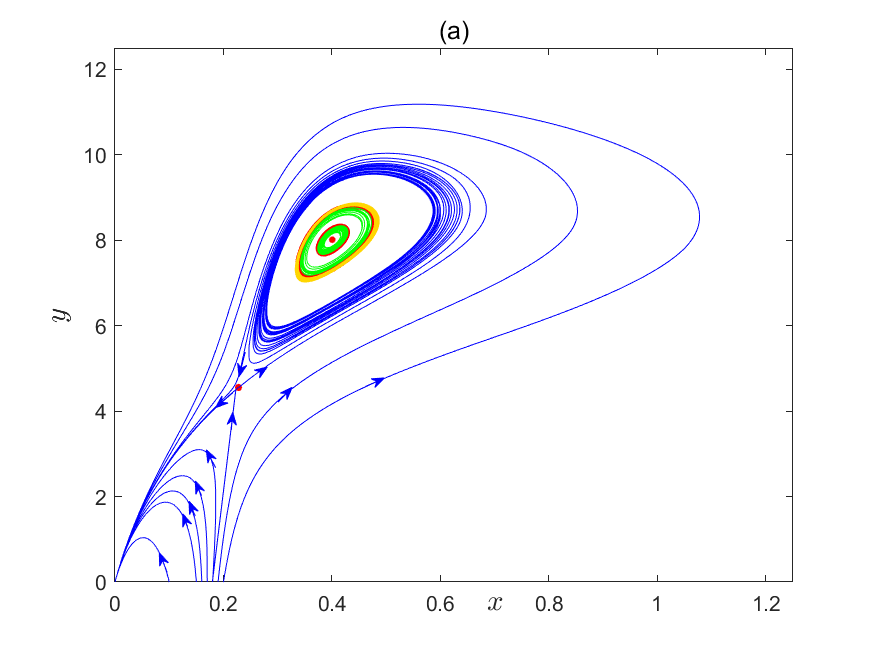}
  \end{minipage}
  \hfill
  \begin{minipage}[c]{0.48\textwidth}
    \centering
    \includegraphics[width=\textwidth]{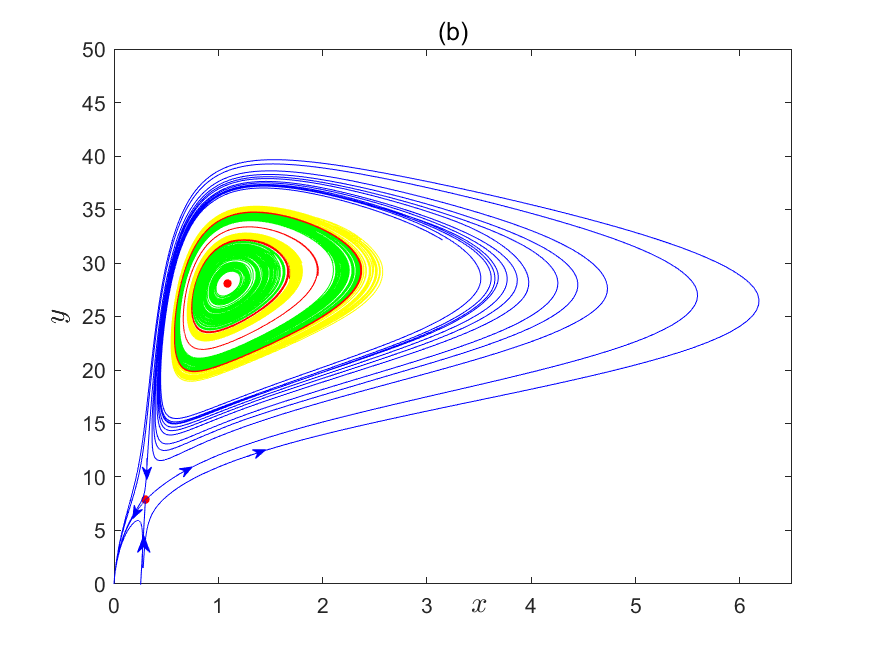}
  \end{minipage}
  \caption{(a) Two limit cycles. (b) Three limit cycles.}
  \label{fig:10}
\end{figure}
The whole proof of Theorem \ref{th:5.3} can be found in \ref{proof_th:5.3}.
Theorem~\ref{th:5.3} establishes that the Hopf bifurcation in system \eqref{eq:2.4} can attain codimension 3 under appropriate parameter conditions. We numerically validate this theoretical prediction by demonstrating both codimension 2 and 3 Hopf bifurcations through carefully selected examples. These bifurcations generate multiple limit cycles when parameters undergo controlled perturbations. Notably, the codimension 2 Hopf bifurcation also arises naturally from a codimension 3 Bogdanov-Takens bifurcation, as guaranteed by Theorem~\ref{thm:cusp_versal}.
\begin{ex}
Consider parameters $c =0.0988432$, $b = 1$, $m =0.0292698$, and define $\lambda_1 = (a,n)$ with $\lambda_{10} = (\widehat{a},\widehat{n}) = (-0.35, 0.05)$. At $\lambda_{10}$, system (2.4) possesses two positive equilibria: $E_1 = (0.22727,4.5454)$ and $E_2 = (0.400544,8.01088)$, with $B_1(\lambda_{10}) = B_3(\lambda_{10}) = 0$ and $B_5(\lambda_{10}) =15668.7> 0$ at $x_2=0.400544$.

Define the mapping
\begin{align*}
\phi_1 : (U \subset \mathbb{R}^2, \lambda_{10}) \rightarrow (V \subset \mathbb{R}^2, (B_1(\lambda_{10}), B_3(\lambda_{10})))
\end{align*}
where the Jacobian determinant satisfies
\[
\det\left(\left. \frac{\partial(B_1, B_3)}{\partial(a, n)}\right) \right|_{\lambda_1 = \lambda_{10}} = 13261 \neq 0.
\]
This establishes $\phi_1$ as a local diffeomorphism with $\lambda_{10}$ as a regular point, confirming system \eqref{eq:2.4} as a universal unfolding of a codimension 2 Hopf bifurcation near $\lambda_{10}$.\par
 The phase portrait (Figure~\ref{fig:10}(a)) reveals striking dynamics:
\begin{itemize}
    \item[(1)] The inner  green trajectory starting at $(0.415,8.1)$ spirals outward;
    \item[(2)] The outer trajectory from $(0.45,8.1)$ expands outward;
    \item[(3)] The outermost yellow trajectory initiating at $(0.476,8.1)$ contracts inward.
\end{itemize}
Between these flows, a key phenomenon emerges: two distinct limit cycles appear around $E_2$ (red curves in Figure~\ref{fig:10}(a)), consisting of a semi-stable inner cycle and a stable outer cycle. This confirms the predicted bifurcation structure under parameter perturbation.
Equilibrium $E_2$ remains unstable, while $E_1$ functions as a saddle-node.
\end{ex}

\begin{ex}\label{ex:codim3}
Set $b=0.02$, $m=0.0391069$ and define $\lambda_2 = (a,c, n)$ with $\lambda_{20} = (\widehat{a}, \widehat{c}, \widehat{n}) = (2.5, 0.0300281, 0.0387063)$. At $\lambda_{20}$, system \eqref{eq:2.4} satisfies $B_1(\lambda_{20}) = B_3(\lambda_{20}) = B_5(\lambda_{20}) = 0$ and $B_7(\lambda_{20}) = 152.691 > 0$, fulfilling the criteria for a codimension 3 bifurcation. The system possesses two positive equilibria: $E_1=(0.300757,7.77023)$ and $E_2=(1.08727,28.0903)$.

Define the mapping
\begin{align*}
\phi_2 : (U \subset \mathbb{R}^3, \lambda_{20}) \rightarrow (V \subset \mathbb{R}^3, (B_1(\lambda_{20}), B_3(\lambda_{20}), B_5(\lambda_{20})))
\end{align*}
with nonvanishing Jacobian determinant
\[
\det\left(\left. \frac{\partial(B_1, B_3, B_5)}{\partial(a, c, n)} \right)\right|_{\lambda_2 = \lambda_{20}}= 18491.108080\neq 0.
\]
 This confirms $\phi_2$ as a local diffeomorphism and $\lambda_{20}$ as a regular point, establishing system \eqref{eq:2.4} as a universal unfolding of the codimension 3 Hopf bifurcation near $\lambda_{20}$.

 The  dynamics (Figure~\ref{fig:10}(b)) exhibit a remarkable structure:
\begin{itemize}
    \item[(1)] Outward spiraling from $(1.2, 28)$ and $(2.1, 32)$ (green);
    \item[(2)] Inward contraction from $(2.37, 32.15)$ and $(1.8, 30)$ (yellow);
    \item[(3)] Outward expansion from $(3.15, 32.15)$ (blue).
\end{itemize}
The most significant outcome emerges in the region surrounding $E_2$: three distinct concentric limit cycles materialize (red curves), comprising an innermost stable cycle, an intermediate unstable cycle, and an outermost stable cycle. This triadic limit cycle configuration provides definitive numerical evidence of the codimension 3 bifurcation. While $E_2$ at $(1.08727,28.0903)$ is unstable and $E_1$ at $(0.300757,7.77023)$ is a saddle-node character, the emergence of triple limit cycles represents the hallmark feature of this system.
\end{ex}

\begin{rem}\label{rem:4.6}
This work provides the   evidence of three concentric limit cycles emerging in any epidemic model through pure codimension 3 Hopf bifurcation. As demonstrated in Example \ref{ex:codim3} (Figure~\ref{fig:10}(b)), this novel configuration originates solely from degeneracies at a single Hopf point without auxiliary bifurcations, a phenomenon rarely  reported for infectious disease systems.

The mechanism enabling this high codimension bifurcation stems from the interaction between monotonic risk aversion and nonmonotonic risk compensation in our SIRS incidence function. This represents a fundamental advance in understanding oscillatory complexity in behavioral epidemiology.

\end{rem}
\begin{rem}\label{th:global_local}
Under the conditions of Theorem \ref{th:5.3} (codimension-3 Hopf bifurcation) and $\mathfrak{D}_{(p,q)}<0$:
\begin{itemize}
    \item[(1)] The outermost limit cycle is always stable, acting as a \textit{separatrix};
    \item[(2)] Trajectories escaping its repulsion are captured by $W^s(E_1)$;
    \item[(3)] The unstable manifold $W^u(E_1)$ connects to $(0,0)$;
    \item[(4)] The disease-free equilibrium $(0,0)$ absorbs all trajectories outside the basin of $E_2$.
\end{itemize}
\end{rem}

\subsection{Nonexistence and existence of limit cycles}

\begin{thm}[Nonexistence of limit cycles]
System \eqref{eq:2.4} has no limit cycles in the positive quadrant if either of the following holds:
\begin{itemize}
    \item[\rm (a)] There are no positive equilibria ($\mathfrak{D}_{(p,q)} > 0$);
    \item[\rm (b)] All positive equilibria are hyperbolic nodes (i.e., for each equilibrium $E_i$, $(\operatorname{Tr} J_{E_i})^2 - 4\det J_{E_i} > 0$ and eigenvalues real with the same sign);
    \item[\rm (c)] $G'(x) > \dfrac{n}{p(x)}$ for all $x > 0$ where $p(x) > 0$.
\end{itemize}
\end{thm}

\begin{proof}
(a) By the Poincare-Bendixson theorem, limit cycles must enclose at least one equilibrium. Absence of equilibria precludes cycles.

(b) For each hyperbolic node $E_i$:
\begin{itemize}
    \item If stable ($\operatorname{Tr} J_{E_i} < 0$), nearby trajectories approach $E_i$ monotonically;
    \item If unstable ($\operatorname{Tr} J_{E_i} > 0$), trajectories diverge radially.
\end{itemize}
The nodal structure prohibits oscillatory behavior near each $E_i$. Since limit cycles require spiral flow around some equilibrium, this topological constraint prevents cycle formation.

(c) Case 1: No positive equilibria $\Rightarrow$ no limit cycles by (a).

Case 2: Positive equilibria exist. Let $\{E_i\}$ be the set of positive equilibria with $x$-coordinates $\{x_{E_i}\}$. Take:
\[
\delta_1 = \frac{1}{2} \min_{i}\{ x_{E_i}\} > 0.
\]
Define $\widetilde{D} = \widetilde{\Sigma} \cap \{(x,y) \mid x \geq \delta_1\}$. Then $\widetilde{D}$ is compact and simply connected, and all positive equilibria lie in $\widetilde{D}$. Applying the time-rescaling $d\tau = p(x)dt$ to system \eqref{eq:5.1}, we obtain the transformed vector field:
\[
\mathbf{F}(x,y) = \left( G(x) - y,  \frac{x - n y}{p(x)} \right).
\]
The divergence of $\mathbf{F}$ is calculated as:
\[
\nabla \cdot \mathbf{F} = \frac{\partial}{\partial x}\left( G(x) - y \right) + \frac{\partial}{\partial y}\left( \frac{x - n y}{p(x)} \right) = G'(x) - \frac{n}{p(x)} > 0 \quad \forall (x,y) \in \widetilde{D}.
\]

By Dulac's criterion, no periodic orbits exist in $\widetilde{D}$. Since any limit cycle must enclose an equilibrium and thus lie in $\widetilde{D}$, we conclude that system \eqref{eq:2.4} has no limit cycles in the positive quadrant.
\end{proof}

\begin{thm}[Existence of limit cycles]\label{thm:exist}
Assume system \eqref{eq:2.4} satisfies $\mathfrak{D}_{(p,q)} < 0$ (two distinct positive equilibria $E_1$ saddle and $E_2$ focus) and $E_2$ is unstable ($\operatorname{Tr} J(E_2)=p(x_2)G'(x_2) - n > 0$). Then there exists at least one stable limit cycle surrounding $E_2$ in the compact invariant set $\widetilde{\Sigma}$.
\end{thm}

\begin{proof}
By Lemma \ref{lem:2.1} and the subsequent dimensional reduction, $\widetilde{\Sigma}$ is compact and positively invariant. All trajectories eventually enter and remain in $\widetilde{\Sigma}$ (see Section 2.1).

Theorem \ref{th:3.6} guarantees that $E_1$ is a saddle and $E_2$ is an unstable focus in $\widetilde{\Sigma}$. Since $E_0(0,0)$ is a stable node (Theorem \ref{th:3.2}), the unstable manifold of $E_1$ connects to $E_0$ (see Figure~\ref{fig:10}).

As $E_2$ is unstable and $\widetilde{\Sigma}$ contains no other attractors, the Poincare-Bendixson theorem implies the existence of a stable limit cycle surrounding $E_2$ in $\widetilde{\Sigma}$.
\end{proof}

\subsection{Hopf bifurcation diagram and differential topology of the bifurcation surface}

We select $(a,c,n)$ as the bifurcation parameters for system \eqref{eq:2.4}. At the critical parameter configuration $(\widehat{a},\widehat{c},\widehat{n})$, the focus quantities satisfy $B_1 = B_3 = B_5 = 0$ and $B_7\neq0$. Letting $(\widehat{a},\widehat{c},\widehat{n})=(\widehat{a}+\mu_1,\widehat{c}+\mu_2,\widehat{n}+\mu_3)$, this degeneracy condition implies that in a neighborhood $U$ of $(\widehat{a},\widehat{c},\widehat{n})$.
The original system is locally topologically equivalent to the normal form:
\begin{equation}\label{eq:5.10}
\begin{aligned}
    \frac{{\rm d} r}{{\rm d}t} &= \mu_1 r + \mu_2 r^3 + \mu_3 r^5 - r^7 + \mathcal{O}(r^9) \\
    \frac{{\rm d} \theta}{{\rm d}t} &= -1 + \mathcal{O}(\mu, r^2),
\end{aligned}
\end{equation}
near the origin, under a smooth parameter-dependent diffeomorphism. This equivalence is established through: (1) translation of equilibrium $E_2$ to the origin, (2) polar coordinate transformation, and (3) normal form reduction via near-identity transformations. Consequently, the bifurcation sets of both systems coincide locally in parameter space. This equivalence permits us to analyze the simplified system \eqref{eq:5.10} to characterize the complete bifurcation structure.

\begin{figure}[h]
  \centering
  \begin{minipage}{0.6\textwidth}
    \centering
    \includegraphics[width=\linewidth]{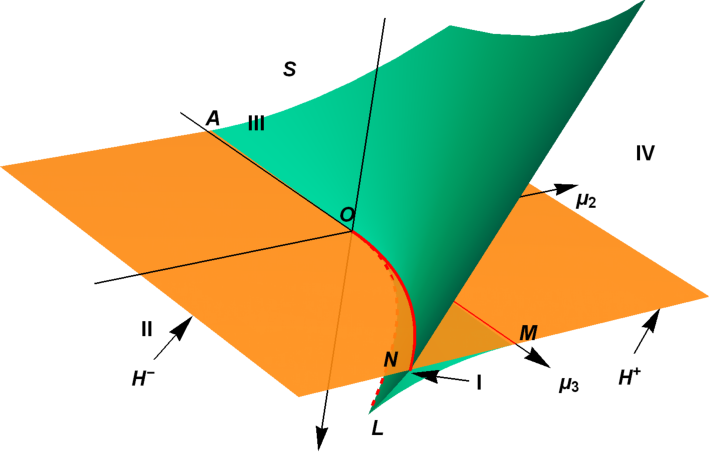}
    \caption{Codimension-3 Hopf bifurcation diagram.}
     \label{fig:111}
  \end{minipage}
\end{figure}

\begin{table}[htbp]
  \centering
  \caption{The codimension 3 phase portraits of equation \eqref{eq:5.10}.}
  \label{tab:phase4}
  \renewcommand{\arraystretch}{1.5}
  \setlength{\tabcolsep}{10pt}
  \begin{tabular}{|l|l|l|l|l|}
    \hline
    \begin{tabular}{@{}c@{}} \\[-3pt] \includegraphics[width=2.2cm, height=2.1cm]{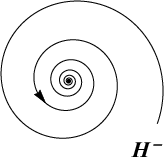} \\[-3pt] \end{tabular} &
    \begin{tabular}{@{}c@{}} \\[-3pt] \includegraphics[width=2.2cm, height=2.1cm]{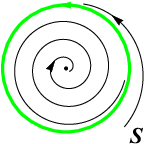} \\[-3pt] \end{tabular} &
    \begin{tabular}{@{}c@{}} \\[-3pt] \includegraphics[width=2.2cm, height=2.1cm]{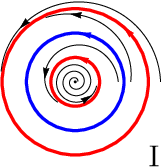} \\[-3pt] \end{tabular} &
    \begin{tabular}{@{}c@{}} \\[-3pt] \includegraphics[width=2.2cm, height=2.1cm]{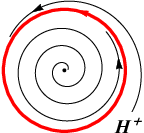} \\[-3pt] \end{tabular} &
    \begin{tabular}{@{}c@{}} \\[-3pt] \includegraphics[width=2.2cm, height=2.1cm]{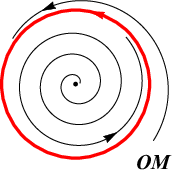} \\[-3pt] \end{tabular} \\
    \hline

    \begin{tabular}{@{}c@{}} \\[-3pt] \includegraphics[width=2.2cm, height=2.1cm]{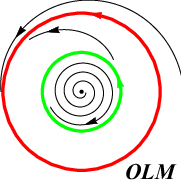} \\[-3pt] \end{tabular} &
    \begin{tabular}{@{}c@{}} \\[-3pt] \includegraphics[width=2.2cm, height=2.1cm]{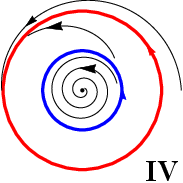} \\[-3pt] \end{tabular} &
    \begin{tabular}{@{}c@{}} \\[-3pt] \includegraphics[width=2.2cm, height=2.1cm]{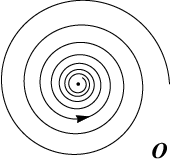} \\[-3pt] \end{tabular} &
    \begin{tabular}{@{}c@{}} \\[-3pt] \includegraphics[width=2.2cm, height=2.1cm]{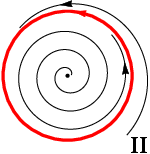} \\[-3pt] \end{tabular} &
    \begin{tabular}{@{}c@{}} \\[-3pt] \includegraphics[width=2.2cm, height=2.1cm]{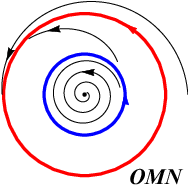} \\[-3pt] \end{tabular} \\
    \hline

    \begin{tabular}{@{}c@{}} \\[-3pt] \includegraphics[width=2.2cm, height=2.1cm]{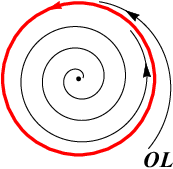} \\[-3pt] \end{tabular} &
    \begin{tabular}{@{}c@{}} \\[-3pt] \includegraphics[width=2.2cm, height=2.1cm]{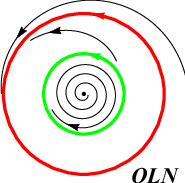} \\[-3pt] \end{tabular} &
    \begin{tabular}{@{}c@{}} \\[-3pt] \includegraphics[width=2.2cm, height=2.1cm]{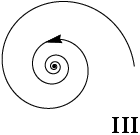} \\[-3pt] \end{tabular} &
    \begin{tabular}{@{}c@{}} \\[-3pt] \includegraphics[width=2.2cm, height=2.1cm]{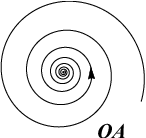} \\[-3pt] \end{tabular} &
    \begin{tabular}{@{}c@{}} \\[-3pt] \includegraphics[width=2.2cm, height=2.1cm]{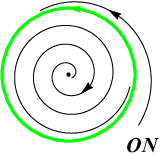} \\[-3pt] \end{tabular} \\
    \hline
  \end{tabular}
\end{table}

The $(\mu_1,\mu_2,\mu_3)$ parameter space exhibits the following bifurcation structure (see Figure~\ref{fig:111}):
\begin{enumerate}
    \item[(1)] The half-plane defined by $\mu_1 = 0$ and $\mu_2 < 0$ ($H^{-}$) corresponds to supercritical Hopf bifurcations.

    \item[(2)] Subcritical Hopf bifurcations emerge on the half-plane $H^{+}$: $\mu_1 = 0$ with $\mu_2 > 0$.

    \item[(3)] Along the $\mu_3$-axis (excluding $\mu_3 = 0$), codimension-2 Hopf bifurcations occur.

    \item[(4)] Saddle-node bifurcations of limit cycles appear on the surface $SN_{lc}$, composed of:
    \begin{itemize}
        \item[(4a)] Surface $S$ for $\mu_3 > 0$ (above $\mu_2\mu_3$-plane);
        \item[(4b)] Surfaces $OLM$ and $OLN$ for $\mu_3 < 0$ (below $\mu_2\mu_3$-plane).
    \end{itemize}
    A single semistable limit cycle exists throughout $SN_{lc}$, including boundaries $ON$ and $OL$.

    \item[(5)] Four generic regions (I-IV) are formed by the intersection of the $\mu_2\mu_3$-plane and $SN_{lc}$.
\end{enumerate}

Table \ref{tab:phase4} illustrates representative phase portraits of system \eqref{eq:5.10} across generic parameter regions and near bifurcation boundaries. Limit cycle stabilities are visually coded as: red circles (stable), blue circles (unstable), and green circles (semi-stable).\par
The classification of singular points for fronts relies on singularity theory, which provides powerful tools for analyzing degenerate behaviors in geometric structures. Building on this foundation, we now apply similar singularity-theoretic techniques to study degenerate Hopf bifurcations in dynamical systems. To investigate the degenerate Hopf bifurcation of the original system \eqref{eq:2.4} and characterize the local differential structure of its bifurcation set, we focus on its normal form (system \eqref{eq:5.10}). Applying singularity theory, we classify the singular points of codimension 3 in the Hopf bifurcation set of the normal form. We review the following criteria for singular points \cite{Kokubu2005, Saji2009}:
Consider an oriented 2-manifold $M^2$ and an oriented Riemannian 3-manifold $(N^3, g)$.
A smooth map $\widetilde{f} \colon M^2 \to N^3$ is classified as a {\it front} if there exists
a unit normal vector field $\bm\nu$ along $f$ such that the induced map
$L := (\widetilde{f}, \bm\nu) \colon M^2 \to T_1 N^3$ is a {\it Legendrian immersion}
(i.e., an isotropic immersion). This requires the pullback of the canonical contact form
on $T_1 N^3$ to vanish identically on $M^2$. Equivalently, $\widetilde{f}$ satisfies the
 orthogonality condition :
\[
g(\widetilde{f}_* X, \bm\nu) = 0 \quad \text{for all} \quad X \in TM^2,
\]
where $\widetilde{f}_*$ denotes the differential of $\widetilde{f}$. The vector field $\bm\nu$ is termed the
{\it unit normal} of the front. Specifically, the following lemma provides criteria for classifying singular points of fronts into cuspidal edges and swallowtails.
\begin{lem}[\cite{Saji2009}]\label{lem:5.10}
Let \( N^3 \) be a Riemannian 3-manifold, $\widetilde{\bm\gamma}(I)\subset U$ (where $I=(-\varepsilon,\varepsilon)$) be a singular set of a front \( \widetilde{f}: U \to N^3 \)  and   \( p = \widetilde{\bm\gamma}(0) \in U \) be a nondegenerate singular point of the front \( \widetilde{f}\).
\begin{itemize}
    \item[(1)] The germ of \( \widetilde{f} \) at \( p \) is a cuspidal edge if and only if the null direction \( \widetilde{\eta}(0) \) is not proportional to the singular direction \(\widetilde{\bm\gamma}'(0) \).

    \item[(2)] The germ of \( \widetilde{f} \) at \( p \) is a swallowtail if and only if:
    \begin{itemize}
        \item \( \widetilde{\eta}(0) \) is proportional to \( \widetilde{\bm\gamma}'(0) \), and
        \item The derivative satisfies:
        \begin{align*}
        \left.\frac{{\rm d}}{{\rm d}t}\right|_{t=0} \det\left(\widetilde{\bm\gamma}'(t), \widetilde{\eta}(t)\right) \neq 0,
        \end{align*}
    \end{itemize}
\end{itemize}
where \(\widetilde{\bm\gamma}'(t) \) and \( \widetilde{\eta}(t) \) are represented as column vectors in \( \mathbb{R}^2 \).
\end{lem}

To analyze the bifurcation structure, we define the surface $\mathcal{BS}:  [0, \infty)  \times \mathbb{R} \rightarrow \mathbb{R}^3$ by
\[
\mathcal{BS}(r,\mu_3) = (\mu_3 r^4 - 2r^6, 3r^4 - 2\mu_3 r^2, \mu_3),
\]
which parametrizes the singularity set. Furthermore, we introduce a potential function $V:  [0, \infty)  \times \mathbb{R}^3 \rightarrow \mathbb{R}$
\[
V(r) = \frac{1}{2}\mu_1 r^2 + \frac{1}{4}\mu_2 r^4 + \frac{1}{6}\mu_3 r^6 - \frac{1}{8}r^8,
\]
whose critical points approximate the equilibrium conditions of \eqref{eq:5.10} when higher-order terms are neglected. For fixed parameters $\bm\mu = (\mu_1,\mu_2,\mu_3) \in \mathbb{R}^3$, we denote $\mathfrak{v}_\mu(r) = V(r)$.

\begin{prop} \label{pr:5.11}
The singular set of $\mathcal{BS}$ is given by
\[
\left\{ (r,\mu_3) \in  [0, \infty)  \times \mathbb{R} \mid r = 0 \ \text{or}\ \mu_3 = 3r^2 \right\}.
\]
\begin{proof}
The partial derivatives of $\mathcal{BS}$ are:
\begin{align*}
\frac{\partial \mathcal{BS}}{\partial r} &= (4\mu_3 r^3 - 12r^5, 12r^3 - 4\mu_3 r, 0) = 4r(3r^2 - \mu_3)(-r^2, 1, 0), \\
\frac{\partial \mathcal{BS}}{\partial \mu_3} &= (r^4, -2r^2, 1).
\end{align*}
The cross product is:
\begin{align*}
\frac{\partial \mathcal{BS}}{\partial r} \wedge \frac{\partial \mathcal{BS}}{\partial \mu_3}
&= \begin{vmatrix}
\bm{i} & \bm{j} & \bm{k} \\
4r(3r^2 - \mu_3)(-r^2) & 4r(3r^2 - \mu_3)(1) & 0 \\
r^4 & -2r^2 & 1
\end{vmatrix} \\
&= 4r(3r^2 - \mu_3) \left(1, r^2, r^4\right)
\end{align*}
The tangent vectors are linearly dependent iff this cross product vanishes, which occurs precisely when $r(3r^2 - \mu_3) = 0$.
\end{proof}
\end{prop}

The following proposition establishes the degeneracy conditions for the potential function, its proof involves straightforward differentiation and is omitted for brevity.

\begin{prop}\label{pr:5.12}
The singularity conditions for $\mathfrak{v}_\mu$ characterize the degeneracy levels of the potential function:
\begin{enumerate}
    \item[\rm (1)] $\mathfrak{v}_\mu'(r) = 0$ iff $\mu_1 r + \mu_2 r^3 + \mu_3 r^5 - r^7 = 0$.
    \item[\rm (2)] $\mathfrak{v}_\mu'(r) = \mathfrak{v}_\mu''(r) = 0$ iff
        \[
        \bm{\mu} = (\mu_3 r^4 - 2 r^6, 3 r^4 - 2 \mu_3 r^2, \mu_3).
        \]
    \item[\rm (3)] $\mathfrak{v}_\mu'(r) = \mathfrak{v}_\mu''(r) = \mathfrak{v}_\mu'''(r) = 0$ iff $\bm{\mu} = (r^6, -3 r^4, 3 r^2)$.
    \item[\rm (4)] $\mathfrak{v}_\mu'(r) = \mathfrak{v}_\mu''(r) = \mathfrak{v}_\mu'''(r) = \mathfrak{v}_\mu^{(4)}(r) = 0$ iff $\bm{\mu} = (0,0,0)$.
\end{enumerate}
\end{prop}
\begin{rem} \label{rem:5.13}
Proposition \ref{pr:5.12} reveals that the surface $\mathcal{BS}$ parametrizes points where the potential function $\mathfrak{v}_\mu$ has degenerate critical points (specifically, where the first and second derivatives vanish simultaneously). This corresponds to the bifurcation set of system \eqref{eq:5.10} where saddle-node bifurcations occur. Higher degeneracies (items 3 and 4) correspond to more singular bifurcation points, the image of the curve $\mathcal{C}:[0,\infty)\rightarrow \mathbb{R}^3; \mathcal{C}(r)=(r^6,-3 r^4,3 r^2)$ is the critical value set (i.e., the image of the singular set) of the surface $\mathcal{BS}$ (see the red dashed curve in Figure \ref{fig:111}).
\end{rem}

Setting $N^3 = \mathbb{R}^3$ with the Euclidean metric $g = \langle \cdot, \cdot \rangle$, we assert that $\mathcal{BS}$ is a front. Indeed, taking the unit normal vector field
\[
\bm\nu = \dfrac{1}{\sqrt{r^8 + r^4 + 1}}\left(1, r^2, r^4\right),
\]
we verify the orthogonality conditions:
\begin{align*}
\left\langle \frac{\partial \mathcal{BS}}{\partial r}, \bm{\nu} \right\rangle = 0, \quad
\left\langle \frac{\partial \mathcal{BS}}{\partial \mu_3}, \bm{\nu} \right\rangle = 0.
\end{align*}

\begin{thm}\label{thm:BS-singularities}
The singularities of $\mathcal{BS}$ are classified as follows:
\begin{enumerate}
    \item[\rm (1)] The germ of \(\mathcal{BS}\) at \((r, \mu_3)\) is a cuspidal edge if and only if \(\mu_3 = 3r^2\) and \(r \neq 0\).
    \item[\rm (2)] The germ of \(\mathcal{BS}\) at \((r, \mu_3)\) is a swallowtail if and only if \(r = \mu_3 = 0\).
\end{enumerate}
\end{thm}

\begin{proof}
By Proposition \ref{pr:5.11}, the singular set of $\mathcal{BS}$ is $\{ (r,\mu_3) \mid r(3r^2 - \mu_3) = 0 \}$. Consider the singular curve parameterized by
\[
\widetilde{\bm\gamma}(r) = (r, 3r^2), \quad r \in I,
\]
with tangent vector $\widetilde{\bm\gamma}'(r) = (1, 6r)$.

To apply Lemma \ref{lem:5.10}, we construct a null vector field $\widetilde{\eta}$ along $\widetilde{\bm\gamma}$. At each point $(r, 3r^2)$, the kernel of $d\mathcal{BS}$ is spanned by $\begin{pmatrix} 1 \\ 0 \end{pmatrix}$, so we take $\widetilde{\eta}(r) = \begin{pmatrix} 1 \\ 0 \end{pmatrix}$.

For (1): When $r \neq 0$ and $\mu_3 = 3r^2$, we have $\widetilde{\bm\gamma}'(r) = (1, 6r)$ and $\widetilde{\eta}(r) = (1, 0)$. These vectors are linearly independent since $6r \neq 0$. By Lemma \ref{lem:5.10}(1), the singularity is a cuspidal edge.

For (2): At $(0,0)$, we have $\widetilde{\bm\gamma}'(0) = (1, 0)$ and $\widetilde{\eta}(0) = (1, 0)$, which are proportional. Consider the determinant function:
\[
\det\left( \widetilde{\bm\gamma}'(r), \widetilde{\eta}(r) \right) = \det \begin{pmatrix} 1 & 1 \\ 6r & 0 \end{pmatrix} = -6r.
\]
Its derivative at $r=0$ is:
\[
\left. \frac{\rm d}{{\rm d}r} \det\left( \widetilde{\bm\gamma}'(r), \widetilde{\eta}(r) \right) \right|_{r=0} = \left. \frac{{\rm d}}{{\rm d}r}(-6r) \right|_{r=0} = -6 \neq 0.
\]
Thus by Lemma \ref{lem:5.10} (2), $(0,0)$ is a swallowtail singularity.
\end{proof}

\begin{rem}\label{rem:5.15}
The potential function $\mathfrak{v}_\mu(r)$ becomes a quartic polynomial under the substitution $s = r^2$ ($s \geq 0$):
\[
\mathfrak{v}_\mu(s) = \frac{1}{2}\mu_1 s + \frac{1}{4}\mu_2 s^2 + \frac{1}{6}\mu_3 s^3 - \frac{1}{8}s^4.
\]
The bifurcation surface $\mathcal{BS}$ defined in Proposition \ref{pr:5.11} corresponds to the set where $\mathfrak{v}_\mu'(s)=\mathfrak{v}_\mu''(s) = 0$, which is precisely the bifurcation set of the normal form system \eqref{eq:5.10} where the radial equation and its first derivative simultaneously vanish:
\[
\frac{{\rm d}r}{{\rm d}t} = 0 \quad \text{and} \quad \frac{\partial}{\partial r}\left(\frac{{\rm d}r}{{\rm d}t}\right) = 0.
\]

In classical singularity theory, the whole swallowtail surface is generated by the universal unfolding of the $A_3$ singularity $V_0(s) = s^4/4$ without the restriction $s \geq 0$. When we enforce $s = r^2 \geq 0$, we obtain only half of the swallowtail surface, specifically the portion where $s \geq 0$.

For completeness, Figure \ref{fig:12} shows the full swallowtail surface in $\mathbb{R}^3$ with coordinates $(\mu_1, \mu_2, \mu_3)$, where the cyan portion corresponds to our bifurcation surface $\mathcal{BS}$ for $s \geq 0$. The cuspidal edge curve $\mu_3 = 3s$ (the red locus) and swallowtail point (the blue point) at the origin are visible.

\end{rem}
\begin{figure}[htbp]
\centering
\includegraphics[width=0.5\textwidth]{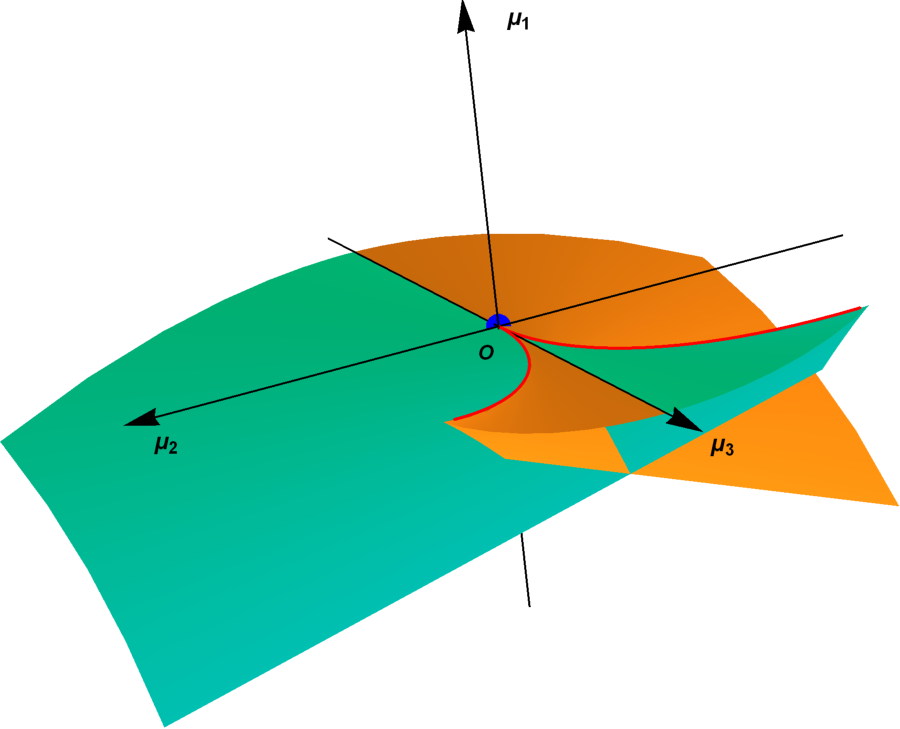}
\caption{The full swallowtail surface. The cyan portion ($s \geq 0$) corresponds to our bifurcation surface $\mathcal{BS}$, while the orange portion ($s < 0$) completes the swallowtail singularity.}
\label{fig:12}
\end{figure}
\subsection{Comparison with Lu et al. (2019) on Hopf Bifurcations}
We compare our results on degenerate Hopf bifurcations with those of Lu et al. \cite{Lu2019}. The following table highlights the key similarities and advancements.

\begin{table}
\centering
\caption{Comparative Hopf Bifurcation Analysis: Lu et al. (2019) vs. This Work.}
\label{tab:Hopf_comparison}
\renewcommand{\arraystretch}{1.3} 
\begin{tabular}{|p{0.3\textwidth}|p{0.25\textwidth}|p{0.35\textwidth}|}
\hline
\textbf{Aspect} & \textbf{Lu et al. (2019)} & \textbf{This work} \\
\hline
Incidence function & Quadratic & Cubic \\
\hline
Max codimension of Hopf & 2 & 3 \\
\hline
Limit cycles vs.\newline incidence monotonicity &
\makecell[l]{Nonmonotonic regime:\\ Up to 2 limit cycles} &
\makecell[l]{Nonmonotonic regime ($\beta < 0$):\\ Up to 2 limit cycles\\
          Monotonic regime ($\beta > 0$):\\ Up to 3 limit cycles} \\
\hline
Focus order & Up to 2nd order & Up to 3rd order \\
\hline
Method for stability \& degeneracy & Second Lyapunov coefficient ($\sigma_2$) & Focal values ($B_7$) and singularity theory \\
\hline
Analysis of weak focus & Up to 2nd order & Up to 3rd-order \\
\hline
Bifurcation diagram & 2D parameter space & 3D parameter space \\
\hline
Singularity analysis & Not applied & Swallowtail singularity \\
\hline
\end{tabular}
\end{table}

Table~\ref{tab:Hopf_comparison} illustrates that our work not only generalizes the incidence function but also achieves higher degeneracy in Hopf bifurcations, leading to the emergence of three limit cycles-a phenomenon not observed in Lu et al.'s model. The incorporation of singularity theory for bifurcation surfaces further enriches the topological understanding of the system.
\section{Concluding Remarks}
This work presents a systematic investigation of high-codimension bifurcations, particularly codimension-three Bogdanov-Takens and degenerate Hopf bifurcations, in SIRS epidemic models with cubic saturated incidence. By extending Lu et al.'s quadratic framework \cite{Lu2019}, we have demonstrated that cubic nonlinearities induce significantly richer and more complex dynamics than quadratic approximations. Key contributions include the rigorous identification and analysis of these high-codimension bifurcations, providing a detailed characterization of their unfolding and the associated multi-cycle dynamics.

Our innovative application of singularity and unfolding theory characterizes the local topology of bifurcation sets, revealing how higher-codimension singularities organize global dynamics. The comprehensive analysis shows distinct dynamical regimes: while nonmonotonic incidence ($\beta < 0$) can lead to double limit cycles, monotonic configurations ($\beta > 0$) can support triple limit cycle coexistence, a phenomenon that remains rare and poorly understood in epidemiological models. This work thus provides one of the few rigorously analyzed mathematical frameworks demonstrating how triple limit cycles can emerge in epidemic models through high-codimension bifurcations.

The integration of singularity theory with bifurcation analysis opens new pathways for exploring sophisticated dynamics in biological systems governed by higher-order nonlinearities. Future research should explore several promising directions: (1) extending the framework to age-structured or spatially heterogeneous models to enhance biological realism; (2) conducting stochastic bifurcation analysis under parameter uncertainty to better reflect real-world conditions; (3) investigating connections between cubic saturation and real-world intervention effects, potentially linking model parameters to measurable public health actions; (4) exploring nilpotent saddle-node bifurcations, which represent a natural extension of our singularity theory approach to higher-order degeneracies and could provide new insights into critical transitions between disease-free and endemic states; and (5) additional work on codimension-four singularities would further unravel the complex interplay between behavioral responses and epidemic thresholds.

This work establishes a mathematical foundation for understanding critical thresholds in disease transmission, with implications for predicting complex outbreak patterns and designing effective public health interventions. The mathematical insights gained from this study may also find applications beyond epidemiology in other biological systems exhibiting saturation effects and nonlinear feedback mechanisms.
\section*{Acknowledgements}
Research of S. Liu is supported by the National Natural Science Foundation of China (Grant Number: 12301627) and the Science and Technology Research Projects of the Education Office of Jilin Province, China (Grant Numbers: JJKH20250046KJ).

\section*{Competing Interests}
The authors have declared that no competing interests exist.

\bibliographystyle{elsarticle-num-names}
\bibliography{SIRS.bib}

\begin{thebibliography}{53}
\expandafter\ifx\csname natexlab\endcsname\relax\def\natexlab#1{#1}\fi
\providecommand{\url}[1]{\texttt{#1}}
\providecommand{\href}[2]{#2}
\providecommand{\path}[1]{#1}
\providecommand{\DOIprefix}{doi:}
\providecommand{\ArXivprefix}{arXiv:}
\providecommand{\URLprefix}{URL: }
\providecommand{\Pubmedprefix}{pmid:}
\providecommand{\doi}[1]{\href{http://dx.doi.org/#1}{\path{#1}}}
\providecommand{\Pubmed}[1]{\href{pmid:#1}{\path{#1}}}
\providecommand{\bibinfo}[2]{#2}
\ifx\xfnm\relax \def\xfnm[#1]{\unskip,\space#1}\fi
\bibitem[{W.~O.~Kermack(1927)}]{1927}
\bibinfo{author}{A.~G.~M. W.~O.~Kermack},
\newblock \bibinfo{title}{A contribution to the mathematical theory of epidemics},
\newblock \bibinfo{journal}{Proceedings of the Royal Society of London. Series A, Containing Papers of a Mathematical and Physical Character} \bibinfo{volume}{115} (\bibinfo{year}{1927}) \bibinfo{pages}{700--721}. \DOIprefix\doi{10.1098/rspa.1927.0118}.
\bibitem[{Li et~al.(1999)Li, Graef, Wang, and Karsai}]{Li1999}
\bibinfo{author}{M.~Y. Li}, \bibinfo{author}{J.~R. Graef}, \bibinfo{author}{L.~Wang}, \bibinfo{author}{J.~Karsai},
\newblock \bibinfo{title}{Global dynamics of a {SEIR} model with varying total population size},
\newblock \bibinfo{journal}{Mathematical Biosciences} \bibinfo{volume}{160} (\bibinfo{year}{1999}) \bibinfo{pages}{191--213}. \DOIprefix\doi{10.1016/s0025-5564(99)00030-9}.
\bibitem[{Liu and Li(2021)}]{Liu2021a}
\bibinfo{author}{S.~Liu}, \bibinfo{author}{M.~Y. Li},
\newblock \bibinfo{title}{Epidemic models with discrete state structures},
\newblock \bibinfo{journal}{Physica D: Nonlinear Phenomena} \bibinfo{volume}{422} (\bibinfo{year}{2021}) \bibinfo{pages}{132903}. \DOIprefix\doi{10.1016/j.physd.2021.132903}.
\bibitem[{Li et~al.(2025)Li, Chen, Liu, Zhang, and Liu}]{Li2025}
\bibinfo{author}{W.~Li}, \bibinfo{author}{X.~Chen}, \bibinfo{author}{S.~Liu}, \bibinfo{author}{C.~Zhang}, \bibinfo{author}{G.~Liu},
\newblock \bibinfo{title}{Using a multi-strain infectious disease model with physical information neural networks to study the time dependence of {SARS}-{C}o{V}-2 variants of concern},
\newblock \bibinfo{journal}{PLOS Computational Biology} \bibinfo{volume}{21} (\bibinfo{year}{2025}) \bibinfo{pages}{e1012778}. \DOIprefix\doi{10.1371/journal.pcbi.1012778}.
\bibitem[{V.~Capasso(1977)}]{V1977}
\bibinfo{author}{G.~S. V.~Capasso, E.~Crosso},
\newblock \bibinfo{title}{I modelli matematici nella indagine epidemiologica. i applicazione all¡¯epidemia di colera verificatasi in bari nel 1973},
\newblock \bibinfo{journal}{ANN. SCLAVO}  (\bibinfo{year}{1977}).
\bibitem[{Capasso and Serio(1978)}]{Capasso1978}
\bibinfo{author}{V.~Capasso}, \bibinfo{author}{G.~Serio},
\newblock \bibinfo{title}{A generalization of the {K}ermack-{McK}endrick deterministic epidemic model},
\newblock \bibinfo{journal}{Mathematical Biosciences} \bibinfo{volume}{42} (\bibinfo{year}{1978}) \bibinfo{pages}{43--61}. \DOIprefix\doi{10.1016/0025-5564(78)90006-8}.
\bibitem[{Liu et~al.(1986)Liu, Levin, and Iwasa}]{Liu1986}
\bibinfo{author}{W.~M. Liu}, \bibinfo{author}{S.~A. Levin}, \bibinfo{author}{Y.~Iwasa},
\newblock \bibinfo{title}{Influence of nonlinear incidence rates upon the behavior of {SIRS} epidemiological models},
\newblock \bibinfo{journal}{Journal of Mathematical Biology} \bibinfo{volume}{23} (\bibinfo{year}{1986}) \bibinfo{pages}{187--204}. \DOIprefix\doi{10.1007/bf00276956}.
\bibitem[{Liu et~al.(1987)Liu, Hethcote, and Levin}]{Liu1987}
\bibinfo{author}{W.~Liu}, \bibinfo{author}{H.~W. Hethcote}, \bibinfo{author}{S.~A. Levin},
\newblock \bibinfo{title}{Dynamical behavior of epidemiological models with nonlinear incidence rates},
\newblock \bibinfo{journal}{Journal of Mathematical Biology} \bibinfo{volume}{25} (\bibinfo{year}{1987}) \bibinfo{pages}{359--380}. \DOIprefix\doi{10.1007/bf00277162}.
\bibitem[{Hethcote and van~den Driessche(1991)}]{Hethcote1991}
\bibinfo{author}{H.~W. Hethcote}, \bibinfo{author}{P.~van~den Driessche},
\newblock \bibinfo{title}{Some epidemiological models with nonlinear incidence},
\newblock \bibinfo{journal}{Journal of Mathematical Biology} \bibinfo{volume}{29} (\bibinfo{year}{1991}) \bibinfo{pages}{271--287}. \DOIprefix\doi{10.1007/bf00160539}.
\bibitem[{Ruan and Wang(2003)}]{Ruan2003}
\bibinfo{author}{S.~Ruan}, \bibinfo{author}{W.~Wang},
\newblock \bibinfo{title}{Dynamical behavior of an epidemic model with a nonlinear incidence rate},
\newblock \bibinfo{journal}{Journal of Differential Equations} \bibinfo{volume}{188} (\bibinfo{year}{2003}) \bibinfo{pages}{135--163}. \DOIprefix\doi{10.1016/s0022-0396(02)00089-x}.
\bibitem[{Tang et~al.(2008)Tang, Huang, Ruan, and Zhang}]{Tang2008}
\bibinfo{author}{Y.~Tang}, \bibinfo{author}{D.~Huang}, \bibinfo{author}{S.~Ruan}, \bibinfo{author}{W.~Zhang},
\newblock \bibinfo{title}{Coexistence of limit cycles and homoclinic loops in a {SIRS} model with a nonlinear incidence rate},
\newblock \bibinfo{journal}{SIAM Journal on Applied Mathematics} \bibinfo{volume}{69} (\bibinfo{year}{2008}) \bibinfo{pages}{621--639}. \DOIprefix\doi{10.1137/070700966}.
\bibitem[{Saha and Ghosh(2022)}]{Saha2022}
\bibinfo{author}{P.~Saha}, \bibinfo{author}{U.~Ghosh},
\newblock \bibinfo{title}{Complex dynamics and control analysis of an epidemic model with non-monotone incidence and saturated treatment},
\newblock \bibinfo{journal}{International Journal of Dynamics and Control} \bibinfo{volume}{11} (\bibinfo{year}{2022}) \bibinfo{pages}{301--323}. \DOIprefix\doi{10.1007/s40435-022-00969-7}.
\bibitem[{Cui and Zhao(2024)}]{Cui2024}
\bibinfo{author}{W.~Cui}, \bibinfo{author}{Y.~Zhao},
\newblock \bibinfo{title}{Saddle-node bifurcation and {B}ogdanov-{T}akens bifurcation of a {SIRS} epidemic model with nonlinear incidence rate},
\newblock \bibinfo{journal}{Journal of Differential Equations} \bibinfo{volume}{384} (\bibinfo{year}{2024}) \bibinfo{pages}{252--278}. \DOIprefix\doi{10.1016/j.jde.2023.11.030}.
\bibitem[{Kuznetsov(2023)}]{Kuznetsov2023}
\bibinfo{author}{Y.~A. Kuznetsov}, \bibinfo{title}{Elements of Applied Bifurcation Theory}, \bibinfo{publisher}{Springer International Publishing}, \bibinfo{year}{2023}. \DOIprefix\doi{10.1007/978-3-031-22007-4}.
\bibitem[{Perko(1996)}]{Perko1996}
\bibinfo{author}{L.~Perko}, \bibinfo{title}{Differential Equations and Dynamical Systems}, \bibinfo{publisher}{Springer US}, \bibinfo{year}{1996}. \DOIprefix\doi{10.1007/978-1-4684-0249-0}.
\bibitem[{Xiao and Ruan(2007)}]{Xiao2007a}
\bibinfo{author}{D.~Xiao}, \bibinfo{author}{S.~Ruan},
\newblock \bibinfo{title}{Global analysis of an epidemic model with nonmonotone incidence rate},
\newblock \bibinfo{journal}{Mathematical Biosciences} \bibinfo{volume}{208} (\bibinfo{year}{2007}) \bibinfo{pages}{419--429}. \DOIprefix\doi{10.1016/j.mbs.2006.09.025}.
\bibitem[{Liu et~al.(2015)Liu, Pang, Ruan, and Zhang}]{Liu2015}
\bibinfo{author}{S.~Liu}, \bibinfo{author}{L.~Pang}, \bibinfo{author}{S.~Ruan}, \bibinfo{author}{X.~Zhang},
\newblock \bibinfo{title}{Global dynamics of avian influenza epidemic models with psychological effect},
\newblock \bibinfo{journal}{Computational and Mathematical Methods in Medicine} \bibinfo{volume}{2015} (\bibinfo{year}{2015}) \bibinfo{pages}{1--12}. \DOIprefix\doi{10.1155/2015/913726}.
\bibitem[{Zhou et~al.(2007)Zhou, Xiao, and Li}]{Zhou2007}
\bibinfo{author}{Y.~Zhou}, \bibinfo{author}{D.~Xiao}, \bibinfo{author}{Y.~Li},
\newblock \bibinfo{title}{Bifurcations of an epidemic model with non-monotonic incidence rate of saturated mass action},
\newblock \bibinfo{journal}{Chaos, Solitons \& Fractals} \bibinfo{volume}{32} (\bibinfo{year}{2007}) \bibinfo{pages}{1903--1915}. \DOIprefix\doi{10.1016/j.chaos.2006.01.002}.
\bibitem[{Lu et~al.(2020)Lu, Huang, Ruan, and Yu}]{Lu2020}
\bibinfo{author}{M.~Lu}, \bibinfo{author}{J.~Huang}, \bibinfo{author}{S.~Ruan}, \bibinfo{author}{P.~Yu},
\newblock \bibinfo{title}{Global dynamics of a susceptible-infectious-recovered epidemic model with a generalized nonmonotone incidence rate},
\newblock \bibinfo{journal}{Journal of Dynamics and Differential Equations} \bibinfo{volume}{33} (\bibinfo{year}{2020}) \bibinfo{pages}{1625--1661}. \DOIprefix\doi{10.1007/s10884-020-09862-3}.
\bibitem[{Lu et~al.(2019)Lu, Huang, Ruan, and Yu}]{Lu2019}
\bibinfo{author}{M.~Lu}, \bibinfo{author}{J.~Huang}, \bibinfo{author}{S.~Ruan}, \bibinfo{author}{P.~Yu},
\newblock \bibinfo{title}{Bifurcation analysis of an {SIRS} epidemic model with a generalized nonmonotone and saturated incidence rate},
\newblock \bibinfo{journal}{Journal of Differential Equations} \bibinfo{volume}{267} (\bibinfo{year}{2019}) \bibinfo{pages}{1859--1898}. \DOIprefix\doi{10.1016/j.jde.2019.03.005}.
\bibitem[{Li et~al.(2015)Li, Li, and Ma}]{Li2015}
\bibinfo{author}{C.~Li}, \bibinfo{author}{J.~Li}, \bibinfo{author}{Z.~Ma},
\newblock \bibinfo{title}{Codimension 3 {B-T} bifurcations in an epidemic model with a nonlinear incidence},
\newblock \bibinfo{journal}{Discrete and Continuous Dynamical Systems - Series B} \bibinfo{volume}{20} (\bibinfo{year}{2015}) \bibinfo{pages}{1107--1116}. \DOIprefix\doi{10.3934/dcdsb.2015.20.1107}.
\bibitem[{Harris(2023)}]{Harris2023}
\bibinfo{author}{E.~Harris},
\newblock \bibinfo{title}{Hybrid {COVID}-19 immunity common by fall of 2022},
\newblock \bibinfo{journal}{JAMA} \bibinfo{volume}{330} (\bibinfo{year}{2023}) \bibinfo{pages}{13}. \DOIprefix\doi{10.1001/jama.2023.10399}.
\bibitem[{Jang et~al.(2024)Jang, Kim, Lee, Son, Ko, and Lee}]{Jang2024}
\bibinfo{author}{G.~Jang}, \bibinfo{author}{J.~Kim}, \bibinfo{author}{Y.~Lee}, \bibinfo{author}{C.~Son}, \bibinfo{author}{K.~T. Ko}, \bibinfo{author}{H.~Lee},
\newblock \bibinfo{title}{Analysis of the impact of {COVID}-19 variants and vaccination on the time-varying reproduction number: statistical methods},
\newblock \bibinfo{journal}{Frontiers in Public Health} \bibinfo{volume}{12} (\bibinfo{year}{2024}). \DOIprefix\doi{10.3389/fpubh.2024.1353441}.
\bibitem[{Rosen(1977)}]{Rosen1977}
\bibinfo{author}{R.~Rosen},
\newblock \bibinfo{title}{Structural stability and morphogenesis},
\newblock \bibinfo{journal}{Bulletin of Mathematical Biology} \bibinfo{volume}{39} (\bibinfo{year}{1977}) \bibinfo{pages}{629--632}. \DOIprefix\doi{10.1007/BF02461210}.
\bibitem[{Arnold et~al.(2012)Arnold, Gusein-Zade, and Varchenko}]{Arnold2012}
\bibinfo{author}{V.~Arnold}, \bibinfo{author}{S.~Gusein-Zade}, \bibinfo{author}{A.~Varchenko}, \bibinfo{title}{Singularities of Differentiable Maps, Volume 1: Classification of Critical Points, Caustics and Wave Fronts}, \bibinfo{publisher}{Birkhäuser Boston}, \bibinfo{year}{2012}. \DOIprefix\doi{10.1007/978-0-8176-8340-5}.
\bibitem[{Wall(1981)}]{Wall1981}
\bibinfo{author}{C.~T.~C. Wall},
\newblock \bibinfo{title}{Finite determinacy of smooth map-germs},
\newblock \bibinfo{journal}{Bulletin of the London Mathematical Society} \bibinfo{volume}{13} (\bibinfo{year}{1981}) \bibinfo{pages}{481--539}. \DOIprefix\doi{10.1112/blms/13.6.481}.
\bibitem[{Golubitsky and Langford(1981)}]{Golubitsky1981}
\bibinfo{author}{M.~Golubitsky}, \bibinfo{author}{W.~F. Langford},
\newblock \bibinfo{title}{Classification and unfoldings of degenerate hopf bifurcations},
\newblock \bibinfo{journal}{Journal of Differential Equations} \bibinfo{volume}{41} (\bibinfo{year}{1981}) \bibinfo{pages}{375--415}. \DOIprefix\doi{10.1016/0022-0396(81)90045-0}.
\bibitem[{Bruce and Giblin(1992)}]{Bruce1992}
\bibinfo{author}{J.~W. Bruce}, \bibinfo{author}{P.~J. Giblin}, \bibinfo{title}{Curves and Singularities: A Geometrical Introduction to Singularity Theory}, \bibinfo{publisher}{Cambridge University Press}, \bibinfo{year}{1992}. \DOIprefix\doi{10.1017/cbo9781139172615}.
\bibitem[{Izumiya et~al.(2014)Izumiya, Romero~Fuster, Ruas, and Tari}]{Izumiya2014}
\bibinfo{author}{S.~Izumiya}, \bibinfo{author}{M.~d.~C. Romero~Fuster}, \bibinfo{author}{M.~A.~S. Ruas}, \bibinfo{author}{F.~Tari}, \bibinfo{title}{Differential Geometry from a Singularity Theory Viewpoint}, \bibinfo{publisher}{World Scientific}, \bibinfo{year}{2014}. \DOIprefix\doi{10.1142/9108}.
\bibitem[{Tang and Zhang(2016)}]{Tang2016}
\bibinfo{author}{Y.~Tang}, \bibinfo{author}{W.~Zhang},
\newblock \bibinfo{title}{Versal unfolding of planar hamiltonian systems at fully degenerate equilibrium},
\newblock \bibinfo{journal}{Journal of Differential Equations} \bibinfo{volume}{261} (\bibinfo{year}{2016}) \bibinfo{pages}{236--272}. \DOIprefix\doi{10.1016/j.jde.2016.03.008}.
\bibitem[{Tang and Zhang(2019)}]{Tang2019}
\bibinfo{author}{Y.~Tang}, \bibinfo{author}{W.~Zhang},
\newblock \bibinfo{title}{Versal unfolding of a nilpotent liénard equilibrium within the odd liénard family},
\newblock \bibinfo{journal}{Journal of Differential Equations} \bibinfo{volume}{267} (\bibinfo{year}{2019}) \bibinfo{pages}{2671--2685}. \DOIprefix\doi{10.1016/j.jde.2019.03.025}.
\bibitem[{Liu et~al.(2021)Liu, Tang, and Zhang}]{Liu2021}
\bibinfo{author}{L.~Liu}, \bibinfo{author}{Y.~Tang}, \bibinfo{author}{W.~Zhang},
\newblock \bibinfo{title}{Versal unfolding of homogeneous cubic degenerate centers in strong monodromic family},
\newblock \bibinfo{journal}{Journal of Differential Equations} \bibinfo{volume}{283} (\bibinfo{year}{2021}) \bibinfo{pages}{136--162}. \DOIprefix\doi{10.1016/j.jde.2021.02.038}.
\bibitem[{Hu et~al.(2011)Hu, Bi, Ma, and Ruan}]{Hu2011}
\bibinfo{author}{Z.~Hu}, \bibinfo{author}{P.~Bi}, \bibinfo{author}{W.~Ma}, \bibinfo{author}{S.~Ruan},
\newblock \bibinfo{title}{Bifurcations of an {SIRS} epidemic model with nonlinear incidence rate},
\newblock \bibinfo{journal}{Discrete and Continuous Dynamical Systems - B} \bibinfo{volume}{15} (\bibinfo{year}{2011}) \bibinfo{pages}{93--112}. \DOIprefix\doi{10.3934/dcdsb.2011.15.93}.
\bibitem[{Alexander and Moghadas(2005)}]{Alexander2005}
\bibinfo{author}{M.~E. Alexander}, \bibinfo{author}{S.~M. Moghadas},
\newblock \bibinfo{title}{Bifurcation analysis of an {SIRS} epidemic model with generalized incidence},
\newblock \bibinfo{journal}{SIAM Journal on Applied Mathematics} \bibinfo{volume}{65} (\bibinfo{year}{2005}) \bibinfo{pages}{1794--1816}. \DOIprefix\doi{10.1137/040604947}.
\bibitem[{She and Yi(2024)}]{She2024}
\bibinfo{author}{G.~She}, \bibinfo{author}{F.~Yi},
\newblock \bibinfo{title}{Stability and bifurcation analysis of a reaction–diffusion {SIRS} epidemic model with the general saturated incidence rate},
\newblock \bibinfo{journal}{Journal of Nonlinear Science} \bibinfo{volume}{34} (\bibinfo{year}{2024}). \DOIprefix\doi{10.1007/s00332-024-10081-z}.
\bibitem[{Xiao and Fang~Zhang(2007)}]{Xiao2007}
\bibinfo{author}{D.~Xiao}, \bibinfo{author}{K.~Fang~Zhang},
\newblock \bibinfo{title}{Multiple bifurcations of a predator-prey system},
\newblock \bibinfo{journal}{Discrete and Continuous Dynamical Systems - B} \bibinfo{volume}{8} (\bibinfo{year}{2007}) \bibinfo{pages}{417--433}. \DOIprefix\doi{10.3934/dcdsb.2007.8.417}.
\bibitem[{Saunders(1980)}]{Saunders1980}
\bibinfo{author}{P.~T. Saunders}, \bibinfo{title}{An Introduction to Catastrophe Theory}, \bibinfo{publisher}{Cambridge University Press}, \bibinfo{year}{1980}. \DOIprefix\doi{10.1017/cbo9781139171533}.
\bibitem[{Chow et~al.(1994)Chow, Li, and Wang}]{Chow1994}
\bibinfo{author}{S.-N. Chow}, \bibinfo{author}{C.~Li}, \bibinfo{author}{D.~Wang}, \bibinfo{title}{Normal Forms and Bifurcation of Planar Vector Fields}, \bibinfo{publisher}{Cambridge University Press}, \bibinfo{year}{1994}. \DOIprefix\doi{10.1017/cbo9780511665639}.
\bibitem[{Dumortier et~al.(1987)Dumortier, Roussarie, and Sotomayor}]{Dumortier1987}
\bibinfo{author}{F.~Dumortier}, \bibinfo{author}{R.~Roussarie}, \bibinfo{author}{J.~Sotomayor},
\newblock \bibinfo{title}{Generic 3-parameter families of vector fields on the plane, unfolding a singularity with nilpotent linear part. the cusp case of codimension 3},
\newblock \bibinfo{journal}{Ergodic Theory and Dynamical Systems} \bibinfo{volume}{7} (\bibinfo{year}{1987}) \bibinfo{pages}{375--413}. \DOIprefix\doi{10.1017/s0143385700004119}.
\bibitem[{Hopf(1942)}]{e.hofp}
\bibinfo{author}{E.~Hopf},
\newblock \bibinfo{title}{Abzweigung einer periodischen lösung von einer stationären lösung eines differentialsystems},
\newblock \bibinfo{journal}{Berichte der Mathematisch-Physikalischen Klasse der Sächsischen Akademie der Wissenschaften zu Leipzig}  (\bibinfo{year}{1942}).
\bibitem[{Bautin(1952)}]{NN}
\bibinfo{author}{N.~N. Bautin},
\newblock \bibinfo{title}{On the number of limit cycles appearing with variation of coefficients from an equilibrium position of focus or center type},
\newblock \bibinfo{journal}{Mat. Sb. (N.S.)}  (\bibinfo{year}{1952}).
\bibitem[{Arnold(1988)}]{Arnold1988}
\bibinfo{author}{V.~I. Arnold}, \bibinfo{title}{Geometrical Methods in the Theory of Ordinary Differential Equations}, \bibinfo{publisher}{Springer New York}, \bibinfo{year}{1988}. \DOIprefix\doi{10.1007/978-1-4612-1037-5}.
\bibitem[{Guckenheimer and Holmes(1983)}]{Guckenheimer1983}
\bibinfo{author}{J.~Guckenheimer}, \bibinfo{author}{P.~Holmes}, \bibinfo{title}{Nonlinear Oscillations, Dynamical Systems, and Bifurcations of Vector Fields}, \bibinfo{publisher}{Springer New York}, \bibinfo{year}{1983}. \DOIprefix\doi{10.1007/978-1-4612-1140-2}.
\bibitem[{B.~D.~Hassard(2007)}]{2007}
\bibinfo{author}{N.~D.~K. B.~D.~Hassard}, \bibinfo{title}{Interpolation Theory and Applications}, \bibinfo{publisher}{American Mathematical Society}, \bibinfo{year}{2007}. \DOIprefix\doi{10.1090/conm/445}.
\bibitem[{Chow and Hale(1982)}]{Chow1982}
\bibinfo{author}{S.-N. Chow}, \bibinfo{author}{J.~K. Hale}, \bibinfo{title}{Methods of Bifurcation Theory}, \bibinfo{publisher}{Springer New York}, \bibinfo{year}{1982}. \DOIprefix\doi{10.1007/978-1-4613-8159-4}.
\bibitem[{Han(2001)}]{Han2001}
\bibinfo{author}{M.~Han},
\newblock \bibinfo{title}{The hopf cyclicity of lienard systems},
\newblock \bibinfo{journal}{Applied Mathematics Letters} \bibinfo{volume}{14} (\bibinfo{year}{2001}) \bibinfo{pages}{183--188}. \DOIprefix\doi{10.1016/s0893-9659(00)00133-6}.
\bibitem[{Arsie et~al.(2022)Arsie, Kottegoda, and Shan}]{Arsie2022}
\bibinfo{author}{A.~Arsie}, \bibinfo{author}{C.~Kottegoda}, \bibinfo{author}{C.~Shan},
\newblock \bibinfo{title}{A predator-prey system with generalized holling type {IV} functional response and {A}llee effects in prey},
\newblock \bibinfo{journal}{Journal of Differential Equations} \bibinfo{volume}{309} (\bibinfo{year}{2022}) \bibinfo{pages}{704--740}. \DOIprefix\doi{10.1016/j.jde.2021.11.041}.
\bibitem[{Kokubu et~al.(2005)Kokubu, Rossman, Saji, Umehara, and Yamada}]{Kokubu2005}
\bibinfo{author}{M.~Kokubu}, \bibinfo{author}{W.~Rossman}, \bibinfo{author}{K.~Saji}, \bibinfo{author}{M.~Umehara}, \bibinfo{author}{K.~Yamada},
\newblock \bibinfo{title}{Singularities of flat fronts in hyperbolic space},
\newblock \bibinfo{journal}{Pacific Journal of Mathematics} \bibinfo{volume}{221} (\bibinfo{year}{2005}) \bibinfo{pages}{303--351}. \DOIprefix\doi{10.2140/pjm.2005.221.303}.
\bibitem[{Saji et~al.(2009)Saji, Umehara, and Yamada}]{Saji2009}
\bibinfo{author}{K.~Saji}, \bibinfo{author}{M.~Umehara}, \bibinfo{author}{K.~Yamada},
\newblock \bibinfo{title}{The geometry of fronts},
\newblock \bibinfo{journal}{Annals of Mathematics} \bibinfo{volume}{169} (\bibinfo{year}{2009}) \bibinfo{pages}{491--529}. \DOIprefix\doi{10.4007/annals.2009.169.491}.
\bibitem[{Bogdanov(1975)}]{Bogdanov1975}
\bibinfo{author}{R.~I. Bogdanov},
\newblock \bibinfo{title}{Versal deformations of a singular point of a vector field on the plane in the case of zero eigenvalues},
\newblock \bibinfo{journal}{Functional Analysis and Its Applications} \bibinfo{volume}{9} (\bibinfo{year}{1975}) \bibinfo{pages}{144--145}. \DOIprefix\doi{10.1007/bf01075453}.
\bibitem[{Bogdanov(1981)}]{Bogdanov}
\bibinfo{author}{R.~I. Bogdanov},
\newblock \bibinfo{title}{Bifurcation of a limit cycle for a family of vector fields on the plane},
\newblock \bibinfo{journal}{Selecta Mathematica Sovietica}  (\bibinfo{year}{1981}).
\bibitem[{Takens(1973)}]{Takens1973}
\bibinfo{author}{F.~Takens},
\newblock \bibinfo{title}{Unfoldings of certain singularities of vectorfields: Generalized hopf bifurcations},
\newblock \bibinfo{journal}{Journal of Differential Equations} \bibinfo{volume}{14} (\bibinfo{year}{1973}) \bibinfo{pages}{476--493}. \DOIprefix\doi{10.1016/0022-0396(73)90062-4}.
\bibitem[{Montaldi(2021)}]{Montaldi2021}
\bibinfo{author}{J.~Montaldi}, \bibinfo{title}{Singularities, Bifurcations and Catastrophes}, \bibinfo{publisher}{Cambridge University Press}, \bibinfo{year}{2021}. \DOIprefix\doi{10.1017/9781316585085}.

\end{thebibliography}

\newpage
\appendix
\vspace{1em}
\section{Proofs of Theorems}

\subsection{Proof of Theorem \ref{th:3.4}}
\label{proof_th:3.4}

\begin{proof}
By Theorem~\ref{th:2.11}(2), the conditions $\mathfrak{D}_{(p,q)} = 0$ ensure the existence of a unique double positive equilibrium $(x^*,y^*)$ satisfying $y^* = \frac{1}{n}x^*$. Moreover, it follows from Lemma~\ref{lem:2.5} and Theorem~\ref{th:2.9} that $\mathfrak{D}_{(p,q)} = 0$  is equivalent to the algebraic condition
\[
(c(n+1) + bmn)(x^*)^3 - n(x^*)^2 + amn x^* + mn = 0
\]
together with $\det(J^*) = 0$, where the Jacobian matrix evaluated at $(x^*, y^*)$ is
\[
J^* =
\begin{pmatrix}
\displaystyle \frac{-2amx^* - 4(bm+c)(x^*)^3 - 3c(x^*)^2y^* - m + 3(x^*)^2}{D} & \displaystyle \frac{-c(x^*)^3}{D} \\[2ex]
1 & -n
\end{pmatrix},
\]
and $D = 1 + a x^* + b (x^*)^3 > 0$ by the assumption $a > -3\sqrt[3]{\frac{1}{4}b}$. Furthermore, the determinant simplifies as follows:
\begin{align*}
\det(J^*) &= \frac{4(c(n+1)+bmn)(x^*)^3 - 3n(x^*)^2 + 2amn x^* + mn}{D} \\
&= \frac{n\left((x^*)^2 - 2am x^* - 3m\right)}{D}.
\end{align*}
Setting $\det(J^*) = 0$ yields the expression:
\[
x^* = am + \sqrt{a^2 m^2 + 3m}.
\]
For the trace calculation:
\begin{align*}
\mathrm{Tr}(J^*) &= \frac{-2amx^* - 4(bm+c)(x^*)^3 - 3c(x^*)^2(\frac{x^*}{n}) - m + 3(x^*)^2}{D} - n \\
&= \frac{-4(bm+c)n(x^*)^3 - 3c(x^*)^3 - 2amnx^* - mn + 3n(x^*)^2}{nD} - n.
\end{align*}
Substituting the equilibrium relation $4(c(n+1)+bmn)(x^*)^3 = 3n(x^*)^2 - 2amnx^* - mn$ yields:
\begin{align*}
\mathrm{Tr}(J^*) &= \frac{c(x^*)^3 - n^2(1 + ax^* + b(x^*)^3)}{nD} = \frac{\mathcal{N}_T}{nD},
\end{align*}
where \(\mathcal{N}_T := c(x^*)^3 - n^2D\).
The coupled system governing variables $c$ and $n$ consists of:
\begin{equation}\label{eq:2.34}
\begin{cases}
\mathcal{N}_T = c(x^*)^3 - n^2D = 0, \\
(c(n+1) + bmn)(x^*)^{3} - n(x^*)^{2} + amnx^* + mn = 0.
\end{cases}
\end{equation}
From the first equation, we explicitly solve for $c$ in terms of $n$:
\begin{equation}\label{eq:2.31.1}
c = \frac{n^2D}{(x^*)^3}.
\end{equation}
Substituting equation \eqref{eq:2.31.1} into the second equation, we have
\begin{equation}
Dn^2 + Dn + mD - (x^*)^{2} = 0.
\end{equation}
Under condition \eqref{eq:3.2}, this quadratic admits precisely one positive solution for $n$ (the critical threshold):
\begin{equation}\label{eq:n-critical}
n^* = \frac{-D + \sqrt{D^2 - 4D(mD - (x^*)^2)}}{2D}.
\end{equation}
The mathematical structure of the solution reveals an inverse cubic relationship between parameters $c$ and $n$ governed by $c = (n^*)^2D/(x^*)^3$, where $n^*$ determines the bifurcation point. This coupled parameter relationship defines the critical threshold at which the system's stability properties undergo qualitative transformation, marking the boundary between distinct dynamical regimes.

\begin{figure}[htbp]
\centering
\includegraphics[width=0.5\textwidth]{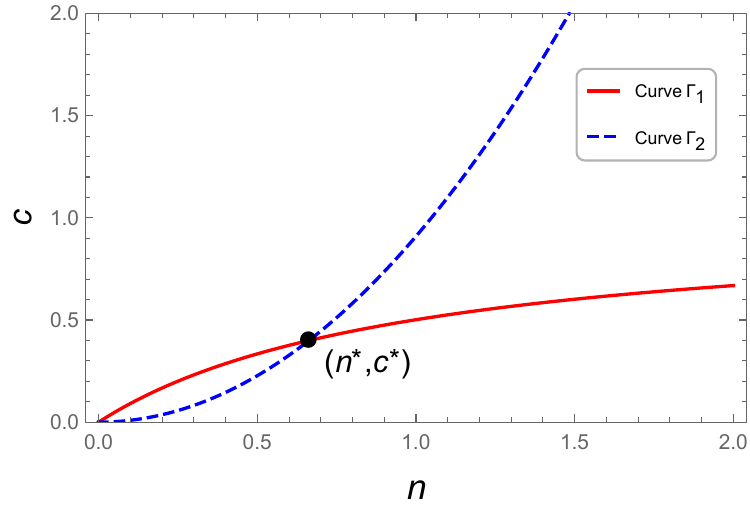}
\caption{Intersection of $\Gamma_1$ (red) and $\Gamma_2$ (blue) at $(n,c) =(n^*,c^*)$.}
\label{fig:6}
\end{figure}

Stability transitions are analyzed through geometric relationships in the $(n,c)$-parameter space (shown as Figure~\ref{fig:6}), and
$\Gamma_1$: $c = \dfrac{n((x^*)^{2} -mD)}{(n+1)(x^*)^{3}}$ (equilibrium condition);
$\Gamma_2$: $c = \dfrac{n^2D}{(x^*)^3}$ (from $\mathcal{N}_T = 0$).
The $c$-difference function between these curves is:
\begin{equation}
d_c(n) := c_{\Gamma_1} - c_{\Gamma_2} = \dfrac{n((x^*)^{2} -mD)}{(n+1)(x^*)^{3}} - \frac{n^2D}{(x^*)^3}.
\end{equation}
Normalizing by multiplying through by $\dfrac{1}{n}$ gives:
\begin{align*}
\frac{1}{n}d_c(n) &= \frac{-Dn^2 - Dn + (x^*)^2 - mD}{(n+1)(x^*)^{3}}.
\end{align*}
Sign analysis of $d_c(n)$ reveals:
\begin{enumerate}[label=\textnormal{(\roman*)}]
\item Case $n = n^*$: At the intersection $(n,c)=(n^*,c^*)$ where $c^* := \dfrac{(n^*)^2D}{(x^*)^3}$, both trace and determinant vanish by definition, hence $d_c(n^*) = 0$.
\item Case $n < n^*$: Here $\dfrac{1}{n}d_c(n) > 0$ implies $d_c(n) > 0$. For fixed $x^*$, decreasing $n$ below $n^*$ elevates $c$ values above $\Gamma_2$ along $\Gamma_1$ (see Figure~6), producing:
\[
\mathcal{N}_T = c(x^*)^3 - n^2D > 0 \Rightarrow \mathrm{Tr}(J^*) = \dfrac{\mathcal{N}_T}{nD} > 0,
\]
characterizing a repelling saddle-node.
\item Case $n > n^*$: Now $\dfrac{1}{n}d_c(n) < 0$ implies $d_c(n) < 0$. Increasing $n$ beyond $n^*$ lowers $c$ values below $\Gamma_2$ along $\Gamma_1$, resulting in:
\[
\mathcal{N}_T = c(x^*)^3 - n^2D < 0 \Rightarrow \mathrm{Tr}(J^*) = \dfrac{\mathcal{N}_T}{nD} < 0,
\]
corresponding to an attracting saddle-node.
\end{enumerate}

We now establish assertion (2) by demonstrating the emergence of a codimension-two Bogdanov-Takens singularity at $(c,n) = (c^*,n^*)$. At this critical parameter pair, both the Jacobian trace $\mathrm{Tr}(J^*)$ and determinant $\det(J^*)$ vanish simultaneously. Through higher-order jet calculations, we verify the non-degeneracy conditions $\zeta\eta \neq 0$ required for this bifurcation type.
\par
Translating the endemic equilibrium $(x^*,y^*)$ to the origin via $x \rightarrow x + x^*,\ y \rightarrow y + y^*$, where
\begin{align}\label{2.37}
x^*=am+\sqrt{a^2 m^2+3 m}, \quad y^*=\frac{1}{n}x^*,
\end{align}
we expand the system in Taylor series about $(0,0)$:
\begin{equation}\label{eq:2.38}
       \begin{aligned}
       \dfrac{{\rm d} x}{{\rm d}t}&=\xi_1x+\xi_2y+\xi_3x^2+\xi_4xy+\mathcal{O}(|(x,y)|^2),\\
         \dfrac{{\rm d} y}{{\rm d}t}&= x- ny,
\end{aligned}\end{equation}
with coefficients
\begin{equation}
\begin{aligned}\label{paraxi}
\xi_1&=n,\quad \xi_2=-n^2,\\
\xi_3&=-\frac{x^*(-2 a^2 m n-6 n+(9 b m n^*+12 c n+6c)x^* +4acn(x^*)^2+3 b n(x^*)^3)}{2n(1+a x^*+b (x^*)^3)^2},\\
\xi_4&=-\frac{c(x^*)^2(2 a x^*+3)}{(1+a x^*+b(x^*)^3)^2}.
\end{aligned}
\end{equation}
Applying the near-identity transformation
\begin{equation}\label{eq:2.39}
x\rightarrow x, \quad \xi_1 x+\xi_2 y \rightarrow y,
\end{equation}
we obtain the simplified system:
\begin{equation}\label{eq:2.40}
       \begin{aligned}
       \dfrac{{\rm d} x}{{\rm d}t}&=y+\Big(\xi_3-\dfrac{\xi_1\xi_4}{\xi_2}\Big)x^2+\dfrac{\xi_4}{\xi_2}xy+\mathcal{O}(|(x,y)|^2),\\
         \dfrac{{\rm d}y}{{\rm d}t}&= \xi_1\Big(\xi_3-\dfrac{\xi_1\xi_4}{\xi_2}\Big)x^2+\dfrac{\xi_1\xi_4}{\xi_2}xy+\mathcal{O}(|(x,y)|^2),
\end{aligned}\end{equation}
The linearization matrix now assumes the Jordan block form:
\begin{align*}A=\begin{pmatrix}
      0 & 1 \\
       0 & 0
\end{pmatrix},
\end{align*}
characteristic of Bogdanov-Takens singularities. A subsequent nonlinear coordinate shift
 \begin{align}\label{eq:2.41}
x\rightarrow x+\frac{\xi_4}{2\xi_2}x^2, \quad y\rightarrow y-\Big(\xi_3-\dfrac{\xi_1\xi_4}{\xi_2}\Big)x^2
\end{align}
yields the second-order normal form:
\begin{equation}\label{eq:2.42}
       \begin{aligned}
       \dfrac{{\rm d} x}{{\rm d}t}&=y,\\
         \dfrac{{\rm d} y}{{\rm d}t}&= \zeta\xi_1 x^2+\eta xy+\mathcal{O}(|(x,y)|^3),
\end{aligned}\end{equation}
where $\zeta=\xi_3-\dfrac{\xi_1\xi_4}{\xi_2}, \eta=2\xi_3-\dfrac{\xi_1\xi_4}{\xi_2}$.\par
\end{proof}

\vspace{1em}
\subsection{Proof of Theorem \ref{th:4.3}}\label{proof_th_4.3}
\begin{proof} One needs to show that there exist a sequence of $C^\infty$ diffeomorphisms and time rescaling(preserving orientations of orbits) so that  system \eqref{eq:4.3} is $\mathcal{K}_{un}$-equivalent  to the normal form \eqref{eq:4.4}, and condition \eqref{eq:4.5} is fulfilled.\par
Making the transform  $ x\rightarrow x+x^*, y\rightarrow y+y^*$ to translate the positive equilibrium $(x^*,y^*)$ to origin and  taking Taylor series about $(x,y) = (0, 0)$ up to order 2, one shows
\begin{equation}\label{eq:3.4}
       \begin{aligned}
       \dfrac{{\rm d} x}{{\rm d}t}&=d_1+d_2x+d_3y+d_4x^2+d_5xy,\\
         \dfrac{{\rm d}y}{{\rm d}t}&= d_6+ x+d_7y,
\end{aligned}\end{equation}
\begin{equation}\label{eq:3.5}
\begin{split}
d_1&=-\frac{n^*(n^*+1)\epsilon_1x^*}{c^*},\\
d_2&=n^*-\frac{\epsilon_1(x^*)^3}{n^*(1+a x^*+b(x^*)^3)^2}\Big(b^* n^* (x^*)^3 +3 a n^*x^*+2ax^*+4n^*+3\Big),\\
d_3&=-\Big(1+\dfrac{\epsilon_1}{c^*}\Big)(n^*)^2,\\
d_4&=-\frac{ \left(x^*\right)^2\epsilon_1 }{n^* \left(1+a x^*+b \left(x^*\right)^3\right)^3}\Big(-a b \left(n^*+3\right) \left(x^*\right)^4-3 b \left(n^*+2\right) \left(x^*\right)^3 +a^2\\
&\quad\times\left(3 n^*+1\right) \left(x^*\right)^2+a \left(8 n^*+3\right) x^*+6 n^*+3\Big)-\frac{x^*}{2n^*(1+a x^*+b (x^*)^3)^2}\\
&\quad\times\Big(-2 a^2 m n^*-6 n^*+(9 b m n^*+12 c^* n^*+6c^*)x^* +4ac^*n^*(x^*)^2\\
&\quad+3 b n^*(x^*)^3\Big),\\
 d_5&=-\frac{\left(x^*\right)^2 \left(2 a x^*+3\right)(c^*+\epsilon_1)}{\left(1+a x^*+b \left(x^*\right)^3\right)^2},\quad d_6=-\frac{\epsilon_2}{n^*}x^*,\quad d_7=-(n^*+\epsilon_2).
\end{split}
\end{equation}
Letting
\begin{align}\label{eq:3.6}
x\rightarrow x, \quad d_1+d_2x+d_3y+d_4x^2+d_5xy\rightarrow y,
\end{align}
this transformation is invertible in some neighborhood of $(x,y)=(0,0)$ for small $|\epsilon|$ and
depends smoothly on the parameters. It is clear that if $\epsilon=0$, the origin $(x,y)=(0,0)$ is a fixed point of this map.
The transformation brings \eqref{eq:3.4} into
\begin{equation}\label{eq:3.7}
       \begin{aligned}
       \dfrac{{\rm d}x}{{\rm d}t}&=y,\\
         \dfrac{{\rm d}y}{{\rm d}t}&= e_1+ e_2x+e_3y+e_4x^2+e_5xy+e_6y^2+P(x,y,\epsilon),
\end{aligned}
\end{equation}
where
\begin{equation}\label{eq:3.8}
\begin{aligned}
e_1&=-d_1d_7+d_3d_6, \quad  e_2=d_3-d_2d_7+d_5d_6, \quad e_3=d_2+d_7-\dfrac{d_1d_5}{d_3},\\
e_4&=d_5-d_4d_7, \quad e_5=2d_4-\dfrac{d_2d_5}{d_3}+\dfrac{d_1d_5^2}{(d_3)^2}, \quad e_6=\dfrac{d_5}{d_3},\\
\end{aligned}
\end{equation}
where $b_i$ and $P(x,y,\epsilon)=\mathcal{O}(|(x,y)|^3)$ are smooth functions of their arguments.

Introduce the new time $\tau$ via the equation
\begin{align}\label{eq:3.9}
{\rm d}t=(1-e_6x){\rm d}\tau,
\end{align}
the direction of time is preserved near the origin for small $|\epsilon|$. One obtains (still denote $\tau$ by $t$)
\begin{equation}\label{eq:3.10}
       \begin{aligned}
       \dfrac{{\rm d}x}{{\rm d}t}&=(1-e_6x)y,\\
         \dfrac{{\rm d} y}{{\rm d}t}&= (1-e_6x)(e_1+ e_2x+e_3y+e_4x^2+e_5xy+e_6y^2+P(x,y,\epsilon)).
\end{aligned}\end{equation}
Taking a coordinate transformation $x\rightarrow x, y-e_6xy\rightarrow y$, which maps the origin into itself for arbitrary  $\epsilon$,  thus one can eliminate the $y^2$-term.  Consequently, one obtains
\begin{equation}\label{eq:3.11}
       \begin{aligned}
       \dfrac{{\rm d} x}{{\rm d}t}&=y,\\
         \dfrac{{\rm d}y}{{\rm d}t}&=f_1+ f_2x+f_3y+f_4x^2+f_5xy+P(x,y,\epsilon),
\end{aligned}\end{equation}
where
\begin{align}\label{eq:3.12}
f_1=e_1,\ f_2=e_2-2e_1e_6,\ f_3=e_3,\ f_4=e_4-2e_2e_6+e_1e_6^2,\ f_5=e_5-e_3e_6.
\end{align}
A transformation  of coordinates
in the $x$-direction
\begin{align}\label{eq:3.13}
  x+\frac{f_2}{2f_4}\rightarrow x, \ y\rightarrow y,
\end{align}
 which maps the origin $(x,y)=(0,0)$  to itself when $\epsilon=0$ and leads to the system
 the $y^2$-term.  Consequently, one obtains
\begin{equation}\label{eq:3.14}
       \begin{aligned}
       \dfrac{{\rm d}x}{{\rm d}t}&=y,\\
         \dfrac{{\rm d}y}{{\rm d}t}&=g_1+g_2y+g_3x^2+g_4xy+P(x,y,\epsilon),
\end{aligned}\end{equation}
where
\begin{align}\label{eq:3.15}
  g_1=f_1-\frac{f_2^2}{4f_4}, \ g_2=f_3-\frac{f_2f_5}{2f_4},\ g_3=f_4,\ g_4=f_5,
\end{align}
when $\epsilon=0$, $g_3=\zeta\neq0$ and $g_4=\eta\neq0$. One performs a scaling by introducing the transformation
\begin{align}\label{eq:3.16}
  \frac{g_4^2}{g_3}x\rightarrow x, \ \frac{g_4^3}{g_3^2}y\rightarrow y,\ \frac{g_3}{g_4}t\rightarrow t,
\end{align}
the variables and time scalings  above will be well defined under the assumption $\zeta\eta\neq0$, the transformation brings \eqref{eq:3.14} into
\begin{equation}\label{eq:3.17}
       \begin{aligned}
       \dfrac{{\rm d}x}{{\rm d}t}&=y,\\
         \dfrac{{\rm d}y}{{\rm d}t}&= \mu_1(\epsilon)+\mu_2(\epsilon)y+ x^2+xy,
\end{aligned}\end{equation}
where
\begin{align}\label{eq:3.18}
  \mu_1=\dfrac{g_1g_4^4}{g_3^3}, \ \mu_2=\dfrac{g_2g_4}{g_3}.
\end{align}
Compositing a sequence of $C^\infty$ diffeomorphisms above, $\mu_1,\mu_2$ can be expressed as the functions of $\epsilon_1$ and $\epsilon_2$ as follows
\begin{equation}\label{eq:3.19}
\begin{split}
\mu_1(\epsilon_1,\epsilon_2)&=r_1\epsilon_1+r_2\epsilon_2+r_3\epsilon_1^2+r_4\epsilon_1\epsilon_2+r_5\epsilon_2^2+\mathcal{O}(|(\epsilon_1,\epsilon_2|^3),\\
\mu_2(\epsilon_1,\epsilon_2)&=s_1\epsilon_1+s_2\epsilon_2+s_3\epsilon_1^2+s_4\epsilon_1\epsilon_2+s_5\epsilon_2^2+\mathcal{O}(|(\epsilon_1,\epsilon_2|^3),
\end{split}
\end{equation}
where
\begin{equation}\label{eq:3.20}
\begin{split}
r_1&=-\frac{(n^*+1)x^*\eta^4}{c^*n^* \zeta^3},\quad r_2=\frac{x^*\eta^4}{ (n^*)^2\zeta^3},\\
r_3&=\dfrac{\eta ^3}{4 (c^*)^2 \zeta ^4(n^*)^2(a x^*+b(x^*)^3+1)^2}\bigg(((n^*)^2 (\eta (43 a^2(x^*)^2-2 a x^*(b (x^*)^3-65)\\
&\quad-b^2 (x^*)^6-2 b (x^*)^3+98)+4 c^* \Upsilon  x^* (8 \zeta -3 \eta )(a x^*+b(x^*)^3+1)^2)+(n^*)^4\\
&\quad\times(-(32 \zeta (2 a x^*+3) (a x^*+b(x^*)^3+1)-\eta(95 a^2 (x^*)^2+10 a x^*(5 b (x^*)^3+26)\\
&\quad-b^2 (x^*)^6+76 b (x^*)^3+176)))+2 c^* n^* x^* (2 \Upsilon  (8 \zeta -3 \eta )(a x^*+b (x^*)^3+1)^2\\
&\quad+\eta (x^*)^2(13 a x^*-b(x^*)^3+20))-2 (n^*)^3 (2 a x^*+3)(16 \zeta(a x^*+b(x^*)^3+1)\\
&\quad-\eta(35 a x^*+13 b(x^*)^3+46))+14 c^* \eta(x^*)^3 (2 a x^*+3))\bigg),
\end{split}
\end{equation}
\begin{equation*}
\begin{split}
r_4&=\dfrac{\eta ^3 }{2 c^* \zeta ^4(n^*)^3(1+a x^*+b(x^*)^3)^3}\Big((-c^* \eta  x^*(1+a x^*+b (x^*)^3) ((x^*)^2(2 n^*(a^2 \Upsilon +19)\\
&\quad+63)+4 a b \Upsilon  n^* (x^*)^4+(x^*)^3 (a(25 n^*+42)+4 b \Upsilon  n^*)+4 a \Upsilon  n^* x^*+2 b^2 \Upsilon  n^*(x^*)^6\\
&\quad-b n^*(x^*)^5+2 \Upsilon  n^*)+c^*(2 a x^*+3)(4(n^*)^2 ((x^*)^2 (2 a^2 \Upsilon +3)+4 a b \Upsilon (x^*)^4+2(x^*)^3 \\
&\quad\times(a+2 b \Upsilon )+4 a \Upsilon  x^*+2 b^2 \Upsilon (x^*)^6+2 \Upsilon)+n^*(x^*)^2(a x^*(8-27 \eta  x^*)-5 b \eta  (x^*)^4\\
&\quad-38 \eta  x^*+12)-11 \eta (x^*)^3(2 a x^*+3))+\eta  n^*(a x^*+b(x^*)^3+1)^3(2 \eta  n^* x^*+n^*+3 \eta  x^*)\\
&\quad-8 c^*(n^*)^2(x^*)^2 (2 a x^*+3)^2)\Big),\\
r_5&=-\dfrac{\eta ^3}{4 \zeta ^4x^*(n^*)^4(a x^*+b(x^*)^3+1)^2}\Big(\eta  n^*x^*(a x^*+b (x^*)^3+1)^2 (n^*+6 \eta  x^*)\\
&\quad-28 c^* \eta(x^*)^4 (2 a x^*+3)+(n^*)^2 (2 a x^*+3)^2 (8 n^*-11 \eta  x^*)\Big),\\
s_1&=\dfrac{\eta (\eta  x^* (-2-a x^*+b (x^*)^3)+(2 a x^*+3) ((n^*)^2-\eta  n^* x^*))}{2 c^* \zeta ^2 x^* (1+a x^*+b(x^*)^3)},\\
s_2&=\frac{\eta  \left(\eta  x^* (\eta -2 \zeta )-n^* (2 \zeta +\eta )\right)}{2 \zeta ^2 \left(n^*\right)^2},\\
s_3&=\dfrac{(2 a x^*+3)}{2c^* \zeta ^3 n^* (1+a x^*+b(x^*)^3)^4}\Big( \Big(\eta \left(n^* x^*(3 a x^*-b(x^*)^3+5\right)+2 x^*\big(a x^*-b (x^*)^3+2)\big)\\
&\quad-(n^*)^2 (2 a x^*+3)\Big)\Big(-c^* \Upsilon  x^*(1+a x^*+b(x^*)^3)+\eta(n^* x^*+x^*)(3 a x^*+b(x^*)^3+4)\\
&\quad+(n^*)^2(2 a x^*+3)\Big)\Big),
\end{split}
\end{equation*}
\begin{equation}
\begin{split}
s_4&=-\dfrac{1}{4 c^* \zeta ^3 (n^*)^4 (1+a x^*+b(x^*)^3)^4}\bigg(-c^* (n^*)^2 x^* (2 a x^*+3)^2 (c^* (2 \Upsilon  n^* (a x^*+b (x^*)^3\\
&\quad+1)+\eta  x^* (4 a \Upsilon x^*+4 b \Upsilon (x^*)^3+4 \Upsilon +7 (x^*)^2))-2 \eta ^2 x^* (n^* (9 a x^*+b (x^*)^3+13)\\
&\quad+8 a x^*+12))+(c^*)^2 \eta (x^*)^2 (2 a x^*+3) (6 \Upsilon (n^*)^2 (a x^*+b(x^*)^3+1)^2+n^* (x^*)^2 (2 a x^* \\
&\quad\times(4-3 \eta  x^*)-8 b \eta  (x^*)^4-5 \eta  x^*+12)+2 \eta (x^*)^3 (2 a x^*+3))-c^* \eta ^2 n^*(x^*)^2 (2 a x^*+3)\\
&\quad\times(a x^*+b(x^*)^3+1)^2 (n^*(\eta  x^*+14)+2 \eta  x^*)+2 \eta ^3 (n^*)^3 (1+a x^*+b(x^*)^3)^4\\
&\quad+c^* (n^*)^4 (2 a x^*+3)^3(3 n^*+2 \eta  x^*)\bigg),\\
s_5&=\dfrac{1}{4 \zeta ^3 (n^*)^4 (1+a x^*+b (x^*)^3)^4}\bigg(-c^* \eta ^2(x^*)^2(2 a x^*+3)(a x^*+b(x^*)^3+1)^2\\
&\quad\times(13 n^*+2 \eta  x^*)+c^* \eta  n^* x^* (2 a x^*+3)^2(a x^*+b(x^*)^3+1)(7 n^*+2 \eta  x^*)\\
&\quad+3 \eta ^3 (n^*)^2(1+a x^*+b (x^*)^3)^4-c^* n^*(2 a x^*+3)^3((n^*)^2-4 \eta ^2 (x^*)^2)\bigg),
\end{split}
\end{equation}
here
\begin{align*}
\Upsilon:&=\frac{\left(x^*\right)^2 }{n^* \left(a x^*+b \left(x^*\right)^3+1\right)^3}\Big(\left(x^*\right)^2 \left(3 a^2 n^*+a^2\right)+\left(x^*\right)^4 \left(-a b n^*-3 a b\right)+x^* \left(8 a n^*+3 a\right)\\
&\quad+\left(x^*\right)^3 \left(-3 b n^*-6 b\right)+6 n^*+3\Big).
\end{align*}
Hence
\begin{equation}
\begin{split}
\det\left(\frac{\partial(\mu_1,\mu_2)}{\partial(\epsilon_1,\epsilon_2)}\right)\bigg|_{\epsilon=0}=&\dfrac{(2 n^*+1) x^*\eta^5}{c^* (n^*)^2\zeta^4}\neq0.
\end{split}
\end{equation}
The above transversality condition ensures $(\mu_1,\mu_2)$ span the parameter space near $(c^*,n^*)$.
The results in Bogdanov \cite{Bogdanov1975,Bogdanov} and Takens \cite{Takens1973} now imply that system \eqref{eq:3.17} (i.e., \eqref{eq:4.4} or
\eqref{eq:4.3}) undergoes Bogdanov-Takens bifurcation when $(\mu_1,\mu_2)$ changes in a small neighborhood
of $(0,0)$.
\end{proof}

\vspace{1em}
\subsection{Proof of Theorem \ref{thm:cusp-equivalence}}
\label{proof-thm:cusp-equivalence}
\begin{proof}
We begin by applying a change of variables to translate the positive equilibrium \((x^*, y^*)\) to the origin. The transformation is defined as:
\[
x \rightarrow x + x^*, \quad y \rightarrow y + y^*,
\]
and we take the Taylor expansion of system (2.4) around the origin. This yields the following system:
\begin{equation}
\begin{aligned}
\dfrac{{\rm d} x}{{\rm d}t} &= \xi_1 x + \xi_2 y + \xi_3 x^2 + \xi_4 x y + \xi_5 x^3 + \xi_6 x^2 y + \xi_7 x^4 + \xi_8 x^3 y + \mathcal{O}(|(x,y)|^2), \\
\dfrac{{\rm d} y}{{\rm d}t} &= x - \xi_1 y.
\end{aligned}
\end{equation}
The coefficients \(\xi_5, \xi_6, \xi_7, \xi_8\) are given by:
\[\begin{aligned}
\xi_5&=-\frac{1}{n (a x^*+b(x^*)^3+1)^3}\Big((x^*)^3 (a^2 c n+13 b n)-a^2 m n+(x^*)^5 (-4 a b^2 m n\\
&\quad-7 a b c n-4 a b c)+(x^*)^4 (4 a b n-19 b^2 m n-19 b c n-13 b c)+(x^*)^2 (-4 a b m n\\
&\quad+2 a c n-a c)+x^* (a n+4 b m n+4 c n+c)-4 b^2 n (x^*)^6+(x^*)^7 (4 b^3 m n+4 b^2 c n\\
&\quad+4 b^2 c)-n\Big),\\
\xi_6&=-\frac{c x^*(a^2(x^*)^2-3 a b(x^*)^4+3 a x^*-6 b(x^*)^3+3)}{(a x^*+b (x^*)^3+1)^3},\\
\xi_7&=-\frac{1}{n(a x^*+b(x^*)^3+1)^4}\Big(a^3 m n+(x^*)^5(-a^2 b^2 m n-5 a^2 b c n-a^2 b c-35 b^2 n)\\
&\quad+(x^*)^4 (a^2 b n+5 a b^2 m n-14 a b c n)+(x^*)^2 (5 a^2 b m n+a^2 c n+a^2 c+14 b n)\\
&\quad+x^*(a^2 (-n)-5 a b m n-a c n-a c)+(x^*)^7 (10 a b^3 m n+14 a b^2 c n+10 a b^2 c)\\
&\quad+(x^*)^6(-10 a b^2 n+45 b^3 m n+45 b^2 c n+35 b^2 c)+a n+5 b^3 n (x^*)^8+(x^*)^3\\
 &\quad\times(-30 b^2 m n-30 b c n-14 b c)+(x^*)^9(-5 b^4 m n-5 b^3 c n-5 b^3 c)+b m n+c n\Big),\\
\xi_8&=\frac{c (2 b(x^*)^3 (2 a^2(x^*)^2+7 a x^*+8)-2 b^2 (x^*)^6 (2 a x^*+5)-1)}{(a x^*+b(x^*)^3+1)^4}.
\end{aligned}\]
Next, we apply a near-identity transformation to simplify the system. Following standard techniques for cusp bifurcations, we define:
\[x \rightarrow x, \quad \xi_1 x + \xi_2 y \rightarrow y,
\]
resulting in:
\begin{equation}
\begin{aligned}
\dfrac{{\rm d} x}{{\rm d}t} &= y + a_1 x^2 + a_2 x y + a_3 x^3 + a_4 x^2 y + a_5 x^4 + a_6 x^3 y + \mathcal{O}(|(x,y)|^4), \\
\dfrac{{\rm d} y}{{\rm d}t} &= n(a_1 x^2 + a_2 x y + a_3 x^3 + a_4 x^2 y + a_5 x^4 + a_6 x^3 y) + \mathcal{O}(|(x,y)|^4),
\end{aligned}
\end{equation}
where the coefficients \(a_1, \ldots, a_6\) satisfy:
\[
\begin{aligned}
a_1 &= \xi_3 - \dfrac{\xi_1 \xi_4}{\xi_2}, \quad a_2 = \dfrac{\xi_4}{\xi_2}, \quad a_3 = \xi_5 - \dfrac{\xi_1 \xi_6}{\xi_2}, \\
a_4 &= \dfrac{\xi_6}{\xi_2}, \quad a_5 = \xi_7 - \dfrac{\xi_1 \xi_8}{\xi_2}, \quad a_6 = \dfrac{\xi_8}{\xi_2}.
\end{aligned}
\]
The condition \(\xi_4 \neq -n \xi_3\) ensures \(a_1 \neq 0\). Applying the nonlinear transformation:
\[
x \rightarrow \frac{1}{n a_1} \left( x + \frac{a_2}{2 a_1 n} x^2 + \frac{2 a_4 - a_2^2}{6 a_1^2 n^2} x^3 + \frac{-a_2^3 - 4 a_4 a_2 + 6 a_6}{24 a_1^3 n^3} x^4 \right),
\]
\[
y \rightarrow \frac{1}{n a_1} \left( y - \frac{1}{n} x^2 - \frac{a_3}{a_1^2 n^2} x^3 + \frac{7 a_1 a_2^2 - 6 a_3 a_2 + 4 a_1 a_4 - 12 a_5}{12 a_1^3 n^3} x^4 \right),
\]
we reduce the system to:
\begin{equation}
\begin{aligned}\label{eq:3.31}
\dfrac{{\rm d}x}{{\rm d}t} &= y, \\
\dfrac{{\rm d} y}{{\rm d}t} &= x^2 + c_1 x y + c_2 x^3 + c_3 x^2 y + c_4 x^4 + c_5 x^3 y + \mathcal{O}(|(x,y)|^4),
\end{aligned}
\end{equation}
where the coefficients \(c_1, c_2, c_3, c_4, c_5\) are given by:
\begin{align*}
c_1 &= \dfrac{2 a_1 + n a_2}{n a_1}, \quad c_2 = \dfrac{a_3}{n a_1^2}, \quad c_3 = \dfrac{a_2^2 n + 2 a_4 n + 4 a_1 a_2 + 6 a_3}{2 a_1^2 n^2}, \\
c_4 &= \dfrac{6(a_2 a_3 + 2 a_5) - a_1 (7 a_2^2 + 4 a_4)}{12 a_1^3 n^2},\\
 c_5& = \dfrac{a_2^3 (-n) + a_2(8 a_4 n + 30 a_3) + 6(a_6 n + 4 a_5) + a_1(4 a_4 - 8 a_2^2)}{6 a_1^3 n^3}.
\end{align*}
Under the condition \(\xi_4 = -2n\xi_3\), we have \(c_1 = 0\). By Lemma \ref{lem:4.6}, this system is \(C^\infty\)-equivalent to:
\begin{align}\label{3.32}
\begin{aligned}
\dfrac{{\rm d}x}{{\rm d}t} &= y, \\
\dfrac{{\rm d} y}{{\rm d}t} &= x^2 + \chi x^3 y + \mathcal{O}(|(x,y)|^4),
\end{aligned}
\end{align}
where \(\chi = c_5 - c_2 c_3\). Substituting the expressions for \(a_i\) in terms of \(\xi_j\), we obtain:
\begin{align}\label{3.33}
\chi=-\frac{1}{n^5 \xi _3^4}\Big(3 n^2 \xi _5^2+\xi _3^2(8 n \xi _5+6 \xi _6)+n \xi _3(4 n \xi _7+3 \xi _8)+5 n \xi _5 \xi _6+4 \xi _3^4+2 \xi _6^2\Big),
\end{align}
which completes the proof.
\end{proof}

\vspace{1em}
\subsection{Proof of Theorem \ref{thm:cusp_versal}}
\label{Proof-thm:cusp_versal}

\begin{proof}
We demonstrate the existence of orientation-preserving $C^\infty$ diffeomorphisms and time rescalings that transform system \eqref{4.12} into the normal form \eqref{4.13} while satisfying the transversality condition \eqref{4.14}. The proof proceeds through three coordinate transformations and parameter rescalings.

First, translate the equilibrium $E^*(x^*,y^*)$ to the origin via the shift
\[
(X_1, Y_1) = \left(x - \left(am + \sqrt{a^2m^2 + 3m}\right),\ y - \frac{1}{n}\left(am + \sqrt{a^2m^2 + 3m}\right)\right),
\]
and expand the translated system \eqref{4.12} in Taylor series about $(0,0)$ up to fourth order. Applying a $C^\infty$ near-identity transformation to eliminate non-resonant terms yields the reduced system
\[
\begin{cases}
\dot{X}_1 = Y_1, \\
\dot{Y}_1 = \sum\limits_{i=0}^4 \alpha_{i0}X_1^i + y_1 \sum\limits_{i=0}^3 \alpha_{i1}X_1^i + R(X_1, Y_1, \epsilon),
\end{cases}
\]
where $R(X_1,Y_1,\epsilon)$ shares the analytic properties of $R(x,y,\epsilon)$ in \eqref{4.13}, subsequent $R$-terms will inherit these properties through each transformation stage. Moreover, coefficients $\alpha_{i0}$ and $\alpha_{i1}$ in the reduced system are as follows:
\begin{align*}
\alpha _{00}=& \epsilon _2 \left(n x^*\right)-\frac{\epsilon _1 \left((2\vartheta +3)^2 m^2 (n+1)\right)}{D},\\
\alpha _{10}=&\frac{\epsilon _2 (n (2 \vartheta +D+3))}{D}-\frac{\epsilon _1 \left((2 \vartheta +3)^2 m^2 (n+1) (2 \vartheta +D+3)\right)}{D x^*},\\
\alpha _{20}=&\frac{\epsilon _1 \left((2 \vartheta +3) m (n+1) \left((\vartheta +3) D-(2 \vartheta +3)^2\right)\right)}{D^3}+\frac{(\vartheta +3) (2 \vartheta +3) m^2 n \epsilon _3}{D^2}\\
&-\frac{\epsilon _2 \left((\vartheta +3) D (m+n)-(2 \vartheta +3)^2 n\right)}{D^2 x^*}+\frac{(-(\vartheta +3)) m n}{D x^*},\\
\alpha _{30}=&\frac{m n x^* \epsilon _3 (2 (\vartheta +3) (2 \vartheta +3)-(\vartheta +4) D)}{D^3}+\frac{n ((\vartheta +4) D-(\vartheta +3) (2 \vartheta +3))}{(2 \vartheta +3) D^2}\\
&-\frac{(n+1) x^* \epsilon _1 \left((\vartheta +4) D^2-(2 \vartheta +3) (4 \vartheta +9) D+2 \vartheta  (2 \vartheta +3) (4 m+9)+27\right)}{D^4}\\
&+\frac{\epsilon _2}{D^2 m^2 (n+1)(m+n^2+n)}\Big(m^3 (-12 \vartheta +6 D+n-17)-m^2 n (n+1) (2 \vartheta\\
& -6 D+(\vartheta +3) n+5)-3 m n^2 (n+1)^2+(2 \vartheta +3) n^3 (n+1)^3\Big),\\
\alpha _{40}=&-\frac{n}{(2 \vartheta +3) D^3 x^*} \left((\vartheta +3) (2 \vartheta +3)^2+(\vartheta +5) D^2-\left(5 \vartheta ^2+26 \vartheta +30\right) D\right)\\
&+\frac{m n \epsilon _3 }{D^4}\left(3 (\vartheta +3) (2 \vartheta +3)^2+(\vartheta +5) D^2-2 \left(5 \vartheta ^2+26 \vartheta +30\right) D\right)\\
&+\frac{(n+1) \epsilon _1}{D^5}\times\Big(-(2 \vartheta +3)^4+(\vartheta +5) D^3-(13 \vartheta ^2+58 \vartheta +60) D^2+(7 \vartheta +15) \\
&\times(2 \vartheta +3)^2 D\Big)-\frac{n \epsilon _2}{2 (\vartheta +3) D^2 m^2 (n+1) x^*}\Big(m (39 \vartheta ^2+118 \vartheta +10 D^2+D (-40 \vartheta \\
&+(\vartheta +5) n-61)-(3 \vartheta ^2+17 \vartheta +21) n+87)+n (n+1) (10 D^2-(25 \vartheta +46) D\\
&-(2 \vartheta +3) ((\vartheta +3) n-3 (\vartheta +2)))\Big),\\
\alpha _{01}=& -\frac{n^3 \epsilon _3}{c}-\frac{n^2 \epsilon _1}{c}-\epsilon _2,\\
\alpha _{11}=&\frac{-\epsilon _1(2 \vartheta +3)^2 m)}{D^2}-\frac{2 (2 \vartheta +3) m^2 \epsilon _3 (6 (\vartheta +1) (2 \vartheta +3)-(11 \vartheta +15) D)}{D^3},\\
\alpha _{21}=&\frac{1}{(2 \vartheta +3) D^3 (n+1)}\Big((2 \vartheta +3)^3+3 D^2 (2 (\vartheta +3)+(\vartheta +4) n)-(2 \vartheta +3) D (7 \vartheta \\
&+2 (\vartheta +3) n+15)\Big)-\frac{m x^* \epsilon _3}{D^4 (n+1)}\Big(3 (2 \vartheta +3)^3+3 D^2 (3 \vartheta +(\vartheta +4) n+8)\\
&-(2 \vartheta +3) D (16 \vartheta +4 (\vartheta +3) n+33)\Big)+\frac{x^* \epsilon _1 \left(3 (\vartheta +2) D-(2 \vartheta +3)^2\right)}{D^3},\\
\alpha _{31}=& \frac{1}{(2 \vartheta +3) D^4 (n+1) x^*}\Big((2 \vartheta +3)^4-2 D^3 (4 \vartheta +2 (\vartheta +5) n+15)+D^2 (32 \vartheta ^2+137 \vartheta \\
&+3 (4 \vartheta ^2+21 \vartheta +24) n+138)-(2 \vartheta +3)^2 D (10 \vartheta +2 (\vartheta +3) n+21)\Big)+\frac{2 m \epsilon _3}{D^5 (n+1)} \\
&\times\Big(-2 (2 \vartheta +3)^4+2 D^3 (3 \vartheta +(\vartheta +5) n+10)-D^2(2(19 \vartheta ^2+79 \vartheta +78)+3 (4 \vartheta ^2\\
&+21 \vartheta +24) n)+(2 \vartheta +3)^2 D (16 \vartheta +3 (\vartheta +3) n+33)\Big)+\frac{\epsilon _1 }{D^5 n}\Big(2 D^3 (10 (\vartheta +3)\\
&+(8 \vartheta +25) n)-2 D^2 (2 (5 \vartheta +8) (5 \vartheta +12)+(44 \vartheta ^2+179 \vartheta +174) n)+(2 \vartheta +3)^2 D\\
 &\times(36 (\vartheta +2)+(34 \vartheta +69) n)-4 (2 \vartheta +3)^4 (n+1)\Big).
\end{align*}

\vspace{1em}
\noindent\textbf{Step 1: Reduction of cubic and quartic terms.}
Consider the formal power series solution to the auxiliary equation
\begin{equation}\label{eq:3.37}
\sum_{i=2}^4 \alpha_{i0}X_1^i {\rm d}X_1 = \alpha_{20}X_2^2 {\rm d}X_2,
\end{equation}
assuming the expansion
\begin{equation}\label{eq:3.38}
X_1 = k_1X_2 + k_2X_2^2 + k_3X_2^3 + k_4X_2^4 + \mathcal{O}(X_2^5).
\end{equation}
Differentiation yields
\begin{equation}\label{eq:3.39}
{\rm d}X_1 = \left(k_1 + 2k_2X_2 + 3k_3X_2^2 + 4k_4X_2^3\right){\rm d}X_2 + \mathcal{O}(X_2^4).
\end{equation}
Substituting \eqref{eq:3.38}-\eqref{eq:3.39} into \eqref{eq:3.37} and matching coefficients (using $\alpha_{20}(0) \neq 0$) produces
\[
k_1 = 1,\quad k_2 = -\frac{\alpha_{30}}{4\alpha_{20}},\quad k_3 = \frac{3\alpha_{30}^2}{16\alpha_{20}^2} - \frac{\alpha_{40}}{5\alpha_{20}},\quad k_4 = \frac{7\alpha_{30}\alpha_{40}}{20\alpha_{20}^2} - \frac{35\alpha_{30}^3}{192\alpha_{20}^3}.
\]
Implementing the transformation
\[
\begin{aligned}
X_1 &= X_2 - \frac{\alpha_{30}}{4\alpha_{20}}X_2^2 + \left(\frac{3\alpha_{30}^2}{16\alpha_{20}^2}-\frac{\alpha_{40}}{5\alpha_{20}}\right)X_2^3 + \left(\frac{7\alpha_{30}\alpha_{40}}{20\alpha_{20}^2}-\frac{35\alpha_{30}^3}{192\alpha_{20}^3}\right)X_2^4 + \mathcal{O}(X_2^5), \\
Y_1 &= Y_2, \\
{\rm d}t &= {\rm d}\tau\left(1 - \frac{\alpha_{30}}{2\alpha_{20}}X_2 + 3\left(\frac{3\alpha_{30}^2}{16\alpha_{20}^2}-\frac{\alpha_{40}}{5\alpha_{20}}\right)X_2^2 + 4\left(\frac{7\alpha_{30}\alpha_{40}}{20\alpha_{20}^2}-\frac{35\alpha_{30}^3}{192\alpha_{20}^3}\right)X_2^3\right),
\end{aligned}
\]
we obtain the simplified system
\[
\begin{cases}
\dot{X}_2 = Y_2, \\
\dot{Y}_2 = \sum\limits_{i=0}^2 \beta_{i0}X_2^i + y_2 \sum\limits_{i=0}^3 \beta_{i1}X_2^i + R(X_2,Y_2,\epsilon),
\end{cases}
\]
where $\beta_{30},\beta_{40} = \mathcal{O}(\epsilon)$ justify absorption into $R(X_2,Y_2,\epsilon)$. Coefficients $\beta_{ij}$ are listed as follows.
\begin{align*}
\beta_{00}=&\alpha _{00},\beta_{10}=\frac{2 \alpha _{10} \alpha _{20}-\alpha _{30} \alpha _{00}}{2 \alpha _{20}},\beta_{20}=\frac{80 \alpha _{20}^3-60 \alpha _{10} \alpha _{30} \alpha _{20}-48 \alpha _{40} \alpha _{20} \alpha _{00}+45 \alpha _{30}^2 \alpha _{00}}{80 \alpha _{20}^2},\\
\beta_{30}=&\frac{210 \alpha _{10} \alpha _{20} \alpha _{30}^2-192 \alpha _{10} \alpha _{20}^2 \alpha _{40}-175 \alpha _{30}^3 \alpha _{00}+336 \alpha _{20} \alpha _{40} \alpha _{30} \alpha _{00}}{240 \alpha _{20}^3},\\
\beta_{40}=&\frac{\alpha _{10}(96 \alpha _{20} \alpha _{30} \alpha _{40}-55 \alpha _{30}^3)}{48 \alpha _{20}^3},\beta_{01}=\alpha _{01},\beta_{11}=\frac{2 \alpha _{11} \alpha _{20}-\alpha _{30} \alpha _{01}}{2 \alpha _{20}},\\
\beta_{21}=&\frac{80 \alpha _{21} \alpha _{20}^2-60 \alpha _{11} \alpha _{30} \alpha _{20}-48 \alpha _{40} \alpha _{20} \alpha _{01}+45 \alpha _{30}^2 \alpha _{01}}{80 \alpha _{20}^2},\\
\beta_{31}=&\frac{1}{240 \alpha _{20}^3}\Big(240 \alpha _{31} \alpha _{20}^3-240 \alpha _{21} \alpha _{30} \alpha _{20}^2-192 \alpha _{11} \alpha _{40} \alpha _{20}^2+210 \alpha _{11} \alpha _{30}^2 \alpha _{20}+336 \alpha _{30} \alpha _{40}\\
&\times \alpha _{20} \alpha _{01}-175 \alpha _{30}^3 \alpha _{01}\Big).
\end{align*}

\vspace{1em}
\noindent\textbf{Step 2: Elimination of linear and quadratic terms.}
Using $\beta_{20} = -\dfrac{(\delta+3)mn}{D^3x^*} + \mathcal{O}(\epsilon) < 0$, remove the $X_2$-term via
\[
X_3 = X_2 + \frac{\beta_{10}}{2\beta_{20}},\quad Y_3 = \frac{Y_2}{-\sqrt{-\beta_{20}}},\quad \tau = -\sqrt{-\beta_{20}}t,
\]
yielding
\[
\begin{cases}
\dot{X}_3 = Y_3, \\
\dot{Y}_3 = \gamma_{00} + X_3^2 + Y_3\left(\gamma_{01} + \gamma_{11}X_3 + \gamma_{21}X_3^2 + \gamma_{31}X_3^3\right) + R(X_3,Y_3,\epsilon).
\end{cases}
\]
Subsequent elimination of $X_3^2Y_3$ through
\[
X_3 = X_4,\quad Y_3 = Y_4 + \frac{\gamma_{21}}{3}Y_4^2,\quad {\rm d}\tau = \left(1 + \frac{\gamma_{21}}{3}Y_4\right){\rm d}t,
\]
produces
\[
\begin{cases}
\dot{X}_4 = Y_4, \\
\dot{Y}_4 = \nu_{00} + X_4^2 + Y_4\left(\nu_{01} + \nu_{11}X_4 + \nu_{31}X_4^3\right) + R(X_4,Y_4,\epsilon),
\end{cases}
\]
with $\nu_{ij}$ coefficients expressed as follows:
\[\begin{aligned}
\nu_{00} = &\gamma_{00},\quad \nu_{01} =\gamma_{01} - \gamma_{00}\gamma_{21},\quad \nu_{11} = \gamma_{11}, \quad\nu_{31} = \gamma_{31},\\
\nu_{02} =& -\dfrac{\gamma_{21}(\gamma_{00}\gamma_{21} + 6\gamma_{01})}{9},\quad \nu_{03} = -\dfrac{2\gamma_{21}^2(\gamma_{00}\gamma_{21} + 3\gamma_{01})}{27}, \quad \nu_{12} = -\dfrac{2\gamma_{21}\gamma_{11}}{3},
\end{aligned}
\]
and
\begin{align*}
\gamma_{00}=&\frac{4 \beta _{20} \beta _{00}-\beta _{10}^2}{4 \beta _{20}^2},\,\gamma_{01}=-\frac{\sqrt{-\beta _{20}} \left(-\beta _{31} \beta _{10}^3+2 \beta _{20} \beta _{21} \beta _{10}^2-4 \beta _{11} \beta _{20}^2 \beta _{10}+8 \beta _{20}^3 \beta _{01}\right)}{8 \beta _{20}^4},\\
\gamma_{11}=&-\frac{\sqrt{-\beta _{20}} \left(3 \beta _{31} \beta _{10}^2-4 \beta _{20} \beta _{21} \beta _{10}+4 \beta _{11} \beta _{20}^2\right)}{4 \beta _{20}^3},\,\gamma_{21}=-\frac{\sqrt{-\beta _{20}} \left(2 \beta _{20} \beta _{21}-3 \beta _{10} \beta _{31}\right)}{2 \beta _{20}^2},\\
\gamma_{31}=&-\frac{\sqrt{-\beta _{20}} \beta _{31}}{\beta _{20}}.
\end{align*}

\vspace{1em}
\noindent\textbf{Step 3: Final parameter scaling.}
Using $\nu_{31}(0) = \varsigma < 0$ where
\[
\varsigma = -\frac{2\sqrt{2}\vartheta(5\vartheta+12)}{(\vartheta+3)(2\vartheta+3)^2Dx^*\sqrt{\frac{(2\vartheta+3)\left((2\vartheta+3)^2-(4\vartheta+9)D\right)^2}{(\vartheta+3)^2D^3x^*}}},
\]
rescale variables via
\[
X_4 = \nu_{31}^{-2/5}X_5,\quad Y_4 = \nu_{31}^{-3/5}Y_5,\quad t = \sqrt[5]{\nu_{31}}\tau,
\]
to obtain system \eqref{4.13} with parameters
\begin{align*}
\mu_1 &= \nu_{31}^{4/5}\nu_{00} = \frac{(2\vartheta+3)^2m(n+1)x^*}{(\vartheta+3)n}\epsilon_1 - \frac{D(x^*)^2}{m(\vartheta+3)}\epsilon_2 + \mathcal{O}(\epsilon^2), \\
\mu_2 &=\nu_{31}^{\frac{1}{5}} \nu _{01}\\
 &=-\frac{1}{c D^4 x^* (\frac{(\vartheta +3) m n}{D x^*})^{3/2}}\bigg(m (c (2 \vartheta +3) m x^* ((2 \vartheta +3)^3+3 D^2 (2 (\vartheta +3)+(\vartheta +4) n)\\
&\quad-(2 \vartheta +3) D (7 \vartheta +2 (\vartheta +3) n+15))+(\vartheta +3) D^3 n^3)\bigg)\epsilon_1-n^3\Big/\Big(c \sqrt{\frac{(\vartheta +3) m n}{D x^*}}\Big)\epsilon_3\\
&\quad+\frac{1}{(2 \vartheta +3) D^3 (n+1) x^* (\frac{(\vartheta +3) m n}{D x^*})^{3/2}}\bigg(n ((x^*)^2 ((2 \vartheta +3)^3+3 D^2 (2 (\vartheta +3)+(\vartheta +4) n)\\
&\quad-(2 \vartheta +3) D (7 \vartheta +2 (\vartheta +3) n+15))-(2 \vartheta ^2+9 \vartheta +9) D^2 m (n+1))\bigg)\epsilon_2 + \mathcal{O}(\epsilon^2), \\
\mu_3 &= \nu_{31}^{-\frac{1}{5}}\nu_{11} = -\frac{m n\epsilon_1}{2 c (2 \vartheta +3) D^6 (x^*)^2(\frac{(\vartheta +3) m n}{D x^*})^{5/2}}\bigg(((\vartheta +3) D^2 (2 c (\vartheta +3) (2 \vartheta +3)^3 m^2 n\\
&\quad+D n^3 x^* ((\vartheta +4) D-(\vartheta +3) (2 \vartheta +3)))+c (2 \vartheta +3) m ((2 \vartheta +3)^3+3 D^2 (2 (\vartheta +3)\\
&\quad+(\vartheta +4) n)-(2 \vartheta +3) D (7 \vartheta +2 (\vartheta +3) n+15))(2 (\vartheta +3) (2 \vartheta +3) D m (2 \vartheta +D+3)\\
&\quad+(x^*)^2 ((\vartheta +4) D-(\vartheta +3) (2 \vartheta +3))))\bigg)+\frac{\epsilon_1}{(2 \vartheta +3) D^3 (n+1) x^* (\frac{(\vartheta +3) m n}{D x^*})^{3/2}}\\
&\quad\times\bigg(n ((x^*)^2 ((2 \vartheta +3)^3+3 D^2 (2 (\vartheta +3)+(\vartheta +4) n)-(2 \vartheta +3) D (7 \vartheta +2 (\vartheta +3) n\\
&\quad+15))-(2 \vartheta ^2+9 \vartheta +9) D^2 m (n+1))\bigg) -\frac{n\epsilon_3}{2 c (2 \vartheta +3) D^4 x^* (\frac{(\vartheta +3) m n}{D x^*})^{3/2}}\\
&\quad\times\bigg( (4 c (\vartheta +3) (2 \vartheta +3)^2 m^3 (6 (2 \vartheta ^2+5 \vartheta +3)-(11 \vartheta +15) D)+D^2 n^3 x^* (-2 \vartheta ^2\\
&\quad-9 \vartheta +(\vartheta +4) D-9))\bigg)+ \mathcal{O}(\epsilon^2).
\end{align*}

The Jacobian determinant evaluates to
\[
\begin{aligned}
\det\left(\frac{\partial(\mu_1,\mu_2,\mu_3)}{\partial(\epsilon_1,\epsilon_2,\epsilon_3)}\right)\bigg|_{\epsilon=0} &= \frac{\varsigma^{4/5}(2\vartheta+3)^6\left((2\vartheta+3)^2-(4\vartheta+9)D\right)}{16(\vartheta+3)^7D^4}\Big(12(\vartheta+1) \\
&\quad\times (2\vartheta+3)^4 + 3\vartheta D^4 - 3\vartheta D^3+ (2\vartheta+3)^2(32\vartheta+69)D^2 \\
&\quad - (2\vartheta+3)^2(82\vartheta^2+240\vartheta+171)D\Big) \neq 0,
\end{aligned}
\]
verifying the transversality condition \eqref{4.14}. Thus, system \eqref{4.12} constitutes a versal unfolding of the codimension-3 cusp singularity.
\end{proof}

\section{Knowledges on $\mathcal{K}$-equivalence and $\mathcal{K}$-codimension}
\label{app:components}

All of these essential definitions, propositions, and theorems can be found in \citet{Montaldi2021}.

\begin{defn}\label{def:contact_equiv} Two map germs $f,g:\mathbb{R}^n\rightarrow \mathbb{R}^p$ are {\it contact equivalent} or {\it $\mathcal{K}$-equivalent}, if there exist,\\
(i) a diffeomorphism $\phi$ of the source $(\mathbb{R}^n, 0)$, and\\
(ii) a matrix $M\in GL_p(\varepsilon_n)$ such that\\
\begin{align}
f\circ\phi(x)=M(x)g(x),
\end{align}
where $f(x)$ and $g(x)$ are written as column vectors, and $M(x)g(x)$ is the usual product of matrix times vector. In particular,  $M(x)$ is the identity, one says $f(x)$ and $g(x)$ are $\mathcal{C}$-equivalent, in this case $f(x)=M(x)g(x)$.
\end{defn}
\begin{prop}\label{prop:same_codim}
If $f,g$ are contact equivalent, then they have the same codimension:
\begin{align}
{\rm codim}(f,\mathcal{K})={\rm codim}(g,\mathcal{K})
\end{align}
\end{prop}
\begin{defn}\label{def:tangent_map} Let $f:\mathbb{R}^n\rightarrow \mathbb{R}^p$ be a smooth map and $\bm v$ be a smooth vector field on the domain of $f$. Then
 ${\rm t}f(\bm v)$ is the vector field along $f$ given by
\begin{align}
{\rm t}f(\bm v)={\rm d}f_x(\bm v(x))=\sum_j\frac{\partial f}{\partial x_j}(x) v_j(x),
\end{align}
The map ${\rm t}f$ is the tangent map of $f$.
\end{defn}
\begin{defn}\label{def:tangent_space_contact} The {\it tangent space} for contact equivalence of a smooth map germ
$f:(\mathbb{R}^n,0)\rightarrow (\mathbb{R}^p,0)$ is defined to be an $\varepsilon_n$-module, a submodule of $\theta(f)\simeq \varepsilon_n^p$
\begin{align}T\mathcal{K}\cdot f={\rm t}f(\mathfrak{m}_n\theta_n)+I_f\theta(f)
\end{align}
\end{defn}
\begin{defn}\label{def:extended_tangent_space}
 The {\it extended tangent space} for contact equivalence of a smooth map germ
$f:(\mathbb{R}^n,0)\rightarrow (\mathbb{R}^p,0)$ is defined to be
\begin{align}
T_e\mathcal{K}\cdot f={\rm t}f(\theta_n)+I_f\theta(f)
\end{align}
\end{defn}

\begin{defn}\label{def:contact_codim}
Let $f:\mathbb{R}^n\rightarrow \mathbb{R}^p$ be a smooth map germ,  the contact-codimension of the map germ $f$ is defined as
\begin{align}
{\rm codim}(f,\mathcal{K})={\rm dim}\Big(\theta(f)\big/T_e\mathcal{K}\cdot f\Big).
\end{align}
A map germ $f$ is therefore said to be of {\it finite $\mathcal{K}$-codimension} if ${\rm codim}(f,\mathcal{K})<\infty$.

\end{defn}

\begin{thm}\label{thm:k_determined}Let $f:(\mathbb{R}^n,0)\rightarrow (\mathbb{R}^p,0)$ be a smooth map germ. If
\begin{align}
\mathfrak{m}_n^{k+1}\theta(f)\subset\mathfrak{m}_n^{k}T\mathcal{K}\cdot f,
\end{align}
then $f$ is $k$-determined with respect to $\mathcal{K}$-equivalence.
\end{thm}

\section{Proof of Theorem \ref{th:5.3}}
\label{proof_th:5.3}
\begin{proof}
According to Lemma \ref{lem:weak_focus}, the origin is a weak focus of order $k$ if $B_i = 0$ for $i=1,2,\dots,2k$ and $B_{2k+1} \neq 0$, which corresponds to codimension $k$. The recurrence relations established in the proof:
\begin{align*}
B_2 &= -\tfrac{1}{2}\nu_2 B_1, \quad
B_4 = -\tfrac{3}{2}\nu_2 B_3, \quad
B_6 = -\tfrac{5}{2}\nu_2 B_5, \quad
B_8 = -\tfrac{7}{2}\nu_2 B_7,
\end{align*}
show that when $B_{2m-1} = 0$ for any $m \geq 1$, then $B_{2m} = 0$ (provided $\nu_2 \neq 0$). Thus, the independent conditions reduce to $B_{2j-1} = 0$ for $j=1,2,\dots,k$, implying codimension $k$. Stability is determined by the sign of the first nonvanishing coefficient $B_{2k+1}$ as specified in Lemma \ref{lem:weak_focus}.

The Taylor series expansion of $F(\theta(x)) - F(x)$ about $x=0$ is
\begin{align*}
F(\theta(x)) - F(x) &= -G(\theta(x) + x_2) + G(x + x_2) + \int_x^{\theta(x)} \frac{n}{p(s + x_2)}  {\rm d}s \\
&= \sum_{i=1}^7 B_i x^i + \mathcal{O}(x^8),
\end{align*}
with coefficients calculated as follows.\\
The first-order coefficient is
\begin{align*}
B_1=F'(\theta(0))-F'(0)=2\bigg( G'(x_2)-\frac{ n}{p(x_2 )}\bigg)= 2\mathscr{G}'(x_2).
\end{align*}
The origin is a weak focus (order $\geq 0$) if $B_1 = 0$, i.e., $\mathscr{G}'(x_2) = 0$. Under this condition, $B_2 = -\frac{1}{2}\nu_2 B_1 = 0$ automatically.

Under the condition $B_1 = 0$ (and hence $B_2 = 0$), we compute
\begin{align*}
B_3 &=\frac{1}{6}\big(F^{(3)}(\theta(0))-F^{(3)}(0)\big)\\
&=\frac{1}{3}\big( G^{(3)}(x_2)+3 G''(x_2)\nu_2+\frac{ n p''(x_2)}{p^2(x_2)}+\frac{3 n p'(x_2)\nu_2}{p^2(x_2)}-\frac{2 n p'^2(x_2)}{p^3(x_2)}\Big)\\
&= \mathscr{G}'''(x_2) + 3\mathscr{G}''(x_2)\nu_2.
\end{align*}
The origin is a weak focus of order $\geq 1$ if additionally $B_3 = 0$, i.e., $\mathscr{G}'''(x_2) = -3\nu_2\mathscr{G}''(x_2)$. Under this condition, $B_4 = -\frac{3}{2}\nu_2 B_3 = 0$ automatically.
\\
Under the conditions $B_1 = B_3 = 0$ (and hence $B_2 = B_4 = 0$), we compute
\begin{align*}
B_5 &=\frac{1}{120}\big(F^{(5)}(\theta(0))-F^{(5)}(0)\big)\\
&=\dfrac{1}{60}\bigg(G^{(5)}(x_2)+10 G^{(4)}(x_2)\nu_2+\frac{(3 h_2 h_5-10 h_3 h_4) G''(x_2)\nu_2}{h_2 h_3}-\frac{24 n p'^4(x_2)}{p^5(x_2)}\\
&\quad+\frac{(3 h_2 h_5-10 h_3 h_4) n p'(x_2)\nu_2}{h_2 h_3 p^2(x_2)}+\frac{ n p^{(4)}(x_2)}{p^2(x_2)}-\frac{6 n p''^2(x_2)}{p^3(x_2)}+\frac{60 n p'^3(x_2)\nu_2}{p^4(x_2)}\\
&\quad+p^{(3)}(x_2)\Big(\frac{10 n \nu_2}{p^2(x_2)}-\frac{8 n p'(x_2)}{p^3(x_2)}\Big)+p''(x_2)\Big(\frac{36 n p'^2(x_2)}{p^4(x_2)}-\frac{60 n p'(x_2)\nu_2}{p^3(x_2)}\Big)\bigg)\\
&= \frac{1}{60}\left( \mathscr{G}^{(5)}(x_2) + 10 \mathscr{G}^{(4)}(x_2)\nu_2 - \frac{h_5 + 10 h_4 \nu_2}{h_2} \mathscr{G}''(x_2) \right).
\end{align*}
The origin is a weak focus of order $\geq 2$ if additionally $B_5 = 0$, i.e.,
\[
\mathscr{G}''(x_2)(h_5 + 10 h_4 \nu_2) = h_2(\mathscr{G}^{(5)}(x_2) + 10 \mathscr{G}^{(4)}(x_2)\nu_2).
\]
 Under this condition, $B_6 = -\frac{5}{2}\nu_2 B_5 = 0$ automatically.
\par
Under the conditions $B_1 = B_3 = B_5 = 0$ (and hence $B_2 = B_4 = B_6 = 0$), we compute $B_7$ from the expression below. The origin is a weak focus of order $\geq 3$ if $B_7 = 0$, and of order exactly $3$ if $B_7 \neq 0$.

The explicit expression for $B_7$ is:
\begin{align*}
B_7 &= \frac{1}{5040}\bigg(\frac{(2 (7 h_4 h_5-h_2 h_7)) G''(x_2)}{h_2^2}-\frac{84 h_5 n p'^3(x_2)}{h_2 p^4(x_2)}+\frac{14 h_4 h_5 n p'(x_2)}{h_2^2 p^2(x_2 )}-\frac{2 h_7 n p'(x_2)}{h_2 p^2(x_2)}\\
&\quad -\frac{1440 n p'^6(x_2)}{p^7(x_2)}+\frac{84 h_5 n p'(x_2) p''(x_2)}{h_2 p^3(x_2)}+\frac{3600 n p'^4(x_2) p''(x_2 )}{p^6(x_2)}-\frac{2160 n p'(x_2 )^2 p''^2(x_2)}{p^5(x_2)}\\
&\quad -\frac{14 h_5 n p^{(3)}(x_2)}{h_2 p^2(x_2 )}-\frac{40 n p^{(3)}(x_2)^2}{p^3(x_2)}-\frac{960 n p^{(3)}(x_2) p'^3(x_2)}{p(x_2)^5}+\frac{720 n p^{(3)}(x_2) p'(x_2) p''(x_2)}{p^4(x_2)}\\
&\quad +\frac{180 n p''^3(x_2)}{p(x_2)^4}-\frac{14 h_5 G^{(4)}(x_2)}{h_2}-\frac{24 n p^{(5)}(x_2) p'(x_2)}{p^3(x_2)}-\frac{60 n p^{(4)}(x_2) p''(x_2)}{p^3(x_2)}\\
&\quad +\frac{180 n p^{(4)}(x_2) p'^2(x_2)}{p^4(x_2)}+\nu_2^3 \Big(-1260 G^{(4)}(x_2)+\frac{1260 h_4 G''(x_2)}{h_2}+\frac{1260 h_4 n p'(x_2)}{h_2 p^2(x_2)}\\
&\quad -\frac{1260 n p^{(3)}(x_2)}{p^2(x_2)}-\frac{7560 n p'^3(x_2)}{p^4(x_2)}+\frac{7560 n p'(x_2) p''(x_2)}{p^3(x_2)}\Big)+\nu_2\Big(\frac{140 h_4^2 n p'(x_2 )}{h_2^2 p^2(x_2)}\\
&\quad -\frac{42 h_6 n p'(x_2)}{h_2 p^2(x_2)}
+\frac{140 h_4^2 G''(x_2)}{h_2^2}-\frac{42 h_6 G''(x_2)}{h_2}-\frac{840 h_4 n p'^3(x_2)}{h_2 p^4(x_2)}+\frac{5040 n p'^5(x_2)}{p^6(x_2)}\\
&\quad +42 G^{(6)}(x_2)-\frac{140 h_4 G^{(4)}(x_2)}{h_2}-\frac{140 h_4 n p^{(3)}(x_2)}{h_2 p^2(x_2)}+\frac{840 h_4 n p'(x_2) p''(x_2)}{h_2 p^3(x_2)}+\frac{42 n p^{(5)}(x_2)}{p^2(x_2)}\\
&\quad -\frac{10080 n p'^3(x_2) p''(x_2)}{p^5(x_2)}+\frac{1260 n p'(x_2) \big(3 p''^2(x_2)+2 p^{(3)}(x_2) p'(x_2)\big)}{p^4(x_2)}\\
&\quad -\frac{420 n \big(p^{(4)}(x_2) p'(x_2)+2 p^{(3)}(x_2) p''(x_2)\big)}{p^3(x_2)}
\Big)+2 G^{(7)}(x_2)+\frac{2 n p^{(6)}(x_2)}{p^2(x_2)}
\bigg)\\
&=\frac{1}{2520}\Big(\mathscr{G}^{(7)}(x_2)+21 \nu_2 \mathscr{G}^{(6)}(x_2)-\frac{ 7\left(90 h_2 \nu_2^3+10 h_4 \nu_2+h_5\right)}{h_2} \mathscr{G}^{(4)}(x_2)\\
&\quad+\frac{630 h_2 h_4 \nu_2^3+\left(70 h_4^2-21 h_2 h_6\right) \nu_2+7 h_4 h_5-h_2 h_7}{h_2^2}\mathscr{G}''(x_2)\Big).
\end{align*}
Under the conditions $B_1 = B_3 = B_5 = 0$, we have $B_6 = B_8 = 0$ automatically from the recurrence relations. The eighth-order coefficient is given by:
\begin{align*}
B_8 &= -\frac{7}{2}\nu_2 B_7.
\end{align*}

The derivatives of $p(x)$ and $G(x)$ used in these expressions are:
\begin{align*}
p(x_2) &= \frac{cx_2^3 }{1+ax_2+bx_2^3},\quad p'(x_2) = \frac{cx_2^2(2ax_2+3)}{(1+ax_2+bx_2^3)^2},\\
p''(x_2) &= \frac{2cx_2(a^2x_2^2+3 ax_2-3 a bx_2^4-6 bx_2^3+3)}{(1+a x_2 +bx_2^3)^3},\\
p^{(3)}(x_2) &= \frac{6 c(-2 bx_2^3(2 a^2x_2^2+7 a x_2 +8)+2 x_2^6 b^2 (2 ax_2+5)+1)}{(1+a x_2+bx_2^3)^4},
\end{align*}
\begin{align*}
p^{(4)}(x_2) &= -\frac{24 c}{\big(1+ax_2+bx_2^3\big)^5}\Big( a^3 bx_2^5+5 a^2 bx_2^4 \big(1-2 bx_2^3\big)+a \big(5 b^3x_2^9-40 b^2x_2^6\\
&\quad +10 bx_2^3+1\big)+3 bx_2^2 \big(5 b^2x_2^6-17 bx_2^3+5\big)\Big),\\
p^{(5)}(x_2) &= \frac{120 c}{(1+ax_2+bx_2^3)^6}\Big(6 a^3 b^2x_2^7+a^2(1-20 b^3x_2^9+33 b^2x_2^6)+6abx_2^2(b^3x_2^9-15 b^2x_2^6\\
&\quad +12bx_2^3+1)+3 bx_2(7 b^3x_2^9-42 b^2x_2^6+30 bx_2^3-2)\Big),
\end{align*}
and
\begin{align*}
G'(x_2)&= \frac{-c x_2^3 + m(a x_2 - b x_2^3 + 2)}{c x_2^3}, \quad & G''(x_2) &= -\frac{2m(a x_2 + 3)}{c x_2^4}, \\[6pt]
G^{(3)}(x_2)&= \frac{6m(a x_2 + 4)}{c x_2^5},\quad & G^{(4)}(x_2)&= -\frac{24m(a x_2 + 5)}{c x_2^6}, \\[6pt]
G^{(5)}(x_2)&= \frac{120m(a x_2 + 6)}{c x_2^7},\quad & G^{(6)}(x_2)&= -\frac{720m(a x_2 + 7)}{c x_2^8}, \\[6pt]
G^{(7)}(x_2)&= \frac{5040m(a x_2 + 8)}{c x_2^9}.
\end{align*}
The explicit expressions for $B_1$ to $B_8$ confirm that the conditions $\mathscr{G}'(x_2) = 0$, $\mathscr{G}'''(x_2) = -3\nu_2\mathscr{G}''(x_2)$, and subsequent higher-order conditions correspond precisely to $B_1 = B_2 = \cdots = B_{2k} = 0$ for $k = 0,1,2,3$. By Lemma \ref{lem:weak_focus}, the origin is therefore a weak focus of order $k$ with codimension $k$, and its stability is determined by the sign of the first nonvanishing odd-index coefficient $B_{2k+1}$, as summarized in the theorem statement.
\end{proof}

\end{document}